\documentclass[10pt]{amsart}
\usepackage{fullpage}
\usepackage{amsfonts}
\usepackage{mathrsfs}
\usepackage{amsmath}
\usepackage{amsthm}
\usepackage{extarrows}
\usepackage{amssymb,bbm}
\usepackage{nccmath}
\usepackage{latexsym,hyperref}
\usepackage{comment}
\usepackage{float}
\usepackage{color}
\usepackage[all]{xy}
\usepackage{accents}
\usepackage{etoolbox}
\usepackage{enumitem}
\usepackage{caption}

\hypersetup{%
    pdfborder = {0 0 0},
    colorlinks,
    citecolor=red!50!black,
    filecolor=Darkgreen,
    linkcolor=blue!40!black,
    urlcolor=cyan!50!black!90
}

\newtheorem{thrm}{Theorem}[section]
\newtheorem{prop}[thrm]{Proposition}
\newtheorem{defn}[thrm]{Definition}
\newtheorem{crl}[thrm]{Corollary}
\newtheorem{lemma}[thrm]{Lemma}

\theoremstyle{definition}
\newtheorem{rmk}[thrm]{Remark}
\newcommand{\nc}{\newcommand}

\numberwithin{equation}{section}

\nc{\tbox}[1]{\fbox{\tt #1}}

\nc{\spl}[1]{\begin{equation}\begin{aligned}#1\end{aligned}\end{equation}}
\nc{\eqa}[1]{\begin{alignat}{99}#1\end{alignat}}
\nc{\eqn}[1]{\begin{alignat*}{99}#1\end{alignat*}}
\nc{\eg}[1]{\begin{gather}#1\end{gather}}
\nc{\egn}[1]{\begin{gather*}#1\end{gather*}}

\usepackage[usenames,dvipsnames]{xcolor}
\nc{\red}{\color{red}}
\nc{\blu}{\color{blue}}
\nc{\br}{\color{Brown}}
\nc{\gre}{\color{green!50!black}}
\nc\el{\nonumber\\}
\nc\nn{\nonumber}

\nc{\mr}{\mathrm}
\nc{\mf}{\mathfrak}

\nc{\al}{\alpha}
\nc{\del}{\delta}
\nc{\eps}{\epsilon}
\nc{\ga}{\gamma}
\nc{\Ga}{\Gamma}
\nc{\ka}{\kappa}
\nc{\la}{\lambda}
\nc{\om}{\omega}
\nc{\si}{\sigma}
\nc{\Si}{\Sigma}
\nc{\bsi}{\boldsymbol\sigma}
\nc{\bSi}{\boldsymbol\Sigma}
\nc{\Ups}{\upsilon}
\nc{\vphi}{\varphi}

\nc{\btau}{\boldsymbol\tau}
\nc{\bdel}{\boldsymbol\delta}

\nc{\id}{\mathrm{id}}
\nc{\gr}{\mathrm{gr}}

\nc{\Ug}{U\mathfrak{g}}
\nc{\Ub}{U\mathfrak{b}}

\nc{\Hk}{\mathsf{H}}
\nc{\ombH}{\overline{\mathbf{H}}}

\nc{\ud}{\underline}
\nc{\tl}{\tilde}

\nc{\mbA}{\mathbf{A}}
\nc{\mbb}{\mathbf{b}}
\nc{\mbB}{\mathbf{B}}
\nc{\mbc}{\mathbf{c}}
\nc{\mbC}{\mathbf{C}}
\nc{\mbd}{\mathbf{d}}
\nc{\mbD}{\mathbf{D}}
\nc{\mbe}{\mathbf{e}}
\nc{\mbE}{\mathbf{E}}
\nc{\mbf}{\mathbf{f}}
\nc{\mbF}{\mathbf{F}}
\nc{\mbg}{\mathbf{g}}
\nc{\mbH}{\mathbf{H}}
\nc{\mbh}{\mathbf{h}}
\nc{\mbi}{\mathbf{i}}
\nc{\mbI}{\mathbf{I}}
\nc{\mbj}{\mathbf{j}}
\nc{\mbJ}{\mathbf{J}}
\nc{\mbk}{\mathbf{k}}
\nc{\mbK}{\mathbf{K}}
\nc{\mbL}{\mathbf{L}}
\nc{\mbM}{\mathbf{M}}
\nc{\mbQ}{\mathbf{Q}}
\nc{\mbq}{\mathbf{q}}
\nc{\mbr}{\mathbf{r}}
\nc{\mbT}{\mathbf{T}}
\nc{\mbu}{\mathbf{u}}
\nc{\mbU}{\mathbf{U}}
\nc{\mbv}{\mathbf{v}}
\nc{\mbV}{\mathbf{V}}
\nc{\mbw}{\mathbf{w}}
\nc{\mbW}{\mathbf{W}}
\nc{\mbX}{\mathbf{X}}
\nc{\mbY}{\mathbf{Y}}
\nc{\mbZ}{\mathbf{Z}}

\nc{\mbbA}{\mathbb{A}}
\nc{\mbbB}{\mathbb{B}}
\nc{\mbbD}{\mathbb{D}}
\nc{\mbbF}{\mathbb{F}}
\nc{\mbbV}{\mathbb{V}}
\nc{\mbbH}{\mathbb{H}}
\nc{\mbbK}{\mathbb{K}}
\nc{\mbbL}{\mathbb{L}}
\nc{\mbbP}{\mathbb{P}}
\nc{\mbbU}{\mathbb{U}}

\nc{\mcA}{\mathcal{A}}
\nc{\mcB}{\mathcal{B}}
\nc{\mcC}{\mathcal{C}}
\nc{\mcD}{\mathcal{D}}
\nc{\mcE}{\mathcal{E}}
\nc{\mcF}{\mathcal{F}}
\nc{\mcG}{\mathcal{G}}
\nc{\mcH}{\mathcal{H}}
\nc{\mcI}{\mathcal{I}}
\nc{\mcK}{\mathcal{K}}
\nc{\mcN}{\mathcal{N}}
\nc{\mcO}{\mathcal{O}}
\nc{\mcQ}{\mathcal{Q}}
\nc{\mcS}{\mathcal{S}}
\nc{\mcP}{\mathcal{P}}
\nc{\mcU}{\mathcal{U}}
\nc{\mcT}{\mathcal{T}}
\nc{\mcV}{\mathcal{V}}
\nc{\mcW}{\mathcal{W}}
\nc{\mcX}{\mathcal{X}}
\nc{\mcY}{\mathcal{Y}}
\nc{\mcZ}{\mathcal{Z}}

\nc{\mfa}{\mathfrak{a}}
\nc{\mfA}{\mathfrak{A}}
\nc{\mfb}{\mathfrak{b}}
\nc{\mfB}{\mathfrak{B}}
\nc{\mfC}{\mathfrak{C}}
\nc{\mfd}{\mathfrak{d}}
\nc{\mfD}{\mathfrak{D}}
\nc{\mfe}{\mathfrak{e}}
\nc{\mfE}{\mathfrak{E}}
\nc{\mff}{\mathfrak{f}}
\nc{\mfF}{\mathfrak{F}}
\nc{\mfg}{\mathfrak{g}}
\nc{\mfgl}{\mathfrak{g}\mathfrak{l}}
\nc{\mfh}{\mathfrak{h}}
\nc{\mfH}{\mathfrak{H}}
\nc{\mfJ}{\mathfrak{J}}
\nc{\mfk}{\mathfrak{k}}
\nc{\mfK}{\mathfrak{K}}
\nc{\mfl}{\mathfrak{l}}
\nc{\mfL}{\mathfrak{L}}
\nc{\mfM}{\mathfrak{M}}
\nc{\mfm}{\mathfrak{m}}
\nc{\mfn}{\mathfrak{n}}
\nc{\mfN}{\mathfrak{N}}
\nc{\mfo}{\mathfrak{o}}
\nc{\mfP}{\mathfrak{P}}
\nc{\mfQ}{\mathfrak{Q}}
\nc{\mfS}{\mathfrak{S}}
\nc{\mfsl}{\mathfrak{s}\mathfrak{l}}
\nc{\mfso}{\mathfrak{s}\mathfrak{o}}
\nc{\mfsp}{\mathfrak{s}\mathfrak{p}}
\nc{\mft}{\mathfrak{t}}
\nc{\mfU}{\mathfrak{U}}
\nc{\mfu}{\mathfrak{u}}
\nc{\mfUqsl}{\mathfrak{U}_q\mathfrak{sl}}
\nc{\mfUsl}{\mathfrak{Usl}}
\nc{\mfV}{\mathfrak{V}}
\nc{\mfX}{\mathfrak{X}}
\nc{\mfY}{\mathfrak{Y}}
\nc{\mfz}{\mathfrak{z}}

\nc{\mrmd}{\mathrm{d}}

\nc{\sal}{\check{\alpha}}
\nc{\cbeta}{\check{\beta}}
\nc{\cd}{\check{d}}
\nc{\cf}{\check{f}}
\nc{\cdelta}{\check{\delta}}
\nc{\ccr}{\check{r}}
\nc{\cs}{\check{s}}

\nc{\bv}{\breve{v}}

\nc{\tc}{\tilde{c}}
\nc{\tr}{\tilde{r}}
\nc{\ts}{\tilde{s}}
\nc{\tv}{\tilde{v}}

\nc{\msA}{\mathsf{A}}
\nc{\msB}{\mathsf{B}}
\nc{\msC}{\mathsf{C}}
\nc{\msc}{\mathsf{c}}
\nc{\msD}{\mathsf{D}}
\nc{\msd}{\mathsf{d}}
\nc{\mse}{\mathsf{e}}
\nc{\msw}{\mathsf{w}}
\nc{\msq}{\mathsf{q}}
\nc{\msg}{\mathsf{g}}
\nc{\msE}{\mathsf{E}}
\nc{\msf}{\mathsf{f}}
\nc{\msF}{\mathsf{F}}
\nc{\msh}{\mathsf{h}}
\nc{\msk}{\mathsf{k}}
\nc{\msH}{\mathsf{H}}
\nc{\msI}{\mathsf{I}}
\nc{\msJ}{\mathsf{J}}
\nc{\msK}{\mathsf{K}}
\nc{\msL}{\mathsf{L}}
\nc{\msP}{\mathsf{P}}
\nc{\msQ}{\mathsf{Q}}
\nc{\msR}{\mathsf{R}}
\nc{\mss}{\mathsf{s}}
\nc{\msS}{\mathsf{S}}
\nc{\msT}{\mathsf{T}}
\nc{\msU}{\mathsf{U}}
\nc{\msv}{\mathsf{v}}
\nc{\msV}{\mathsf{V}}
\nc{\msx}{\mathsf{x}}
\nc{\msX}{\mathsf{X}}
\nc{\msY}{\mathsf{Y}}
\nc{\msZ}{\mathsf{Z}}

\nc{\End}{\mathrm{End}}
\nc{\Ext}{\mathrm{Ext}}
\nc{\Hom}{\mathrm{Hom}}
\nc{\Ima}{\mathrm{Image}}
\nc{\Ind}{\mathrm{Ind}}
\nc{\Ker}{\mathrm{Ker}}
\nc{\RHom}{\mathrm{RHom}}
\nc{\Sym}{\mathrm{Sym}}

\nc{\mtc}{\mathtt{c}}
\nc{\mtD}{\mathtt{D}}
\nc{\mte}{\mathtt{e}}
\nc{\mtE}{\mathtt{E}}
\nc{\mtf}{\mathtt{f}}
\nc{\mtF}{\mathtt{F}}
\nc{\mth}{\mathtt{h}}
\nc{\mtH}{\mathtt{H}}
\nc{\mtV}{\mathtt{V}}
\nc{\mtX}{\mathtt{X}}
\nc{\mty}{\mathtt{y}}

\nc{\ddeg}{\mathtt{deg}}
\nc{\dimm}{\mathtt{dim}}
\nc{\lmod}{\mathtt{lmod}}
\nc{\opp}{\mathtt{opp}}
\nc{\rmod}{\mathtt{rmod}}

\nc{\mmod}{\mathrm{mod}}
\nc{\nbh}{\mathrm{nbh}}

\nc{\lrh}{\leftrightharpoons}
\nc{\iso}{\stackrel{\sim}{\longrightarrow}}
\nc{\liso}{\stackrel{\sim}{\longleftarrow}}
\nc{\wh}{\widehat}
\nc{\wt}{\widetilde}
\nc{\lra}{\longrightarrow}
\nc{\ra}{\rightarrow}
\nc{\into}{\hookrightarrow}
\nc{\onto}{\twoheadrightarrow}

\nc{\C}{\mathbb{C}}
\nc{\N}{\mathbb{N}}
\nc{\Z}{\mathbb{Z}}

\nc{\SW}{\mathsf{SW}}

\nc{\ot}{\otimes}
\nc{\op}{\oplus}
\nc{\ol}{\overline}
\nc{\un}{\underline}

\nc{\lan}{\langle}
\nc{\ran}{\rangle}

\nc{\Xp}{x_{j_1,\ldots,j_p}^{-1}}
\nc{\Xeta}{x_{\eta_1,\ldots,\eta_e}^{-1}}

\nc{\qu}{\quad}
\nc{\qq}{\qquad}
\nc\Tr{{\rm tr}}

\usepackage[scr=boondoxo]{mathalfa}
\nc{\key}{{\mathscr{k}}}
\nc{\ley}{{\mathscr{l}}}
\nc{\utr}{{\vartriangle}}
\nc{\dtr}{{\triangledown}}
\nc{\ua}{{\uparrow}}
\nc{\da}{{\downarrow}}
\usepackage{tikz-cd}

\renewcommand{\,}{\kern 0.1em} 

\begin{document}

\vspace{1.2cm}

\begin{center}
{\Large{\textbf{Representations of twisted Yangians of types B, C, D: II}}} 

\bigskip

Nicolas Guay$^{1a}$, Vidas Regelskis$^{2b}$, Curtis Wendlandt$^{1c}$

\end{center}

\bigskip

\begin{center} \small 
$^1$ University of Alberta, Department of Mathematics, CAB 632, Edmonton, AB T6G 2G1, Canada.\\ 
$^2$ University of York, Department of Mathematics, York, YO10 5DD,  UK. \\
\smallskip
E-mail: $^a$\,nguay@ualberta.ca $^b$\,vidas.regelskis@york.ac.uk $^c$\,cwendlan@ualberta.ca
\end{center}

\patchcmd{\abstract}{\normalsize}{}{}{}

\begin{abstract} \small 
We continue the study of finite-dimensional irreducible representations of twisted Yangians associated to symmetric pairs of types B, C and D, with focus on those of types BI, CII and DI. After establishing that, for all twisted Yangians of these types, the highest weight of such a module necessarily satisfies a certain set of relations, we classify the finite-dimensional irreducible representations of twisted Yangians for the pairs $(\mfso_N,\mfso_{N-2} \oplus \mfso_2)$ and $(\mfso_{2n+1},\mfso_{2n})$. 
\end{abstract}

\makeatletter
\@setabstract
\makeatother

\medskip

\tableofcontents

\thispagestyle{empty}

\setlength{\parskip}{1ex}


\section{Introduction}

Quantized enveloping algebras of Kac-Moody algebras, and in particular of affine Lie algebras, have been preeminent examples of quantum groups since the 1980's. In the past few years, there has been an increase of research activity focusing on certain coideal subalgebras of those quantized enveloping algebras, for instance on the quantum symmetric pairs associated to finite-dimensional simple Lie algebras (see e.g.~\cite{KP,BW,BSWW,Ko2,Le3}), which date back to work of G. Letzter in the 1990's \cite{Le1,Le2}, and on those associated to symmetrizable Kac-Moody algebras (see e.g.~\cite{BK1,BK2,FLLW1,FLLW2,FL2}), which appeared for the first time in the work of S. Kolb \cite{Ko1}. See also \cite{HK,ES,BKLW,FL1}.

Important examples of coideal subalgebras of quantized enveloping algebras appeared even before the aforementioned work of G. Letzter: around 1990, G. Olshanskii introduced in \cite{Ol} the twisted Yangians of type AI and AII (that is, those associated to the symmetric pairs $(\mfgl_N,\mfso_N)$ and $(\mfgl_N,\mfsp_N)$), which are coideal subalgebras of the Yangian of $\mfgl_N$, the later being one of the two families of quantized enveloping algebras of affine type $A^{(1)}$. These twisted Yangians and their representations have been quite well studied over the years, mainly in the work of A. Molev, M. Nazarov \textit{et al.}, see e.g. \cite{Mo1,Mo2,Mo3,Mo4,Mobook,Na,KN1,KN2,KN3,KNP}. In particular, A. Molev obtained in \cite{Mo1} and \cite{Mo2} a classification of their finite-dimensional irreducible representations. 

In \cite{GR}, two of the present authors constructed twisted Yangians for all symmetric pairs $(\mfg,\mfg^{\rho})$ where $\mfg$ is an orthogonal or a symplectic Lie algebra and $\mfg^{\rho}$ is the Lie subalgebra fixed by an involution $\rho$ of $\mfg$.  These new twisted Yangians of type B-C-D are also coideal subalgebras of the Yangian of $\mfg$ and their basic properties were established in \cite{GR}. In \cite{GRW2}, we initiated a study of their representation theory, proving some general results (e.g. that finite-dimensional irreducible representations are highest weight modules - see Theorem 4.5 in \textit{loc. cit.}) and obtaining, for twisted Yangians of type CI, DIII and BCD0, a full classification of finite-dimensional irreducible representations in terms of certain polynomials, in analogy with \cite{Dr} (for the non-twisted Yangians) and \cite{Mo1,Mo2}: see Theorems 6.2, 6.5 and 6.6 in \cite{GRW2}. 

The present article is a sequel to our work  \cite{GRW2}. We start addressing the classification problem of finite-dimensional irreducible representations of twisted Yangians for the symmetric pairs $(\mfg_N,\mfg_{N-q}\oplus \mfg_q)$ of types BI, CII and DI(a) (see Table \ref{Table:G}), and we obtain complete results for those associated to the symmetric pairs $(\mfso_N,\mfso_{N-2}\op \mfso_2)$ and $(\mfso_{2n+1},\mfso_{2n})$. In Section 2, we recall the relevant preliminary material from \cite{AMR}, \cite{GR} and \cite{GRW2}. In Section 3, we consider the twisted Yangians of orthogonal type when $N=3$ and $N=4$. These twisted Yangians are known to be isomorphic to the twisted Yangian of $(\mfgl_2,\mfso_2)$ or a tensor product of these, as was established in \cite{GRW1}: see the beginning of Subsections \ref{sec:so4so2so2} and \ref{sec:so3so2} for precise statements of these isomorphisms. Consequently, the classification of finite-dimensional irreducible representations in these two low rank cases can be translated from the corresponding result of A. Molev established in \cite{Mo1,Mo2}: see Propositions \ref{P:so4class} and \ref{P:so3class}.  The latter of these two propositions plays a role in the proofs of two of the main results of this paper, namely Proposition \ref{P:necessary} and Theorem \ref{T:BI(b)-Class}.

In Section \ref{sec:Nec}, some of the important results established in \cite{GRW2} are strengthened for the symmetric pairs of types BI, CII and DI(a). Proposition \ref{P:necessary} is the main result in this section and gives necessary conditions for an irreducible highest weight module of a twisted Yangians to be finite-dimensional. This provides essentially one half of the proofs of Theorems \ref{T:DI(a)-Class} and \ref{T:BI(b)-Class}. Its proof relies on a similar result in \cite{GRW2}, namely Proposition 4.18, and on Theorems 6.5 and 6.6 also from \cite{GRW2}, which are the classification theorems of finite-dimensional irreducible modules for twisted Yangians of type BCD0.  Proposition \ref{P:int} provides additional restrictions on the complex number $\al$ which appears in the statement of Proposition \ref{P:necessary}: its proof boils down to computing the highest weight of a certain highest weight module over $\mfg_N^{\rho}$, which is done in Lemma \ref{L:gtw}. In particular, this proposition shows that, when $\mfg_q\ncong \mfso_2$, the parameter $\alpha$ can only take certain rational values and must satisfy a specific inequality.

In order to prove the classification theorem for the twisted Yangians of the pair $(\mfso_N,\mfso_{N-2} \oplus \mfso_2)$, it is necessary to first construct a family of one-dimensional representations parametrized by $\C$: this is achieved in Lemma \ref{L:K-1dim}. It turns out that, up to a twist by an automorphism, the one-dimensional representations provided by this lemma exhaust all of them - see Proposition \ref{P:1dimso2}. Twisted Yangians can be defined in terms of the reflection equation \eqref{TX-RE} and the symmetry relation \eqref{TX-symm}, as recalled in Subsection \ref{subsec:twYa}. It follows from this that each twisted Yangian admits a one-dimensional representation obtained by sending its matrix of generators to a certain matrix $\mcG(u)$ which is known explicitly and is part of the definition of the twisted Yangian (Definition \ref{D:TX}): see the paragraph after \eqref{grho->X}. We call it the trivial representation. This raises the question of the existence of other one-dimensional representations for twisted Yangians in general. Twisted Yangians of type CI and DIII also admit a family of one-dimensional representations parametrized by $\C$. It turns out that for twisted Yangians of type BCD0, BI, CII  and DI(a) with $\mfg_{q} \ncong \mfso_2$, the one-dimensional representations are all twists by automorphisms of the trivial representation: this is the content of Proposition \ref{P:no1dim}.

The last section provides proofs of the classification theorems of finite-dimensional irreducible representations for the twisted Yangians associated to the symmetric pairs $(\mfso_N,\mfso_{N-2} \oplus \mfso_2)$ and $(\mfso_{2n+1},\mfso_{2n})$.  These are Theorems \ref{T:DI(a)-Class} and \ref{T:BI(b)-Class} respectively. Proofs of analogous results for other twisted Yangians of type BI, CII or DI(a) are expected to be substantially more complicated and will be presented elsewhere. (Please see the end of this introduction for a brief explanation of the extra difficulties in these more general cases.) The classification is in terms of a scalar $\al$ and certain polynomials which impose conditions on the highest weight of a finite-dimensional irreducible representation. In the $(\mfso_N,\mfso_{N-2}\oplus \mfso_2)$-case, the proof of the necessity of those conditions is actually given earlier by Proposition \ref{P:necessary}. The $(\mfso_{2n+1},\mfso_{2n})$-case is substantially more complicated. The main differences are due to the role played by the action of an involution of the twisted Yangian on the highest weight of a finite-dimensional irreducible representation (Lemma \ref{BI:P-psi}). This leads to a new condition on the roots of a certain polynomial, which is the content of Proposition \ref{P:q=1-nec}. This additional condition is a new feature which is not present in the classification theorems for twisted Yangians of type A \cite{Mobook,MR} or types CI, DIII or BCD0 \cite{GRW2}.

Finally, we explain how finite-dimensional irreducible representations can be realized as subquotients of tensor products of fundamental representations along with a one-dimensional representation in the case of $(\mfso_N,\mfso_{N-2} \oplus \mfso_2)$: see Corollaries \ref{C:DI(a)fun} and \ref{C:BI(b)fun}. This is consistent with previously known results for twisted Yangians of type A, CI, DIII and BCD0 (see \cite{Mobook} and \cite{GRW2}).

 To complete the task of classifying the finite-dimensional irreducible representations of each twisted Yangian studied in \cite{GRW2}, similar results to those obtained in Section \ref{sec:main}  must be proven for the twisted Yangians which are, in the notation of Table \ref{Table:G} below, of type CII as well as types BI and DI(a) with $q\geq 3$. Like the twisted Yangians for the symmetric pairs $(\mfso_{2n+1},\mfso_{2n})$ studied in Subsection \ref{subsec:q=1}, each of these twisted Yangians has the common property that the necessary conditions established in Section \ref{sec:Nec} are not sufficient for determining precisely when the irreducible quotient of a Verma module is finite-dimensional. It is, however, possible to prove results similar to Proposition \ref{P:q=1-nec} for each of these quantum algebras, which strengthen significantly the conditions of Section \ref{sec:Nec}. In particular, for the twisted Yangians of the symmetric pairs $(\mfsp_{N},\mfsp_{N-q}\oplus \mfsp_q)$ of type CII, this leads to a complete classification of finite-dimensional irreducible modules. These results, which will be presented in \cite{GRW3}, are notably more complicated to prove than their counterparts in Section \ref{sec:main}. For instance, their proofs seem to require understanding how to pass representation theoretic information between certain isomorphic presentations of twisted Yangians. 
 
 For the twisted Yangians of the symmetric pairs $(\mfso_N,\mfso_{N-q}\oplus \mfso_q)$ which are of type BI and DI(a) with $q\geq 3$, there are additional difficulties which arise. One such difficulty involves showing that, in the notation of Definition \ref{D:assoc}, an irreducible highest weight module which is associated to the scalar $\alpha=N/4-1/2$ and polynomials $P_1(u)=\cdots=P_n(u)=1$ is finite-dimensional. Lemma \ref{L:gtw} illustrates that these modules are closely related to the spinor representations of $\mfso_{2\ley}$, where
 $\ley=q/2$ if $q$ is even and $\ley=(N-q)/2$ otherwise. This difficulty, which will be considered in future work, does not arise for the twisted Yangians corresponding to the 
 symmetric pairs $(\mfso_{2n+1},\mfso_{2n})$ which are studied in the present paper. Indeed, for these twisted Yangians the modules under consideration can be obtained by restricting the (finite-dimensional) spinor representations of the Yangian for $\mfso_{2n+1}$ which are provided by Lemma 5.18 of \cite{AMR}.

{\it Acknowledgements.} The first and third named authors gratefully acknowledge the financial support of the Natural Sciences and Engineering Research Council of Canada provided via the Discovery Grant Program and the Alexander Graham Bell Canada Graduate Scholarships - Doctoral Program, respectively. Part of this work was done during the second named author's visits to the University of Alberta; he thanks the University of Alberta for the hospitality. The second named author was supported in part by the Engineering and Physical Sciences Research Council of the United Kingdom, grant number EP/K031805/1; he gratefully acknowledges the financial support.

 \section{Preliminaries}
 
Throughout this manuscript we will employ mostly the same notation as in \cite{GRW2}. Thus in many instances we will try to be concise and will refer to {\it loc.~cit.} for complete details. We start by recalling some basic definitions.

Let $N=2n$ or $N=2n+1$ with $n\in\N$. We will denote by $\mfg_N$ either the orthogonal Lie algebra $\mfso_N$ or the symplectic Lie algebra $\mfsp_N$ (only when $N=2n$). The Lie algebra $\mfg_N$ can be realized as a Lie subalgebra of $\mfgl_N$ as follows. We label the rows and columns of matrices in $\mfgl_N$ by the indices $\mcI_N=\{-n,\ldots,-1,(0),1,\ldots,n\}$, where $(0)$ is omitted if $N=2n$.   Set $\theta_{ij}=1$ in the orthogonal case and $\theta_{ij}=\mathrm{sign}(i)\cdot \mathrm{sign}(j)$ in the symplectic case for $i,j\in \mcI_N$. We also introduce a smaller set $\mcI_N^{+} =\mcI_N\cap \Z_{\geq 0}$.  

For each $i,j\in \mcI_N$, let $E_{ij}$ denote the usual elementary matrix of $\mfgl_N$. Define the transposition $t$ by $(E_{ij})^{t} = \theta_{ij} E_{-j,-i}$ and set $F_{ij} = E_{ij} - (E_{ij})^{t}$ so that 
\[
[F_{ij},F_{kl}] = \delta_{jk} F_{il} - \delta_{il} F_{kj} + \delta_{j,-l} \theta_{ij} F_{k,-i} - \delta_{i,-k} \theta_{ij} F_{-j,l} \qu\text{and}\qu F_{ij} + \theta_{ij}F_{-j,-i}=0 .
\] 
Then $\mfg_N$ is isomorphic to $\mr{span}_{\C} \{ F_{ij}  \, : i,j\in \mcI_N \}$ and $\mathrm{span}_{\C} \{ F_{ii} \, : \, 1 \le i \le n \}$ forms a Cartan subalgebra which will be denoted by $\mfh_N$. 
Given a Lie algebra $\mathfrak{a}$ its universal enveloping algebra will be denoted by $\mfU\mathfrak{a}$.

Introduce the permutation operator $P = \sum_{i,j\in \mcI_N} E_{ij} \ot E_{ji}$ and the one-dimensional projector $Q=P^{t_1}=P^{t_2}$. Let $I$ denote the identity matrix. Then $P^2=I$, $PQ=QP=\pm Q$ and $Q^2=N Q$, which will be useful below. Here (and further in this paper) the upper sign corresponds to the orthogonal case and the lower sign to the symplectic case.

Let tensor products be defined over the field of complex numbers. For a matrix $X$ with entries $x_{ij}$ in an associative algebra $A$ we write
\[
X_s = \sum_{i,j\in\mcI_N} \underbrace{ I \ot \cdots \ot I}_{s-1} \ot E_{ij} \ot I \ot \cdots \ot I \ot x_{ij} \in \End(\C^N)^{\ot k} \ot A .
\]
Here $k \in \N_{\ge 2}$ and $1\le s\le k$; it will always be clear from the context what $k$ is.


\subsection{Yangians of type B-C-D and their finite-dimensional irreducible representations} \label{subsec:Ya}

We introduce elements $t_{ij}^{(r)}$ with $i,j\in \mcI_N$ and $r\in\Z_{\ge 0}$ such that $t^{(0)}_{ij}= \del_{ij}$. Combining these into formal power series $t_{ij}(u) = \sum_{r\ge 0} t_{ij}^{(r)} u^{-r}$, we can then form the generating matrix $T(u)= \sum_{i,j\in \mcI_N} E_{ij} \ot t_{ij}(u)$. 
The $R$-matrix that we need is $R(u)=I - u^{-1} P + (u-\ka)^{-1} Q$, where $\ka = N/2 \mp 1$.

\begin{defn} [\cite{AACFR}]\label{D:X}
The extended Yangian $X(\mfg_N)$ is the unital associative $\C$-algebra generated by elements $t_{ij}^{(r)}$ with $i,j\in \mcI_N$ and $r\in\Z_{\ge 1}$ satisfying the relation
\eq{ \label{RTT}
R(u-v)\,T_1(u)\,T_2(v) = T_2(v)\,T_1(u)\,R(u-v) .
}
The Hopf algebra structure of $X(\mfg_N)$ is given by 
\[
\Delta: T(u) \mapsto T(u)\ot T(u), \qq S: T(u)\mapsto T(u)^{-1},\qquad \epsilon: T(u)\mapsto I. 
\]
\end{defn}

The expansion of the defining relation \eqref{RTT} in terms of the generating series $t_{ij}(u)$   can be found in {\it e.g.}~Definition 3.1 of \cite{GRW2}. 

For each series $f(u)\in 1+u^{-1}\C[[u^{-1}]]$ there is an automorphism $\mu_f$ of $X(\mfg_N)$ which is given by the assignment $\mu_f : T(u) \mapsto f(u)\, T(u)$. The Yangian $Y(\mfg_N)$ is defined as the subalgebra of $X(\mfg_N)$ consisting of the elements stable under all the automorphisms of the form $\mu_f$, namely
\[
Y(\mfg_N) = \{ y \in X(\mfg_N) : \mu_f(y) = y \text{ for any } f(u) \in1+u^{-1}\C[[u^{-1}]] \}.
\]
By Theorem 3.1 of \cite{AACFR} and the proof of Theorem 3.1 from \cite{AMR}, there is a  central series $y(u)=1+\sum_{r\geq 1} y_r u^{-r}$ such that 
\begin{equation*}
 T^t(u+\ka)T(u)=T(u)T^t(u+\ka)=y(u)y(u+\ka)\cdot I.
\end{equation*}
It was proven in Corollary 3.2 of \cite{AMR} that the coefficients $\tau_{ij}^{(r)}$ of the all of the series $\tau_{ij}(u)=y(u)^{-1}t_{ij}(u)$ generate the subalgebra $Y(\mfg_N)$. The corresponding generating matrix $y(u)^{-1}T(u)$ will be denoted by $\mcT(u)$.

Let us now recall elements of the representation theory of $X(\mfg_N)$. A representation $V$ of $X(\mfg_N)$ is a \textit{highest weight representation} if there exists a nonzero vector $\xi\in V$ such that $V=X(\mfg_N)\,\xi$ and the following 
conditions are satisfied:
\begin{alignat*}{4}
  & t_{ij}(u)\,\xi=0 \quad &&\text{ for all } && i<j\in\mcI_N, \quad &&\text{ and } \nonumber\\
  & t_{ii}(u)\,\xi=\lambda_i(u)\,\xi \quad  &&\text{ for all } && i\in \mcI_N, && 
\end{alignat*}
where $\lambda_i(u)=1+\sum_{r\ge1}\lambda_{i}^{(r)}u^{-r}$ are formal power series in $u^{-1}$ with coefficients $\lambda_i^{(r)}\in \C$. The vector $\xi$ is called the \textit{highest weight vector} of $V$, and the $N$-tuple $\lambda(u)=(\lambda_{-n}(u),\ldots,\lambda_n(u))$ is called the \textit{highest weight} of $V$. By Theorem 5.1 of \cite{AMR}, every finite-dimensional irreducible representation $V$ of the algebra $X(\mfg_N)$ is a highest weight representation. Moreover, $V$ contains a unique highest weight vector, up to a constant factor.

Given an $N$-tuple $\lambda(u)$, the Verma module $M(\lambda(u))$ is defined as the quotient of $X(\mfg_N)$ by the left ideal generated by all the coefficients of the series 
$t_{ij}(u)$ with $i<j\in\mcI_N$ and $t_{ii}(u)-\lambda_i(u)$ with $i\in\mcI_N$.

\begin{prop} [{\cite[Proposition 5.14]{AMR}}] \label{P:X-nontriv} 
The Verma module $M(\lambda(u))$ is non-trivial if and only if the components of the highest
weight satisfy
\begin{equation}
 \frac{\lambda_{-i}(u)}{\lambda_{-i-1}(u)}=\frac{\lambda_{i+1}(u-\ka+n-i)}{\lambda_{i}(u-\ka+n-i)} \qu \text{for}\qu i\in \mcI_N^+\setminus\{ n\} . \label{HWT:Ext.non-trivial}
\end{equation}
\end{prop}

If $M(\lambda(u))$ is non-trivial, then it has a unique irreducible (non-zero) quotient $L(\lambda(u))$, and any irreducible highest weight $X(\mfg_N)$-module with the highest weight $\lambda(u)$ is isomorphic to $L(\lambda(u))$.

\begin{thrm} [{\cite[Theorem 5.16]{AMR}}] \label{T:X-class} 
Let $\lambda(u)$ satisfy \eqref{HWT:Ext.non-trivial} so that the Verma module $M(\lambda(u))$ is non-trivial. Then the irreducible $X(\mfg_N)$-module $L(\lambda(u))$ is finite-dimensional if and only if there exist monic polynomials $P_1(u),\ldots, P_n(u)$ in $u$ such that
\[
 \frac{\lambda_{i-1}(u)}{\lambda_i(u)}=\frac{P_i(u+1)}{P_i(u)} \quad  \textit{ for all }\quad 2\leq i\leq n, 
\]
and in addition
\begin{alignat*}{2}
&\frac{\lambda_0(u)}{\lambda_1(u)}=\frac{P_1(u+1/2)}{P_1(u)} \quad &&\textit{if }\quad \mfg_N=\mfso_{2n+1}, \\
&\frac{\lambda_{-1}(u)}{\lambda_1(u)}=\frac{P_1(u+2)}{P_1(u)} \quad &&\textit{if }\quad \mfg_N=\mfsp_{2n}, \\ 
&\frac{\lambda_{-1}(u)}{\lambda_2(u)}=\frac{P_1(u+1)}{P_1(u)} \quad &&\textit{if }\quad \mfg_N=\mfso_{2n} .
\end{alignat*} 
\end{thrm}

The polynomials $P_1(u),\ldots,P_n(u)$ are called the Drinfeld polynomials associated to $L(\lambda(u))$: they are uniquely determined by the highest weight $\lambda(u)$. Moreover, two finite-dimensional irreducible modules $L(\lambda(u))$ and $L(\lambda^\sharp(u))$ share the same $n$-tuple of Drinfeld polynomials if and only if there is $f(u)\in 1+u^{-1}\C[[u^{-1}]]$ such that $L(\lambda^\sharp(u))=L(\lambda(u))^{\mu_f}$; here $L(\lambda(u))^{\mu_f}$ denotes the $X(\mfg_N)$-module obtained by twisting $L(\lambda(u))$ with the automorphism $\mu_f$. 

Let us now focus on the finite-dimensional irreducible representations of the Yangian $Y(\mfg_N)$. By Corollary~5.19 of \cite{AMR}, any such representation is isomorphic to the restriction of an $X(\mfg_N)$-module $L(\la(u))$ to the subalgebra $Y(\mfg_N)$, where the components of $\la(u)$ satisfy the conditions of Theorem \ref{T:X-class}. In particular, finite-dimensional irreducible representations of $Y(\mfg_N)$ are parametrized by $n$-tuples $\mathbf{P}=(P_1(u),\ldots,P_n(u))$ of monic polynomials in $u$.

Assuming $\alpha\in \C$, let $L(i:\alpha)$ denote the irreducible highest weight representation of $Y(\mfg_N)$ associated to the tuple $\bf P$ such that $P_j(u)=1$ if $j\neq i$, and $P_i(u)=u-\alpha$. These are the so-called \textit{fundamental representations} of~$Y(\mfg_N)$. 
(These were studied in detail in Subsection 5.4 of \cite{AMR}, see also Subsection 12.1.D of \cite{CP}.)
The representations $L(i:\alpha)$ play an important role: any finite-dimensional irreducible representation of $Y(\mfg_N)$ is isomorphic to a subquotient of a tensor product of fundamental representations. (This is formulated precisely in Corollary 12.1.13 of \cite{CP}.)

Let us now recall some aspects of the representation theory of $\mfg_N$. Following Section~4.2 of \cite{Mobook}, for any $n$-tuple $\lambda=(\lambda_1,\ldots,\lambda_n)\in \C^n$ we denote by $V(\lambda)$ the irreducible $\mfg_N$-module with the highest weight $\lambda$. That is, $V(\lambda)$ is the irreducible module generated by a nonzero vector $\xi$ such that $F_{ij}\xi=0$ for all $i<j$ and $F_{kk}\xi=\lambda_k \xi$ for all $1\leq k \leq n$. 

The module $V(\lambda)$ is finite-dimensional if and only if $\lambda_{i-1}-\lambda_i\in \Z_{\geq 0}$ for all $2\leq i\leq n$ and 
\begin{equation}\label{g-class}
\begin{alignedat}{90}
 -\lambda_1 &\in \Z_{\geq 0} &\qu& \text{if}\qu \mfg_N=\mfsp_N,\\
 -2\lambda_1 &\in \Z_{\geq 0} && \text{if}\qu \mfg_N=\mfso_{2n+1},\\
 -\lambda_1-\lambda_2 &\in \Z_{\geq 0} && \text{if}\qu \mfg_N=\mfso_{2n}.
 \end{alignedat}
\end{equation}

By Proposition 3.11 of \cite{AMR} the assignment $F_{ij}\mapsto \tfrac{1}{2}(t_{ij}^{(1)}-\theta_{ij}t_{-j,-i}^{(1)})$, for all $i,j\in \mcI_N$, extends to an embedding 
$\mfU\mfg_N\into X(\mfg_N)$, and consequently we may regard any $X(\mfg_N)$-module as a $\mfg_N$-module.   

We conclude this subsection with a simple corollary of Theorem \ref{T:X-class}. 

\begin{crl}\label{C:g-action}
Suppose that $L(\la(u))$ is finite-dimensional with highest weight vector $\xi$ and Drinfeld tuple $\mathbf{P}=(P_1(u),\ldots,P_n(u))$. Then the $\mfg_N$-module $\mfU\mfg_N\xi$ is a highest weight module with highest weight $\la=(\la_i)_{i=1}^n$ whose components are given by
\eqa{
 \lambda_i &=-A(\mathbf{P},u)-\sum_{2\le a\le i} \deg\,P_a(u)\; \text{ for all } \; 1\leq i\leq n,
 \label{g-action}
\intertext{where} 
A(\mathbf{P},u) &= \begin{cases}
                \deg\, P_1(u) &\text{ if }\; \mfg_N=\mfsp_{2n},\\
                \tfrac{1}{2}\deg\, P_1(u) &\text{ if }\; \mfg_N=\mfso_{2n+1},\\
                \tfrac{1}{2}(\deg\, P_1(u)-\deg\, P_2(u)) &\text{ if }\; \mfg_N=\mfso_{2n}.
               \end{cases}
\label{Ag_N}
}
\end{crl}
%


\subsection{Twisted Yangians and their finite-dimensional irreducible representations}

Twisted Yangians of types B, C and D were introduced in \cite{GR} and the study of their representation theory was initiated in \cite{GRW2}. Let us briefly recall their main algebraic properties in the cases relevant for this paper. For complete details and proofs of the results below, please consult \cite{GR} and \cite{GRW2}.

\subsubsection{Twisted Yangians of type BDI and CII} \label{subsec:twYa}

We will focus on twisted Yangians for symmetric pairs $(\mfg_N,\mfg_p \op \mfg_q)$ with $p+q=N$ and $p\geq q>0$, that, in terms of the notation introduced in Section 2.2 of \cite{GRW2}, are of types BI, CII and DI(a). 
These symmetric pairs are of the form $(\mfg_N,\mfg_N^\rho)$ where $\rho$ is the involution of $\mfg_N$ given by $\rho(X) = \mcG X \mcG^{-1}$ for all $X\in\mfg_N$ with $\mcG$ the corresponding matrix in the table below, and $\mfg_N^\rho$ is the subalgebra of $\mfg_N$ fixed by~$\rho$.

\begin{table}[H]
\centering
\caption{Symmetric pairs} \label{Table:G} 

\begin{tabular}{|c|c|c|c|}
 \hline 
 Type & \multicolumn{2}{|c|}{Symmetric pair $(\mfg_N,\mfg_N^\rho)$} & $\mcG$ \\ 
 \hline
 BI(a) & $(\mfso_{2n+1},\mfso_{2n+1-q}\oplus \mfso_q)$ & $q=0\mod 2$  &  \rule{0pt}{4ex}   $\sum_{i=-\frac{p-1}{2}}^{\frac{p-1}{2}} E_{ii}-\sum_{i=\frac{p+1}{2}}^{n} (E_{ii} + E_{-i,-i})$ \\[1em]
 BI(b) & $(\mfso_{2n+1},\mfso_{2n+1-q}\oplus \mfso_q)$ & $q=1\mod 2$ & $-\sum_{i=-\frac{q-1}{2}}^{\frac{q-1}{2}} E_{ii}+ \sum_{i=\frac{q+1}{2}}^{n} (E_{ii} + E_{-i,-i})$ \\[1em]
 CII & $(\mfsp_{2n},\mfsp_{2n-q}\oplus \mfsp_q)$  & $q=0\mod 2$ & $\sum_{i=1}^{\frac{p}{2}}(E_{ii}+E_{-i,-i})-\sum_{i=\frac{p}{2}+1}^{n}(E_{ii}+E_{-i,-i})$\\[1em]
 DI(a) & $(\mfso_{2n},\mfso_{2n-q}\oplus \mfso_q)$  & $q=0\mod 2$ & $\sum_{i=1}^{\frac{p}{2}}(E_{ii}+E_{-i,-i})-\sum_{i=\frac{p}{2}+1}^{n}(E_{ii}+E_{-i,-i})$\\[1em]
 \hline
\end{tabular}
\end{table}

\begin{rmk}
Note that we have not included pairs of type DI(b), namely $(\mfso_{2n},\mfso_{p}\op\mfso_q)$ with both $p$ and $q$ odd. In this case the matrix $\mcG$ cannot be chosen to be diagonal. Also note that, in type BI(b), when $q=1$, $\mfso_1 = \{ 0 \}$.
\end{rmk}

We will write $g_{ij}$ for the $(i,j)$-th entry of $\mcG$. We have $\mcG=\sum_{i\in\mcI_N} g_{ii}E_{ii}$ since $g_{ij}=0$ if $i\ne j$. Let $\mcG(u)=(g_{ij}(u))_{i,j\in \mcI_N}$ be given by 
\eq{
\mcG(u) = \frac{d I - u\,\mcG}{d - u} \qu\text{with}\qu d=\frac{p-q}4. \label{G(u)}
}
Notice that $\mcG(u)=\mcG$ if $p=q$.

\begin{defn} [{\cite[Definition~3.1]{GR}}] \label{D:TX}
The extended twisted Yangian $X(\mfg_N,\mcG)^{tw}$ is the subalgebra of $X(\mfg_N)$ generated by the coefficients $s_{ij}^{(r)}$, with $i,j \in\mcI_N$ and $r\in\Z_{\ge 1}$, of the entries $s_{ij}(u) = g_{ij} +\sum_{r=1}^{\infty} s_{ij}^{(r)} u^{-r}$ of the $S$-matrix
\[
S(u) =\sum_{i,j\in\mcI_N} E_{ij} \ot s_{ij}(u) = T(u-\ka/2)\,\mcG(u)\,T^t(-u+\ka/2). 
\]
\end{defn}

The algebra $X(\mfg_N,\mcG)^{tw}$ is a left coideal subalgebra of $X(\mfg_N)$: $\Delta(X(\mfg_N,\mcG)^{tw}) \subset X(\mfg_N) \otimes X(\mfg_N,\mcG)^{tw}$. The $S$-matrix $S(u)$ satisfies the reflection equation 
\eq{ \label{TX-RE}
R(u-v)\,S_1(u)\,R(u+v)\,S_2(v) = S_2(v)\,R(u+v)\,S_1(u)\,R(u-v)
}
and the symmetry relation
\eq{ \label{TX-symm}
S^t(u) = \,S(\ka-u) \pm \frac{S(u)-S(\ka-u)}{2u-\ka} + \frac{\Tr(\mcG(u))\,S(k-u) - \Tr(S(u))\cdot I }{2u-2\ka} . 
}
The above two relations are in fact the defining relations of $X(\mfg_N,\mcG)^{tw}$. (Their form in terms of $s_{ij}(u)$ can be found in (4.4) and (4.5) of \cite{GR}.) By Theorem 4.2 of \cite{GR}, the algebra generated by the abstract elements $s_{ij}^{(r)}$ subject to those relations, which is called $\mcB(\mcG)$ in {\it loc.~cit.}, is isomorphic to $X(\mfg_N,\mcG)^{tw}$. 

The algebra defined in the same way as $\mcB(\mcG)$ except with 
the symmetry relation \eqref{TX-symm} omitted and $(s_{ij}^{(r)},s_{ij}(u),S(u))$ replaced with generators denoted $(\wt s_{ij}^{(r)},\wt s_{ij}(u),\wt S(u))$ is called the \textit{extended reflection algebra} of $(\mfg_N,\mfg_N^\rho)$ and is denoted $\wt X(\mfg_N,\mcG)^{tw}$. In Section 5.1 of \cite{GR}, where $\wt X(\mfg_N,\mcG)^{tw}$ was denoted $\mathcal{XB}(\mcG)$, it was shown that there exists a central series $c(u)=1+\sum_{r\geq 1}c_r u^{-r}\in \wt X(\mfg_N,\mcG)^{tw}[[u^{-1}]]$ such that 
\begin{equation}
 X(\mfg_N,\mcG)^{tw}\cong \wt X(\mfg_N,\mcG)^{tw}/(c(u)-1) \label{c(u)}.
\end{equation}
We refer the reader to {\it loc.~cit.}~for a more complete treatment of $\wt X(\mfg_N,\mcG)^{tw}$ and return our attention to the twisted Yangian $X(\mfg_N,\mcG)^{tw}$.

By Proposition 3.1 of \cite{GR}, the product $S(u)\,S(-u) = w(u)\cdot I$ defines a formal power series $w(u) = 1 + \sum_{r\ge 1} w_{2r} u^{-2r}$ with coefficients central in $X(\mfg_N,\mcG)^{tw}$. The twisted Yangian $Y(\mfg_N,\mcG)^{tw}$ is defined as the quotient of $X(\mfg_N,\mcG)^{tw}$ by the ideal generated by the coefficients of the {\it unitary relation} $S(u)\,S(-u)=I$, that is
\eq{
Y(\mfg_N,\mcG)^{tw} = X(\mfg_N,\mcG)^{tw} / (w(u)-1) . \label{Y=X/(w-1)}
}
By Theorem~3.1 of \cite{GR} the algebra $Y(\mfg_N,\mcG)^{tw}$ is isomorphic to the subalgebra of $Y(\mfg_N)$ generated by the coefficients $\si_{ij}^{(r)}$ with $r\ge1$ of the matrix entries $\si_{ij}(u)$ of the $S$-matrix $\Si(u)$ defined by 
\[
\Si(u) = \mcT(u-\ka/2)\, \mcG(u)\, \mcT^t(-u+\ka/2) . 
\]

Denote by $Z(\mfg_N,\mcG)^{tw}$ the subalgebra of $X(\mfg_N,\mcG)^{tw}$ generated by the even coefficients $w_{2r}$ with $r\ge0$. The subalgebra $Z(\mfg_N,\mcG)^{tw}$ is the centre of $X(\mfg_N,\mcG)^{tw}$. Moreover, the following tensor decomposition holds
\eq{
X(\mfg_N,\mcG)^{tw} \cong Z(\mfg_N,\mcG)^{tw} \ot Y(\mfg_N,\mcG)^{tw} . \label{X=Z*Y}
} 

Given $g(u)\in1+u^{-2}\C[[u^{-2}]]$, the assignment
\eq{
\nu_g : S(u) \mapsto g(u-\ka/2)\,S(u) \label{nu_g}
}
extends to an automorphism $\nu_g$ of $X(\mfg_N,\mcG)^{tw}$, and $Y(\mfg_N,\mcG)^{tw}$, viewed as a subalgebra of $X(\mfg_N)$, is the $\nu_g$-stable subalgebra of $X(\mfg_N,\mcG)^{tw}$: see Corollary 3.1 of \cite{GR}. 

A second family of automorphisms of $X(\mfg_N,\mcG)^{tw}$ is provided by conjugating $S(u)$ by certain invertible matrices: let $A\in GL(N)$ satisfy 
$AA^t=I$ and $A\mcG A^t=\mcG$. Then, by Remark 3.2 of \cite{GR}, the assignment 
\begin{equation}
 \alpha_A: S(u)\mapsto AS(u)A^t \label{al_A}
\end{equation}
 defines an automorphism $\alpha_A$ of $X(\mfg_N,\mcG)^{tw}$ which factors through the quotient $Y(\mfg_N,\mcG)^{tw}$.

For each $i,j\in \mcI_N$, set $F_{ij}^{\rho}=(g_{ii}+g_{jj})F_{ij}$ and define $\bar g_{ij}=(g_{ij}-\delta_{ij})d$, where we recall that $d=(p-q)/4$. By Proposition 3.9 of \cite{GRW2}, the fixed point subalgebra $\mfU\mfg_N^\rho\subset \mfU\mfg_N$ is generated by the $F_{ij}^{\rho}$, and the assignments
\begin{equation}
F_{ij}^{\rho}\mapsto s_{ij}^{(1)}-\bar{g}_{ij} \qu \text{and} \qu F_{ij}^{\rho}\mapsto \si_{ij}^{(1)}-\bar{g}_{ij} \quad \text{ for all }\quad i,j\in \mcI_N \label{grho->X}
\end{equation}
extend to injective algebra homomorphisms $\mfU\mfg_N^\rho\into X(\mfg_N,\mcG)^{tw}$ and $\mfU\mfg_N^\rho\into Y(\mfg_N,\mcG)^{tw}$, respectively.

The matrix $\mcG(u)$ provides the trivial representation $V(\mcG)$ of $X(\mfg_N,\mcG)^{tw}$ defined by the map $\eps : S(u) \mapsto \mcG(u)$, which is the restriction of the counit of $X(\mfg_N)$ to $X(\mfg_N,\mcG)^{tw}$. When $\mfg_N^\rho=\mfg_p\op\mfg_q$ with $\mfg_p\ncong \mfso_2\ncong\mfg_q$, this is a unique one-dimensional representation of $X(\mfg_N,\mcG)^{tw}$, up to twisting by automorphisms of the form $\nu_g$: see Subsection \ref{subsec:1dim-gen}. When $(\mfg_N,\mfg_N^\rho)=(\mfso_N,\mfso_{N-2}\oplus \mfso_2)$ with 
$N\geq 5$, $V(\mcG)$ is a special case of a one-parameter family of one-dimensional representations which will be constructed in Subsection \ref{subsec:1dim-so2}, and if $N=4$ it belongs to a two-parameter family of one-dimensional representations; this will be explained in Subsection \ref{sec:so4so2so2}. In the case when $(\mfg_N,\mfg_N^\rho)=(\mfso_3,\mfso_2)$, $V(\mcG)$ is again a particular member of a one-parameter family of one-dimensional representations, as will be seen in Subsection \ref{sec:so3so2}. 

We will also make use of the rational function (see Lemma 2.2 of \cite{GRW2}) 
\eq{
p(u) =  1 \mp \frac{1}{2u-\ka} + \frac{ {\rm tr}(\mcG(u))}{2u-2\ka} \qq\text{satisfying}\qq p(u)\,p(\ka-u) = 1 - \frac{1}{(2u-\ka)^2}\,. \label{p(u)}
}
When it is necessary to emphasize the dependence of $p(u)$ on the matrix $\mcG$ we will denote it instead by $p_\mcG(u)$. 

Lastly, we remark that we will sometimes write $X(\mfg_N,\mfg_N^\rho)^{tw}$ (resp.~$Y(\mfg_N,\mfg_N^\rho)^{tw}$) instead of $X(\mfg_N,\mcG)^{tw}$ (resp.~$Y(\mfg_N,\mcG)^{tw}$).

\subsubsection{Highest weight representations} \label{subsub:HWtw}

We now briefly summarize the relevant results of the highest weight theory for representations of $X(\mfg_N,\mcG)^{tw}$ obtained in Section 4 of \cite{GRW2}.

A representation $V$ of $X(\mfg_N,\mcG)^{tw}$ is called a \textit{highest weight representation} if there exists a nonzero vector $\eta\in V$ such that $V=X(\mfg_N,\mcG)^{tw}\eta$ and the following conditions are met: 
\begin{alignat*}{4}
  &s_{ij}(u)\,\eta=0 \quad &&\text{ for all }\quad && i<j \in \mcI_N , \quad &&\text{ and } \nonumber\\
  &s_{ii}(u)\,\eta=\mu_i(u)\,\eta \quad  &&\text{ for all }\quad && i\in \mcI_N^+, && 
\end{alignat*}
where $\mu_i(u)=g_{ii}+\sum_{r\ge1} \mu_{i}^{(r)}u^{-r}$ are formal power series in $u^{-1}$ with coefficients $\mu_i^{(r)}\in \C$. The vector $\eta$ is called the \textit{highest weight vector} and the $(N-n)$-tuple $\mu(u)=\left((\mu_{0}(u)),\mu_{1}(u),\ldots,\mu_{n}(u)\right)$ is called the \textit{highest weight} of $V$ (here $(\mu_{0}(u))$ is only present if $N=2n+1$). Given a highest weight $\mu(u)$, we shall frequently make use of the corresponding tuple $\wt \mu(u)$ whose components are given by 
\begin{equation}
 \wt\mu_i(u)=(2u-n+i)\,\mu_i(u)+\sum_{\ell=i+1}^n\mu_\ell(u) \qu\text{for all}\qu i\in \mcI_N^+. \label{tilde-mu(u)}
\end{equation}

As in the case of $X(\mfg_N)$, every finite-dimensional irreducible representation $V$ of $X(\mfg_N,\mcG)^{tw}$ is a highest weight representation (Theorem 4.5 of \cite{GRW2}). Additionally, $V$ contains a highest weight vector $\eta$, unique up to scalar multiplication. Given a tuple $\mu(u)$, the Verma module $M(\mu(u))$ for $X(\mfg_N,\mcG)^{tw}$ is defined as the quotient of $X(\mfg_N,\mcG)^{tw}$ by the left ideal generated by all the coefficients of the series $s_{ij}(u)$ with $i<j\in\mcI_N$ and $s_{ii}(u) - \mu_i(u)$ with $i\in\mcI_N^+$. If $M(\mu(u))$ is non-trivial, then it is a highest weight module with the highest weight $\mu(u)$ and the highest weight vector $1_{\mu(u)}$ equal to the image of the identity element $1\in X(\mfg_N,\mcG)^{tw}$ under the natural quotient map  $X(\mfg_N,\mcG)^{tw}\to M(\mu(u))$.

Whenever the symbols $[\pm]$ and $[\mp]$ occur, the lower sign corresponds to the type BI(b) case while the upper sign corresponds to all the other cases.

\begin{prop}[{\cite[Proposition 4.17]{GRW2}}]\label{P:nontriv}
Let $\mu(u)=(\mu_i(u))_{i\in \mcI_N^+}$, where for each $i$ we have $\mu_i(u)\in g_{ii}+u^{-1}\C[[u^{-1}]]$. Then the $X(\mfg_N,\mcG)^{tw}$ Verma module 
$M(\mu(u))$ is nontrivial if and only if 
\begin{equation}
 \wt \mu_i(u)\,\wt \mu_i(-u+n-i)=\wt \mu_{i+1}(u)\,\wt \mu_{i+1}(-u+n-i) \; \text{ for all }\; i\in \mcI_{N}^+\setminus\{ n\},\label{nontriv.1}
\end{equation}
and in addition, if $\mfg_N=\mfso_{2n+1}$, the series $(\ka-u)\cdot \mathscr{g}(\ka-u)\,\wt \mu_0(u)$, where
\begin{equation}
\mathscr{g}(u) = [\pm]\left(\frac{p+q-4u}{p-q-4u}\right), \label{g(u)} 
\end{equation}
is invariant under the transformation $u\mapsto \ka-u$:
\begin{equation}
 u\cdot \mathscr{g}(u)\wt \mu_0(\ka-u)=(\ka-u)\cdot \mathscr{g}(\ka-u)\wt \mu_0(u). \label{nontriv.2}
\end{equation}
\end{prop}
In Proposition 4.17 of \cite{GRW2} the relation \eqref{nontriv.2} was presented instead in the form 
\begin{equation*}
u\cdot \wt \mu_0(\ka-u)=(\ka-u)\cdot p_I(u)\,p(u)^{-1}\wt \mu_0(u),
\end{equation*}
%
where $p_I(u)$ is the rational function $p(u)$ corresponding to the pair $(\mfg_N,\mfg_N)$, {\it i.e.}~$\mcG=I$. It was proven of \cite[Remark 4.4]{GRW2} that $p_I(u)\,p(u)^{-1}=\mathscr{g}(\ka-u)\,\mathscr{g}(u)^{-1}$ and hence that the two forms of \eqref{nontriv.2} are equivalent. We now take a moment to provide a more elegant proof of this equality which holds more generally whenever $(\mfg_N,\mfg_N^\rho)=(\mfg_N,\mfg_{p}\oplus \mfg_q)$ with $p\neq q$. In addition, we state an important corollary of this proof which will be used in Section \ref{sec:1dim}.

\begin{prop}
The following equality holds whenever $p\neq q$: \[ p_I(u)\,p(u)^{-1}=\mathscr{g}(\ka-u)\,\mathscr{g}(u)^{-1}.\]
\end{prop}
\begin{proof}
The defining symmetry relation of $X(\mfg_N,\mcG)^{tw}$ \eqref{TX-symm} is equivalent to 
\begin{equation}
 p(u)s_{ij}(\ka-u)=\theta_{ij}s_{-j,-i}(u)\mp \frac{s_{ij}(u)}{2u-\ka}+\frac{\delta_{ij} \Tr (S(u))}{2u-2\ka} \; \text{ for all }\; i,j\in \mcI_N. \label{TrS.0}
\end{equation}
Taking the sum of both sides over all $i=j$ we obtain 
\begin{equation}
 p(u)\Tr(S(\ka-u))=\left(1\mp \frac{1}{2u-\ka}+\frac{N}{2u-2\ka} \right)\Tr(S(u))=p_I(u)\Tr(S(u)). \label{TrS}
\end{equation}
Applying the counit $\epsilon$ (i.e. the trivial representation) to both sides we obtain 
\[
 p(u)\,\Tr(\mcG(\ka-u))=p_I(u)\Tr(\mcG(u)).
\]
When $p\neq q$ the series $\Tr(\mcG(u))$ is an invertible element of $\C[[u^{-1}]]$, and therefore the above relation can be equivalently expressed as
\[
 p_I(u)\,p(u)^{-1}=\frac{\Tr(\mcG(\ka-u))}{\Tr(\mcG(u))}.
\]
as $\Tr(\mcG(u))=\Tr\left(\frac{(p-q)I-4u\mcG}{p-q-4u}\right)=(p-q)\frac{p+q-4u}{p-q-4u}$, the right hand side is just $\mathscr{g}(\ka-u)\mathscr{g}(u)^{-1}$. \qedhere
\end{proof}

Replacing the relation \eqref{TrS.0} with equation (5.23) of \cite{GR}, we obtain the following corollary:
\begin{crl}\label{C:c(u)}
 Suppose that $p\neq q$. Then the central series $c(u)\in \wt X(\mfg_N,\mcG)^{tw}[[u^{-1}]]$ (see \eqref{c(u)}) satisfies the relation 
 \begin{equation*}
  c(u)=\frac{\mathscr{g}(\ka-u)}{\mathscr{g}(u)}\cdot \frac{\Tr(\wt S(u))}{\Tr(\wt S(\ka-u))}. \label{c(u)<->Tr}
 \end{equation*}
\end{crl}

We also remark that \eqref{TrS} holds for any symmetric pair of type B-C-D, and in particular it implies that $\Tr(S(u))=\Tr(S(\ka-u))$ when $(\mfg_N,\mfg_N^\rho)=(\mfg_N,\mfg_N)$.

When the $X(\mfg_N,\mcG)^{tw}$ Verma module $M(\mu(u))$ is non-trivial, it admits a unique irreducible quotient $V(\mu(u))$, and every finite-dimensional irreducible module is isomorphic to a module of this form.

Given a representation $V$ of $X(\mfg_N,\mcG)^{tw}$, denote by $V^0$ the subspace 
\begin{equation*}
 V^0=\{\xi\in V\,:\,s_{ij}(u)\,\xi=0 \; \text{ for all }\; i<j\}. 
\end{equation*}
The next corollary follows from a modification of the proof uniqueness for the highest weight vector of Theorem 4.5 in \cite{GRW2} and is analogous to Corollaries 3.2.8 and 4.2.7 of \cite{Mobook}:
\begin{crl}\label{C:hw-1dim}
 Let $\mu(u)$ satisfy the conditions of Proposition \ref{P:nontriv} and let $\xi\in V(\mu(u))$ be a highest weight vector. Then $V(\mu(u))^0=\C\xi$. 
\end{crl}

Let $\eta$ denote the highest weight vector of the irreducible $X(\mfg_N,\mcG)^{tw}$-module $V(\mu(u))$, and $\xi$ the highest weight vector of the irreducible $X(\mfg_N)$-module $L(\lambda(u))$. By Proposition 4.10 of \cite{GRW2}, $X(\mfg_N,\mcG)^{tw}(\xi\otimes \eta)$ is a highest weight  $X(\mfg_N,\mcG)^{tw}$-module with the highest weight vector $ \xi\otimes \eta$, and the highest weight $\gamma(u)$ whose components are determined by the relations
\begin{equation}
\wt\ga_i(u)=\wt \mu_{i}(u)\la_i(u-\ka/2)\la_{-i}(-u+\ka/2) \qu\text{for all}\qu i\in \mcI_N^+. \label{HWT:tensors}
\end{equation}
Define the non-negative integers $\ley$ and $\key$ by 
\begin{equation}
 \ley=\begin{cases}
       \frac{N-q}{2}\; & \text{ if } \; (\mfg_N,\mfg_N^\rho) \text{ is type BI(b)},\\
       \frac{q}{2}\; & \text{ otherwise,}
      \end{cases} \quad \text { and } \quad \key=n-\ley. \label{ley-key}
\end{equation} 
Since $(\mfg_N,\mfg_N^\rho)$ is not type DI(b), both $\ley$ and $\key$ are non-negative integers, with $\ley<n$ and $\key>0$ except in the case when 
$(\mfg_N,\mfg_N^\rho)$ is type BI(b) with $q=1$ ({\it i.e.}~when $(\mfg_N,\mfg_N^\rho)=(\mfso_{2n+1},\mfso_{2n})$). In Subsection 4.3 of \cite{GRW2}, the pair $(\ley,\key)$ was denoted $(\boldsymbol{\ell},\boldsymbol{k})$.

An important instance of \eqref{HWT:tensors} occurs when $\mu(u)$ is taken to be $(g_{ii}(u))_{i\in \mcI_N}$, in which case $V(\mu(u))=V(\mcG)$ and \eqref{HWT:tensors}
provides formulas for the highest weight $\ga(u)$ of the $X(\mfg_N,\mcG)^{tw}$-module $X(\mfg_N,\mcG)^{tw}\xi \subset L(\lambda(u))$. By \eqref{G(u)}, we have 
\begin{equation}
\wt g_{ii}(u)=2u\left(\frac{p-q[\pm]4u}{p-q-4u} \right) \quad \text{ for all }\quad \key+1\leq i\leq n, \label{eq:wtgii1}
\end{equation}
while for $i\in \mcI_N^+$ satisfying $0\leq i\leq \key$  we have by \eqref{g(u)} that
\begin{equation}
 \wt g_{ii}(u)=(2u-\ley)\left(\frac{p-q[\mp]4u}{p-q-4u} \right) + \ley\left(\frac{p-q[\pm]4u}{p-q-4u} \right) =[\pm]2u\left(\frac{p+q-4u}{p-q-4u} \right)=2u\cdot\mathscr{g}(u). \label{eq:wtgii2}
\end{equation}
Hence \eqref{HWT:tensors} becomes 
\begin{equation*}
 \wt \ga_i(u)=\begin{cases}
               2u\cdot \mathscr{g}(u)\la_i(u-\ka/2)\la_{-i}(-u+\ka/2) \; &\text{ if } \; 0\leq i\leq \key,\\
               2u\left(\frac{p-q[\pm]4u}{p-q-4u} \right)\la_i(u-\ka/2)\la_{-i}(-u+\ka/2) \; &\text{ if } \; \key+1\leq i\leq n
              \end{cases}
\end{equation*}
for each $i\in \mcI_N^+$: see also Corollary 4.11 of \cite{GRW2}.

In \cite{GRW2} the finite-dimensional irreducible representations of $X(\mfg_N,\mfg_N^\rho)^{tw}$ were classified for $(\mfg_N,\mfg_N^\rho)$ of type CI, DIII and BCD0. In the present paper the corresponding classification results for the extended twisted Yangians $X(\mfg_N,\mfg_N)^{tw}$ of type BCD0 will play an important role, and hence we recall them here.

We emphasize that all the definitions provided in Subsection \ref{subsec:twYa} still apply for the (extended) twisted Yangians of type BCD0: one must just substitute $(p,q)=(N,0)$. In particular, 
the definition of $\mcG(u)$ provided by \eqref{G(u)} collapses to $\mcG(u)=I$. The definitions of $\key$ and $\ley$ given in \eqref{ley-key} also extend  to include pairs of this type, where we have $(\key,\ley)=(n,0)$.

Let $\delta=1$ if $\mfg_N=\mfsp_{N}$ and $\delta=0$ if $\mfg_N=\mfso_N$. 
\begin{thrm}[Theorems 6.5 and 6.6 of \cite{GRW2}]\label{T:BCD0-class}
 Suppose that $\mu(u)=(\mu_i(u))_{i\in \mcI_N^+}$ satisfies the conditions of Proposition \ref{P:nontriv}. Then the irreducible $X(\mfg_N,\mfg_N)^{tw}$-module 
 $V(\mu(u))$ is finite-dimensional if and only if there exists monic polynomials $P_1(u),\ldots,P_n(u)$ such that 
\[
  \frac{\wt \mu_{i-1}(u)}{\wt \mu_i(u)}=\frac{P_i(u+1)}{P_i(u)} \quad \text{ with  }\quad P_i(u)=P_i(-u+n-i+2) \; \text{ for all }\; 2\leq i\leq n, 
\]
and $P_1(u)$ satisfies $P_1(u)=P_1(-u+\ka+2^\delta)$ together with the relation
\begin{align*}
  \frac{\wt \mu_{1}(\ka-u)}{\wt \mu_{2^{1-\delta}}(u)}&=\frac{P_1(u+2^\delta)}{P_1(u)}\cdot \frac{\ka-u}{u} \; \text{ if }\; \mfg_N=\mfg_{2n}, \\
  \frac{\wt \mu_{0}(u)}{\wt \mu_{1}(u)}&=\frac{P_1(u+\tfrac{1}{2})}{P_1(u)} \; \text{ if }\; \mfg_N=\mfso_{2n+1}.
\end{align*}
Moreover, when $V(\mu(u))$ is finite-dimensional the associated tuple $(P_1(u),\ldots,P_n(u))$ is unique.
\end{thrm}
We end this subsection by noting that the isomorphism \eqref{X=Z*Y} together with the definition of $Y(\mfg_N,\mcG)^{tw}$ implies the following statement:
\begin{prop} \label{P:Y^tw-fd}
The isomorphism classes of finite-dimensional irreducible representations of $Y(\mfg_N,\mcG)^{tw}$ can be naturally identified with the isomorphism classes of finite-dimensional irreducible $X(\mfg_N,\mcG)^{tw}$-modules in which the central series $w(u)$ acts as the identity operator. 
\end{prop}
%

\subsection{Preliminary properties of polynomials}\hspace{0mm}

In this subsection we prove two elementary results pertaining to polynomials satisfying certain symmetry relations: see Lemmas \ref{L:poly1} and \ref{L:poly2}. Both of these lemmas are generalizations of similar results which have appeared in Chapters 3 and 4 of \cite{Mobook}. As was the case in {\it loc.~cit.}, these results play a role in classifying the finite-dimensional irreducible representations of twisted Yangians. 
\begin{lemma}\label{L:poly1}
 Let $\alpha,\beta\in \C$, $l\in \Z$ and $m\in \mathbb{Q}$. Suppose that $P(u)$ and $Q(u)$ are both monic polynomials such that $P(u)=P(-u+l)$ and $Q(u)=Q(-u+l)$. Suppose
 also that $P(\alpha)\neq 0\neq Q(\beta)$ and that 
 \begin{equation}
  \frac{P(u+m)}{P(u)}\cdot \frac{\alpha-u}{\alpha+u-l+m}=\frac{Q(u+m)}{Q(u)}\cdot \frac{\beta-u}{\beta+u-l+m}. \label{poly1}
 \end{equation}
Then $P(u)=Q(u)$ and $\alpha=\beta$. 
\end{lemma}
 
\begin{proof}
This is a generalization of a result proven as part of the proof of Theorem 4.4.3 of \cite{Mobook} (see in particular (4.58)). There the statement of the lemma was proven in the special case where $l=m=1$. The same argument works in the general case, and we repeat it here for the sake of the reader. 
 
If $\alpha=\beta$ then $\frac{P(u+m)}{P(u)}=\frac{Q(u+m)}{Q(u)}$, which implies that the rational function $f(u)=Q(u)/P(u)$ is periodic. This is impossible unless
 $f(u)$ is constant, and,  since $P(u)$ and $Q(u)$ are monic, this is only possible if $f(u)=1$. Hence $P(u)=Q(u)$. 
 
Therefore it suffices to show that an equality of the form \eqref{poly1} is impossible unless $\alpha=\beta$. We prove this by induction on $k$, where
 $k=\frac{1}{2}(\deg\,P(u)+\deg\,Q(u))$. If $k=0$ then this follows from the fact that \eqref{poly1} collapses to $\frac{\alpha-u}{\alpha+u-l+m}=\frac{\beta-u}{\beta+u-l+m}$.
 Suppose inductively that \eqref{poly1} is impossible whenever $\alpha\neq \beta$ and $k<M$ for some $M\in \mathbb{N}$. Assume now $k=M$. By symmetry, we may assume without loss of generality 
 that $\deg\, P(u)\geq 2$. Additionally, without loss of generality we may assume that $P(u)$ and $Q(u)$ have no common roots. Let $u_0$ be a root of $P(u)$ such that $u_0+m$ is not a root. Then \eqref{poly1} implies $u_0=l-m-\beta$. 
 Write $P(u)=P'(u)(u-u_0)(u+u_0-l)$ and set $\beta'=l-u_0$. Then we have 
  \begin{equation*}
  \frac{P'(u+m)}{P'(u)}\cdot \frac{\alpha-u}{\alpha+u-l+m}=\frac{Q(u+m)}{Q(u)}\cdot \frac{\beta'-u}{\beta'+u-l+m},
 \end{equation*}
 and $\alpha\neq \beta'$ due to the fact that $P(\beta')=0$. By the induction hypothesis, this is impossible. 
\end{proof}

\begin{lemma}\label{L:poly2}
Let $\alpha\in \C$, $l\in \Z$ and $m\in \mathbb{Q}$. Suppose that $P(u)$ is a monic polynomial such that $P(u)=P(-u+l)$. Then there exists a pair $(\ell_\al^m,P_\al^m(u))$, where $\ell_\al^m\in \Z_{\geq 0}$ and
$P_\al^m(u)$ is a monic polynomial  such that $P_\al^m(u)=P_\al^m(-u+l)$, satisfying
\begin{equation}
 \frac{P(u+m)}{P(u)}\cdot \frac{\alpha-u}{\alpha+u-l+m}=\frac{P_\al^m(u+m)}{P_\al^m(u)}\cdot \frac{(\alpha-m\,\ell_\al^m)-u}{(\alpha-m\,\ell_\al^m)+u-l+m} \; \text{ and }\; P_\al^m(\alpha-m\,\ell_\al^m)\neq 0. \label{poly2} 
\end{equation}
Moreover, the pair $(\ell_\al^m,P_\al^m(u))$ is unique with $P_\al^m(u)$ equal to $P(u)$ divided by 
\begin{equation}
Q(u)=\prod_{k=0}^{\ell_\al^m-1}(u-\alpha+km)(u-l+\alpha-km). \label{poly1'}
\end{equation}
\end{lemma}

\begin{proof}
We first define the pair $(\ell_\al^m,P_\al^m(u))$ and show that it satisfies the desired properties. For each $a\geq 0$, set
\begin{equation}
 P^{(a)}(u)=\frac{P(u)}{\prod_{k=0}^{a-1}(u-\alpha+km)(u-l+\alpha-km)}\in \C(u), \label{def:P^a}
\end{equation}
where $P^{(0)}(u)=P(u)$. Note that $P^{(a)}(u)$ will be a monic polynomial in $u$ satisfying $P^{(a)}(u)=P^{(a)}(-u+l)$ whenever $P(u)$ is divisible by $\prod_{k=0}^{a-1}(u-\alpha+km)(u-l+\alpha-km)$. Define
\[
\ell=\begin{cases}
      0 \;\text{ if } \;P(\alpha)\neq 0,\\
      \min_{k\geq 1}\{P^{(k-1)}(\alpha-(k-1)m)=0,\,P^{(k)}(\alpha-km)\neq 0 \} \; \text{ otherwise.}
     \end{cases}
\]
It a straightforward consequence of the above definitions that $P^{(\ell)}(u)$ is a monic polynomial in $u$ satisfying $P^{(\ell)}(u)=P^{(\ell)}(-u+l)$.
%
%
We may now set $\ell_\al^m=\ell$ and $P_\al^m(u)=P^{(\ell)}(u)$. It remains to verify \eqref{poly2}. By the definition of $\ell_\al^m$ we have $P_\al^m(\alpha-m\,\ell_\al^m)\neq 0$. Moreover, 
\begin{align*}
 \frac{P(u+m)}{P(u)}&=\frac{P_\al^m(u+m)}{P_\al^m(u)}\cdot \frac{\prod_{k=0}^{\ell_\al^m-1}(u-\alpha+(k+1)m)(u-l+\alpha-(k-1)m)}{\prod_{k=0}^{\ell_{\al}^m-1}(u-\alpha+km)(u-l+\alpha-km)}\\
                    &=\frac{P_\al^m(u+m)}{P_\al^m(u)}\cdot \frac{(\alpha-\ell_\al^m m)-u}{(\alpha-\ell_\al^m m)+u-l+m}\cdot \frac{\alpha+u-l+m}{\alpha-u},
\end{align*}
which implies that \eqref{poly2} holds. 

Finally, note that the uniqueness of $(\ell_\al^m,P_\al^m(u))$ is an immediate corollary of Lemma \ref{L:poly1}. 
\end{proof}

The statement of Lemma \ref{L:poly2} did not explicitly appear in \cite{Mobook}, but similar ideas were needed in the proof of Theorem 4.4.14 in {\it loc.~cit.} 
Note that in both of these lemmas the assumption that $m\in \mathbb{Q}$ is not necessary. We will however only be concerned with this case.

Let us end this subsection by introducing some notation related to polynomials: 

\begin{itemize}

\item If $P(u)$ is a polynomial in $u$, we denote by $Z(P(u))$ its zero set. 

\item Given $\alpha,\beta \in \C$ such that $\alpha-\beta\in \Z$, we define the string $S(\alpha,\beta)$ corresponding to $(\alpha,\beta)$ to be the~set
\begin{equation}\label{string}
 S(\alpha,\beta)=\begin{cases}
                  \{\beta,\beta+1,\ldots,\alpha-1\} & \text{ if } \alpha-\beta\in \Z_{>0}, \\ 
                  \emptyset & \text{ otherwise. } 
                 \end{cases}
\end{equation}

\end{itemize}


\section{Representations of low rank twisted Yangians of type BDI} \label{sec:lowrank}

In this section, we use the classification results for finite-dimensional irreducible representations of the Olshanskii twisted Yangian $Y^+(2)$ \cite{Mo1} together with the isomorphisms from \cite{GRW1} to classify all finite-dimensional irreducible representations of the twisted Yangians $X(\mfso_4,\mfso_2\op\mfso_2)^{tw}$ and $X(\mfso_3,\mfso_2)^{tw}$. The classification obtained for $X(\mfso_3,\mfso_2)^{tw}$ will provide a necessary step in proving the main results in Subsection \ref{subsec:q=1}. Here we also obtain explicit formulas for evaluation morphisms $X(\mfso_4,\mfso_2\op\mfso_2)^{tw} \onto \mfU(\mfso_2\op\mfso_2)$ and $X(\mfso_3,\mfso_2)^{tw} \onto \mfU\mfso_2$, and study the corresponding evaluation modules. 

In order to distinguish between the generators of $X(\mfg_N,\mfg_N^\rho)^{tw}$ and those of $Y^+(2)$, we shall follow the convention established of \cite{GRW1} and denote the 
generators of $Y^+(2)$ by $s^{\circ (r)}_{ij}$, where $i,j\in \{\pm1\}$ and $r\geq0$. These generators are then arranged as the coefficients of the various series' $s_{ij}^\circ(u)$, which 
in turn form the $(i,j)^{th}$ entry of the matrix $S^\circ(u)$. 
Similarly, the generators of the special twisted Yangian $SY^+(2)$ are denoted by $\si^{\circ(r)}_{ij}$, and the
corresponding series and matrix are denoted by $\si^\circ_{ij}(u)$ and $\Si^\circ(u)$, respectively. The defining relations of $Y^+(2)$ are the reflection equation 
\[ 
R^{\circ}(u-v)\, S_1^{\circ}(u)\, R^{\circ t}(-u-v)\, S_2^{\circ}(v) = S_2^{\circ}(v)\, R^{\circ t}(-u-v)\, S_1^{\circ}(u)\, R^{\circ}(u-v) ,
\] 
where $R^{\circ}(u) = 1 - u^{-1} \sum_{i,j\in\mcI_2} E_{ij} \ot E_{ji}$ and the symmetry relation
\eq{
S^{\circ t}(-u) = S^\circ(u) + \frac{S^\circ(u)-S^\circ(-u)}{2u} \,. \label{Ols:symm}
}
The special twisted Yangian $SY^+(2)$ is the quotient of $Y^+(2)$ by the ideal generated by the coefficients of $\mathrm{sdet}S^\circ(u)-1$, where $ \mathrm{sdet}S^\circ(u)$ is the Sklyanin determinant \cite{Mobook}.

We now recall the classification results for finite-dimensional irreducible representations of $Y^+(2)$ and $SY^+(2)$.
A representation $V$ of $Y^+(2)$ is called a \textit{highest weight representation} if there exists a nonzero vector $\xi\in V$ such that $V=Y^+(2)\,\xi$, $s_{-1,1}^{\circ}(u)\,\xi=0$ and $s^\circ_{11}(u)\,\xi=\mu^\circ(u)\,\xi$ for some formal series $\mu^\circ(u)\in 1+u^{-1}\C[[u^{-1}]]$. As usual, we call $\mu^\circ(u)$ the highest weight of $V$, and the vector $\xi$ the highest weight vector. These same definitions can be given in the $SY^+(2)$ setting after replacing $s_{ab}^\circ(u)$ with $\sigma^\circ_{ab}(u)$ for each $a,b\in \{-1,1\}$. 

%

Given $\mu^\circ(u)\in 1+u^{-1}\C[[u^{-1}]]$, the Verma module $M(\mu^\circ(u))$ for $Y^+(2)$ is defined the same way as for $X(\mfg_N)$ and $X(\mfg_N,\mcG)^{tw}$, and is always non-trivial.
It admits a unique irreducible quotient $V(\mu^\circ(u))$, and any irreducible highest weight module with the highest weight $\mu^\circ(u)$ is isomorphic to $V(\mu^\circ(u))$. We will also employ the notation $M(\mu^\circ(u))$ and $V(\mu^\circ(u))$ for the $SY^+(2)$ Verma module of highest weight $\mu^\circ(u)$ and its irreducible quotient, respectively. In the $SY^+(2)$ case, the Verma module $M(\mu^\circ(u))$ is non-trivial provided that $\mu^\circ(-u)\mu^\circ(u-1)=1$.  The distinction between $SY^+(2)$ and $Y^+(2)$ will always be clear from context.  

The following classification result is a restatement of Theorems 4.4 and 5.4 of \cite{Mo1} (see also Theorems 4.3.3 and 4.4.3 of \cite{Mobook}): the irreducible $Y^+(2)$-module  $V(\mu^\circ(u))$ is finite-dimensional if and only if there exists a scalar $\ga^\circ\in \C$ together with a monic polynomial $P^\circ(u)$ such that $P^\circ(u)=P^\circ(-u+1)$, $P^\circ(\ga^\circ)\neq 0$, and 
\eq{
\frac{\mu^\circ(-u)}{\mu^\circ(u)}= \frac{2u+1}{2u-1} \cdot \frac{P^\circ(u+1)}{P^\circ(u)}\cdot \frac{u-\ga^\circ}{u+\ga^\circ} \,. \label{Y+2:findim}
}
In this case, the pair $(P^\circ(u),\ga^\circ)$ is unique. The same classification result holds if $Y^+(2)$ is replaced with $SY^+(2)$, and in this case it establishes a bijective correspondence between finite-dimensional irreducible representations and pairs of the form $(P^\circ(u),\ga^\circ)$: see \cite[Corollary 4.4.5]{Mobook}.


\subsection{Twisted Yangians for the symmetric pair \texorpdfstring{$(\mfso_{4},\mfso_{2}\op\mfso_2)$}{}} \label{sec:so4so2so2}

The isomorphism between the twisted Yangian $X(\mfso_4,\mfso_2\oplus\mfso_2)^{tw}$ and the tensor product of Olshanskii twisted Yangians $SY^+(2)\ot Y^+(2)$ was established in Proposition 4.16 and Corollary 4.17 of \cite{GRW1}. Let $K=E_{11}-E_{-1,-1} \in \End(\C^2)$ and let $\Sigma^\circ(u)$ denote the $S$-matrix of $SY^+(2)$ and $S^\bullet(u)$ that of $Y^+(2)$. Set $V=\C^2\otimes \C^2$ with ordered basis given by $v_{-2}=e_{-1}\otimes e_{-1}$, $v_{-1}=e_{-1}\otimes e_{1}$, $v_1=e_1\otimes e_{-1}$ and $v_2=-e_1\otimes e_1$. By identifying $V$ with $\C^4$ equipped with basis $\{v_{-2},v_{-1},v_1,v_2\}$, we can consider $S(u)$ as an element of $\End\,V\otimes X(\mfso_4,\mfso_2\op\mfso_2)^{tw}[[u^{-1}]]$.  Then the map
\eq{
\chi \;:\; S(u) \mapsto -\Sigma_1^\circ(u-1/2) K_1 S^\bullet_2(u-1/2) K_2 \label{iso:so4}
}
defines an isomorphism $X(\mfso_4,\mfso_2\oplus\mfso_2)^{tw} \cong SY^+(2)\ot Y^+(2)$: see Section E in \textit{loc.~cit.} for the precise meaning of the right-hand side of \eqref{iso:so4}. The sign difference between \eqref{iso:so4} and (4.59) of \cite{GRW1} is due to the fact that the matrix $\mcG$ that we use equals the matrix $-\mcG'$ used in \textit{loc.~cit.} We will use this result to obtain a complete description of the finite-dimensional irreducible representations of $X(\mfso_4,\mfso_2\oplus\mfso_2)^{tw}$ using those of $SY^+(2)$ and $Y^+(2)$ as recalled above.

\begin{prop}\label{P:so4class}
Let the components of $\mu(u) = (\mu_1(u),\mu_2(u))$ satisfy the conditions of Proposition \ref{P:nontriv} so that the irreducible $X(\mfso_4,\mfso_2\op\mfso_2)^{tw}$-module $V(\mu(u))$ exists. Then $V(\mu(u))$ is finite-dimensional if and only if there exists 
a tuple $(Q(u),P(u),\alpha,\beta)$, where $\alpha,\beta\in \C$ and $P(u)$, $Q(u)$ are monic polynomials in $u$, such that $P(u)=P(-u+2)$, $Q(u)=Q(-u+2)$, $P(\alpha)\neq 0$, $Q(\beta)\neq 0$, and 
\[
 \frac{\widetilde{\mu}_1(u)}{\widetilde{\mu}_2(u)}=\frac{P(u+1)}{P(u)}\cdot \frac{\alpha-u}{\alpha+u-1} , \qq
 \frac{\widetilde{\mu}_1(1-u)}{\widetilde{\mu}_2(u)}=\frac{u}{1-u}\cdot\frac{Q(u+1)}{Q(u)}\cdot \frac{\beta-u}{\beta+u-1} .
\]
Moreover, when they exist, the pair $(Q(u), P (u))$ and the scalars $\al,\beta$ are uniquely determined.
\end{prop}

\begin{proof} The proof of this proposition is very similar to that of Proposition 5.4 of \cite{GRW2}.
We begin by showing that the $X(\mfso_4,\mfso_2\oplus\mfso_2)^{tw}$-module $V(\mu(u))$, viewed as a $SY^+(2)\ot Y^+(2)$-module via the isomorphism $\chi$, is isomorphic to $ V(\mu^\circ(u))\ot V(\mu^\bullet(u))$, where the pair $(\mu^\circ(u),\mu^\bullet(u))$ is completely determined by the relations 
\begin{equation}
\wt{\mu}_1(u)=(2u-2)\cdot\mu^\circ(\wt u)\,\mu^\bullet(-\wt u), \qq 
\wt{\mu}_2(u)=-2u\cdot\mu^\circ(\wt u)\,\mu^\bullet(\wt u), \qquad
\mu^\circ(-u)\, \mu^\circ(1-u)=1\label{so4:class:4}.
\end{equation}
with $\wt u=u-1/2$. Writing the map \eqref{iso:so4} explicitly we have that 
\eq{
\chi \; : \; \begin{cases}\; 
\begin{aligned}
s_{11}(u) &\mapsto \si^\circ_{11}(\wt u)\,s^\bullet_{-1,-1}(\wt u) , \qu  & s_{-1,2}(u) &\mapsto \si^\circ_{-1,1}(\wt u)\,s^\bullet_{11}(\wt u) , \\
s_{22}(u) &\mapsto -\si^\circ_{11}(\wt u)\,s^\bullet_{11}(\wt u) , & s_{1,-2}(u) & \mapsto -\si^\circ_{1,-1}(\wt u)\,s^\bullet_{-1,-1}(\wt u),
\end{aligned}
\end{cases}
\label{so4:class:1}
}

Moreover, the computation at the beginning of the proof of \cite[Corollary 4.17]{GRW1} shows that 
\[
\chi \; : \; s_{1,-2}(-\wt{u})\,s_{-1,2}(\wt{u}) - s_{11}(-\wt{u})\,s_{22}(\wt{u}) \mapsto  s^\bullet_{-1,-1}(-u)\,s^\bullet_{11}(u-1) .
\]
Letting $\xi\in V(\mu(u))$ denote the highest weight vector, this in turn implies that 
\[
s^\bullet_{-1,-1}(-u)\, s^\bullet_{11}(u-1)\,\xi=-\mu_1(-\wt{u})\,\mu_2(\wt{u})\,\xi.
\] 
Using the symmetry relation \eqref{Ols:symm} of $Y^+(2)$, we can rewrite the equality above as 
\[
 \left(s^\bullet_{11}(u)+\frac{s^\bullet_{11}(u)-s^\bullet_{11}(-u)}{2u} \right)s^\bullet_{11}(u-1)\,\xi=-\mu_1(-\wt{u})\,\mu_2(\wt{u})\,\xi.
\]
By induction on the coefficients $s^{\bullet(r)}_{11}$ of $s^\bullet_{11}(u)$, this implies that there exists $\mu^\bullet(u)\in 1+u^{-1}\C[[u^{-1}]]$ satisfying 
\eq{
s^\bullet_{11}(u)\,\xi=\mu^\bullet(u)\,\xi. \label{so4:class:3}
}
Moreover, $\mu^\bullet(u)$ is uniquely determined by the relation 
\eq{
 \left(\mu^\bullet(u)+\frac{\mu^\bullet(u)-\mu^\bullet(-u)}{2u} \right)\mu^\bullet(u-1) = - \mu_1(-\wt{u})\,\mu_2(\wt{u}) . \label{so4:class:2}
}
Combining \eqref{so4:class:3} with \eqref{so4:class:1} implies that $\xi$ is also an eigenvector for the action of $\si^\circ_{11}(u)$ with weight $\mu^\circ(u)$ defined by $\mu_2(u) = -\mu^\circ(\wt u)\,\mu^\bullet(\wt u)$ or equivalently by
\[
\mu_1(u) = \mu^\circ(\wt u)\,\Big(\mu^\bullet(-\wt u)+\frac{\mu^\bullet(\wt u) - \mu^\bullet(-\wt u)}{2\wt u}\Big),
\]
where we have used that $\mu^\circ(-u)^{-1}=\mu^\circ(u-1)$: see the proof of \cite[Proposition 5.4]{GRW2}.
Using the notation introduced in \eqref{tilde-mu(u)} the last two equalities can be rewritten in the equivalent form \eqref{so4:class:4}. Finally, since $[\si^\circ_{ij}(u),s^\bullet_{kl}(v)]=0$, it follows immediately from the definition of $\xi$ and the two formulas $\chi(s_{-1,2}(u))=\si^\circ_{-1,1}(\wt u)\,s^\bullet_{11}(\wt u)$ and $\chi(s_{12}(u))=\sigma_{11}^\circ(\wt u)s_{-1,1}^\bullet(\wt u)$ that $\si^\circ_{-1,1}(u)\,\xi=s^\bullet_{-1,1}(u)\,\xi=0$. Thus, by the irreducibility of $V(\mu(u))$ we can conclude that
\[
V(\mu(u))\cong V(\mu^\circ(u))\ot V(\mu^\bullet(u)). 
\]

We can now use the isomorphism above to determine exactly when $V(\mu(u))$ is finite-dimensional. As recalled above \eqref{Y+2:findim}, the module $V(\mu^\circ(u))\ot V(\mu^\bullet(u))$  is finite-dimensional if and only if there exists a tuple ($P^\circ(u),P^\bullet(u),\ga^\circ,\ga^\bullet)$, where $\ga^\circ,\ga^\bullet\in \C$ and $P^\circ(u),Q^\bullet(u)$ are monic polynomials in $u$ such that $P^\circ(u)=P^\circ(-u+1)$, $P^\bullet(u)=P^\bullet(-u+1)$, $P^\circ(\ga_1)\ne0$, $P^\bullet(\ga_2)\ne0$, and the following equations hold: 
\begin{equation}
 \frac{\mu^\circ(-u)}{\mu^\circ(u)}=\frac{2u+1}{2u-1}\cdot\frac{P^\circ(u+1)}{P^\circ(u)}\cdot\frac{u-\ga^\circ}{u+\ga^\circ} 
\qu\text{and}\qu 
\frac{\mu^\bullet(-u)}{\mu^\bullet(u)}=\frac{2u+1}{2u-1}\cdot\frac{P^\bullet(u+1)}{P^\bullet(u)}\cdot \frac{u-\ga^\bullet}{u+\ga^\bullet} . \label{LRR:DI(a).6}
\end{equation}
Set $P(u)=P^\bullet(\wt u)$, $Q(u)=P^\circ(\wt u)$, $\al=\ga^\bullet+\tfrac{1}{2}$ and $\beta=\gamma^\circ+\tfrac{1}{2}$. Substituting $u\mapsto \wt u$, the above relations become
\[
 \frac{\mu^\circ(-\wt u)}{\mu^\circ(\wt u)}=\frac{2u}{2u-2}\cdot \frac{P(u+1)}{P(u)}\cdot \frac{u-\alpha}{u+\alpha-1}
\qu\text{ and }\qu
\frac{\mu^\bullet(-\wt u)}{\mu^\bullet(\wt u)}=\frac{2u}{2u-2}\cdot \frac{Q(u+1)}{Q(u)}\cdot \frac{u-\beta}{u+\beta-1},
\]
respectively. By relations \eqref{so4:class:4}, 
\begin{equation*}
 \frac{\wt \mu_1(u)}{\wt \mu_2(u)}=\frac{2-2u}{2u}\cdot \frac{\mu^\bullet(-\wt u)}{\mu^\bullet(\wt u)}
\qu\text{ and }\qu
\frac{\wt \mu_1(1-u)}{\wt \mu_2(u)}=\frac{\mu^\circ(-\wt u)}{\mu^\circ(\wt u)}.
\end{equation*}
Therefore, \eqref{LRR:DI(a).6} is equivalent to 
\begin{equation*}
 \frac{\wt \mu_1(u)}{\wt \mu_2(u)}= \frac{P(u+1)}{P(u)}\cdot \frac{\alpha-u}{u+\alpha-1}\quad \text{ and } \quad \frac{\wt \mu_1(1-u)}{\wt \mu_2(u)}=\frac{u}{1-u}\cdot \frac{Q(u+1)}{Q(u)}\cdot \frac{\beta-u}{u+\beta-1}.
\end{equation*}
Moreover, $P(u)=P^\bullet(u-1/2)=P^\bullet(-u+1/2+1)=P(-u+2)$, and the same is true for $Q(u)$. Finally, $P(\alpha)=P^\bullet(\alpha-1/2)=P^\bullet(\ga^\bullet)\neq 0$. Similarly, $Q(\beta)\neq 0$.
\end{proof}

We now construct the evaluation morphism $X(\mfso_4,\mfso_2\op\mfso_2)^{tw}\onto \mfU(\mfso_2\op\mfso_2)$. The fixed point subalgebra $\mfU(\mfso_2\op\mfso_2) \subset \mfU\mfso_4$ is generated by the elements $F_{11}$ and $F_{22}$. We recall that $\mcG = {\rm diag}(-1,1,1,-1)$ in this case.

\begin{prop} \label{P:so4:evhom} 
The assignment 
\begin{equation}
s_{ij}(u)\mapsto g_{ij}+2g_{ij}F_{ij}u^{-1}+\delta_{ij}(F_{11}^2-F_{22}^2)u^{-2} \label{so4:evhom}
\end{equation}
defines a surjective algebra homomorphism $\mathrm{ev}:X(\mfso_4,\mfso_2\op\mfso_2)^{tw}\onto \mfU(\mfso_2\op\mfso_2)$. 
\end{prop}

\begin{proof}
The Lie algebra $\mfso_2$ is one-dimensional, that is $\mfso_2 \cong \C F^\circ_{11}$, since $F^\circ_{1,-1}=F^\circ_{-1,1}=0$ and $F^\circ_{-1,-1}=-F^\circ_{11}$. Recall the evaluation homomorphism given by Proposition 3.11 of \cite{MNO}:
\eq{
{\rm ev}^\circ : Y^+(2) \onto \mfU\mfso_2 , \qu s^\circ_{ij}(u) \mapsto  \del_{ij} + (u+ 1/2)^{-1} F^\circ_{ij} . \label{Y+2:evhom}
}

Let $\Phi$ be the isomorphism $\mfso_2\op\mfso_2 ( = \C F^\circ_{11} \op \C F^\bullet_{11} ) \iso \mfso_4^\rho ( = \C F_{11} \op \C F_{22} )$ given by $F^\circ_{11} \mapsto F_{11}+F_{22}$ and $F^\bullet_{11} \mapsto F_{22}-F_{11}$. The map $\Phi$ induces an isomorphism $\wh\Phi : \mfU\mfso_2\ot\mfU\mfso_2\iso \mfU\mfso_4^\rho = \mfU(\mfso_2\op\mfso_2)$, and so the composition $\wh\Phi({\rm ev}^\circ \ot {\rm ev}^\bullet)$ yields a surjective homomorphism $Y^+(2)\ot Y^+(2) \onto \mfU(\mfso_2\op\mfso_2)$ (here ${\rm ev}^\bullet$ denotes the evaluation homomorphism \eqref{Y+2:evhom} for the second copy of $Y^+(2)$ in the tensor product).

Finally, by composing the resulting map with the embedding $\wt\chi : X(\mfso_4,\mfso_2\op\mfso_2) \into Y^+(2) \ot Y^+(2)$ given by $S(u) \mapsto - S_1^\circ(\wt u)K_1 S^\bullet_2(\wt u) K_2$ (see Proposition 4.16 of \cite{GRW1}) we obtain precisely \eqref{so4:evhom}.
%
%
%
\end{proof}

The morphism $\rm ev$ allows us to extend $\mfso_2\oplus\mfso_2$-modules to $X(\mfso_4,\mfso_2\op\mfso_2)^{tw}$-modules. As usual, modules obtained this way are called evaluation modules. Let $V(\mu_1,\mu_2)$ denote the irreducible $\mfso^\rho_4$-module with the highest weight $(\mu_1,\mu_2)$; this is the one-dimensional representation of $\mfso_4^\rho$ in which $F_{ii}$ acts as multiplication by $\mu_i\in \C$. The corollary below follows directly from the formula \eqref{so4:evhom}.

\begin{crl}\label{C:so4}
The evaluation module $V(\mu_1,\mu_2)$ with $\mu_1,\mu_2 \in \C$ is isomorphic to the $X(\mfso_4,\mfso_2\op\mfso_2)^{tw}$-module $V(\mu(u))$ with 
\[
\mu_i(u)=g_{ii}+2g_{ii}\mu_i u^{-1}+(\mu_1^2-\mu_2^2)u^{-2} \quad \text{ for } \quad 1\leq i\leq 2.
\]
\end{crl}

The collection $\{V(\mu_1,\mu_2)\}_{\mu_1,\mu_2\in \C}$ provides a two-parameter family of one-dimensional representations of $X(\mfso_4,\mfso_2\oplus \mfso_2)^{tw}$. Note that 
the trivial representation $V(\mcG)$ may be recovered in the special case where $(\mu_1,\mu_2)=(0,0)$. In Remark \ref{R:1dim}, it will be explained that these are essentially all of the one-dimensional representations of $X(\mfso_4,\mfso_2\oplus \mfso_2)^{tw}$.


\subsection{Twisted Yangians for the symmetric pair \texorpdfstring{$(\mfso_{3},\mfso_{2})$}{}} \label{sec:so3so2}

The isomorphism between the twisted Yangian $X(\mfso_3,\mfso_{2})^{tw}$ and the Olshanskii twisted Yangian $Y^+(2)$ was established in Proposition 4.3 of \cite{GRW1}. Here we recall the necessary details of this isomorphism, which will allows us to obtain a complete description of the finite-dimensional irreducible representations of $X(\mfso_3,\mfso_2)^{tw}$ using those of~$Y^+(2)$.

Let the standard basis of $\C^2$ be given by vectors $e_{-1}$ and $e_1$ and let $V$ be the three--dimensional subspace of $\C^2\ot\C^2$ spanned by vectors $v_{-1}=e_{-1}\ot e_{-1}$, $v_{0}=\tfrac{1}{\sqrt{2}} (e_{-1}\ot e_1+e_1\ot e_{-1})$, $v_{1}=-e_1\ot e_1$. Upon identifying $V$ with $\C^3$ we may view the matrix $S(u)$ as an element of $\End V \ot X(\mfso_3,\mfso_{2})^{tw}[[u^{-1}]]$. Moreover, the operator $\tfrac{1}{2}R^\circ(-1)=\tfrac{1}{2}(I+P) \in \End(\C^2\ot\C^2)$ is a projector of $\C^2\ot\C^2$ onto the subspace $V$ and the mapping
\eqa{
& \varphi \;:\; X(\mfso_3,\mfso_{2})^{tw} \to Y^{+}(2) , \qu 
S(u) \mapsto \tfrac{1}{2} R^\circ_{12}(-1)\, S^\circ_1(2u-1)\,R^\circ_{12}(-4u+1)^{t}\, S^\circ_2(2u) K_1 \, K_2 \label{iso:so3so2} 
}
where $K=E_{11}-E_{-1,-1}$, is an isomorphism of algebras whose restriction to the subalgebra $Y(\mfso_3,\mfso_{2})^{tw}$ induces a isomorphism between $Y(\mfso_3,\mfso_{2})^{tw}$ and $SY^\pm(2)$. 
Denoting $u-1/2$ by $\wt{u}$ the map \eqref{iso:so3so2} explicitly reads as
\spl{ \label{map:so3so2}
s_{-1,-1}(u) &\mapsto s^\circ_{-1,-1}(2\wt{u})\, s^\circ_{-1,-1}(2 u) + \tfrac{1}{4 u-1}\, s^\circ_{-1,1}(2\wt{u})\, s^\circ_{1,-1}(2 u),
\\
s_{-1,0}(u) &\mapsto - \tfrac{1}{\sqrt{2}}\, s^\circ_{-1,-1}(2\wt{u})\, s^\circ_{-1,1}(2 u) - \tfrac{1}{\sqrt{2} (4 u-1)} \big(  s^\circ_{-1,1}(2\wt{u}) s^\circ_{11}(2 u) + 4u s^\circ_{-1,1}(2\wt{u})\, s^\circ_{-1,-1}(2 u) \big) ,
\\
s_{-1,1}(u) &\mapsto \tfrac{4 u }{1-4 u}\, s^\circ_{-1,1}(2\wt{u})\, s^\circ_{-1,1}(2 u),
\\
s_{0,-1}(u) &\mapsto \tfrac{1}{\sqrt{2}}\, s^\circ_{1,-1}(2\wt{u})\, s^\circ_{-1,-1}(2 u) +  \tfrac{1}{\sqrt{2}(4 u-1)} \big( s^\circ_{11}(2\wt{u})\,s^\circ_{1,-1}(2 u) + 4u\,s^\circ_{-1,-1}(2\wt{u})\, s^\circ_{1,-1}(2 u) \big),
\\
s_{00}(u) &\mapsto -\tfrac{1}{2} \big(s^\circ_{-1,1}(2\wt{u})\, s^\circ_{1,-1}(2 u) + s^\circ_{1,-1}(2\wt{u})\, s^\circ_{-1,1}(2 u) \big) 
\\
& \qu\; -\tfrac{1}{8 u-2} \big( (s^\circ_{-1,-1}(2\wt{u}) + 4u\, s^\circ_{11}(2\wt{u}))\, s^\circ_{-1,-1}(2 u) + (4u\, s^\circ_{-1,-1}(2\wt{u}) + s^\circ_{11}(2\wt{u}))\, s^\circ_{11}(2 u)\big) ,
\\
s_{01}(u) &\mapsto -\tfrac{1}{\sqrt{2} (4 u-1)} s^\circ_{-1,-1}(2\wt{u})\, s^\circ_{-1,1}(2 u) -\tfrac{1}{\sqrt{2}} s^\circ_{-1,1}(2\wt{u})\, s^\circ_{11}(2 u) -\tfrac{4u}{\sqrt{2} (4 u-1)} s^\circ_{11}(2\wt{u})\, s^\circ_{-1,1}(2 u) ,
\\
s_{1,-1}(u) &\mapsto \tfrac{4 u}{1-4 u}\, s^\circ_{1,-1}(2\wt{u})\, s^\circ_{1,-1}(2 u),
\\
s_{10}(u) &\mapsto \tfrac{1}{\sqrt{2}}\, s^\circ_{11}(2\wt{u})\, s^\circ_{1,-1}(2 u) + \tfrac{1}{\sqrt{2} (4 u-1)} \big( 4u\, s^\circ_{1,-1}(2\wt{u})\, s^\circ_{11}(2 u) + s^\circ_{1,-1}(2\wt{u})\, s^\circ_{-1,-1}(2 u)\big),
\\
s_{11}(u) &\mapsto s^\circ_{11}(2\wt{u})\, s^\circ_{11}(2 u) + \tfrac{1}{4 u-1}\, s^\circ_{1,-1}(2\wt{u})\, s^\circ_{-1,1}(2 u),
}
which can be deduced using the formulas in the proof of Proposition 5.8 of \cite{GRW2}.

\begin{prop}\label{P:so3class}
Let $\mu(u) = (\mu_0(u),\mu_1(u))$ satisfy the conditions of Proposition \ref{P:nontriv} so that the irreducible $X(\mfso_3,\mfso_2)^{tw}$-module $V(\mu(u))$ exists. Then $V(\mu(u))$ is finite-dimensional if and only if there exists a monic polynomial $P(u)$ in $u$ and a scalar $\al \in \C$ such that $P(\al)\ne 0$ and $P(u)=P(-u+3/2)$ and
\eq{
\frac{\wt\mu_0(u)}{\wt\mu_1(u)} = \frac{P(u+1/2)}{P(u)} \cdot \frac{\al - u}{\al + u - 1} . \label{so3so2:findim}
}
Moreover, when they exist, the polynomial $P(u)$ and the scalar $\al$ are uniquely determined.
\end{prop}

\begin{proof}

The proof is very similar to that of Proposition 5.8 of \cite{GRW2}. By Proposition \ref{P:nontriv} the components $\wt\mu_0(u)$ and $\wt\mu_1(u)$ are required to satisfy 
\eq{
\wt\mu_0(1/2-u) = \frac{1/2-u}{u}\cdot \frac{4u+1}{4u-3} \cdot \wt\mu_0(u) ,\qq
\wt\mu_0(u)\,\wt\mu_0(-u+1) = \wt\mu_1(u)\,\wt\mu_1(-u+1) . \label{so3so2:mu}
}
We may choose $\mu^\circ(u) \in 1 + u^{-1}\C[[u^{-1}]]$ such that
\eq{
\wt\mu_0(u) = -2 u\cdot\frac{4 u-3}{4u-1} \cdot\mu^\circ(-2\wt{u})\,\mu^\circ(2 u) , \qq \wt\mu_1(u) = 2u\cdot\mu^\circ(2\wt{u})\,\mu^\circ(2u) , \label{so3so2:mu01}
}
so that both equalities in \eqref{so3so2:mu} are satisfied. (See Lemma 5.7 of \cite{GRW2}.) Recall that $\wt\mu_0(u) = (2u-1)\,\mu_0(u) + \mu_1(u)$ and $\wt\mu_1(u)=2u\,\mu_1(u)$. Hence
\eq{
\mu_0(u) = -\frac{1}{2\wt{u}} \left( \frac{2u(4u-3)}{4u-1}\, \mu^\circ(-2\wt{u}) + \mu^\circ(2\wt{u})\right) \mu^\circ(2u) , \qq \mu_1(u) = \mu^\circ(2\wt{u})\,\mu^\circ(2u) . \label{so3so2:mu1}
}

Let $V(\mu^\circ(u))$ denote the irreducible highest weight $Y^+(2)$ module with the highest weight $\mu^\circ(u)$ and let~$\xi$ denote its highest weight vector. $V(\mu^\circ(u))$ may be viewed as a $X(\mfso_3,\mfso_2)^{tw}$-module via the isomorphism~$\varphi$. It is immediate from \eqref{map:so3so2} that $s_{ij}(u)\,\xi=0$ for all $i<j$ and 
\eq{
s_{11}(u)\,\xi = \mu^\circ(2\wt{u})\,\mu^\circ(2u)\,\xi = \mu_1(u)\,\xi. \label{so3so2:s11}
} 
To compute $s_{00}(u)\,\xi$ we need to use the following formula, which follows from the defining relations of $Y^+(2)$:
\eqn{
[s_{-1,1}^\circ(2\wt{u}),s_{1,-1}^\circ(2u)] 
& = \tfrac{1}{4 u-1} \left( s^\circ_{11}(2 u)\,s^\circ_{11}(2 \wt{u}) - s^\circ_{-1,-1}(2 \wt{u})\, s^\circ_{-1,-1}(2 u) \right) 
\\
& + \tfrac{4 u}{4 u-1} \left( s^\circ_{11}(2 u)\, s^\circ_{-1,-1}(2 \wt{u}) - s^\circ_{11}(2 \wt{u})\, s^\circ_{-1,-1}(2 u) \right).
}
Combining this formula with the equality $[s^\circ_{ii}(u),s^\circ_{jj}(v)]\xi= 0$ for all $i,j\in \{\pm 1\}$,  we obtain
\[
\varphi(s_{00}(u))\xi=  -\tfrac{1}{4 u-1} \left(4 u\, s^\circ_{-1,-1}(2 \wt{u})+s^\circ_{11}(2 \wt{u})\right)  s^\circ_{11}(2 u)\xi.
\]
%
Applying the symmetry relation \eqref{Ols:symm} to $s^\circ_{-1,-1}(2 \wt{u})$ in the equality above and using \eqref{so3so2:mu1} we deduce that
\eq{
s_{00}(u)\,\xi = -\frac{1}{2\wt{u}} \left( \frac{2u(4u-3)}{4u-1}\, \mu^\circ(-2\wt{u}) + \mu^\circ(2\wt{u})\right) \mu^\circ(2u)\,\xi = \mu_0(u)\,\xi. \label{so3so2:s00}
}
Equalities \eqref{so3so2:s11} and \eqref{so3so2:s00} show that, as an $X(\mfso_3,\mfso_2)^{tw}$-module, $V(\mu^\circ(u))$ is isomorphic to $V(\mu(u))$.

We can now use this isomorphism to determine exactly when $V(\mu(u))$ is finite-dimensional. This occurs precisely when there exists a monic polynomial $P^\circ(u)$ satisfying $P(u)=P^\circ(-u+1)$ together with a scalar $\gamma^\circ\in \C\setminus Z(P^\circ(u))$ such that \eqref{Y+2:findim} holds. Using \eqref{so3so2:mu01}, \eqref{Y+2:findim} can be written in the equivalent form 
\[
 \frac{\wt \mu_0(u)}{\wt \mu_1(u)}= -\frac{4 u-3}{4 u-1}\cdot \frac{\mu^\circ(1-2u)}{\mu^\circ(2 u-1)} = \frac{P^\circ(2 u)}{P^\circ(2 u-1)}\cdot \frac{\ga^\circ - 2u+1}{\ga^\circ +2u-1} = \frac{P(u+1/2)}{P(u)}\cdot \frac{\al - u}{\al +u-1},
\]
where $P(u)=2^{-\deg\, P^\circ(u)}P^\circ(2u-1)$ and $\al=(\ga^\circ+1)/2$. With this definition of $P(u)$ we have $P(\al)\neq 0$, $P(u)=P(-u+3/2)$, and the uniqueness of $P(u)$ is guaranteed by the uniqueness of $P^\circ(u)$. 
\end{proof}

As in the previous section, we proceed by constructing the evaluation morphism $X(\mfso_3,\mfso_2)^{tw}\onto \mathfrak{U}\mfso_2$ and the corresponding evaluation module. Recall that $d=(p-q)/4=1/4$ for the pair $(\mfso_3,\mfso_2)$. We have the following analogues of Proposition~5.8 and Corollary 5.9 of \cite{GRW2}:
 
\begin{prop} \label{P:so3so2:evhom}
%
The assignment
\eq{
s_{ij}(u)\mapsto \frac{1}{u-d} \left( g_{ij} u - \del_{ij} d + \frac{1}{u+d}\left(d\,(F_{11}^{ \rho})^2+ u\,F^{ \rho}_{ij}\right) \right) \label{so3so2:evhom}
}
defines a surjective algebra homomorphism $\mathrm{ev}:X(\mfso_3,\mfso_2)^{tw}\onto \mathfrak{U}\mfso_2$. 
\end{prop}

\begin{proof}
Recall that $\mfso_2 = \C F^\circ_{11}$. By composing the map ${\rm ev}^\circ$ in \eqref{Y+2:evhom} with the map $\varphi$ in \eqref{iso:so3so2} we find that $s_{ij}(u)\mapsto 0$ if $i\ne j$ and 
\eqn{
s_{-1,-1}(u) & \mapsto 1+\frac{1}{(u-1/4)(u+1/4)} \left(1/4(F^\circ_{-1,-1})^2 + u F^\circ_{-1,-1}\right) ,
\\
s_{00}(u) & \mapsto \frac{1+4 u}{1-4 u} -\frac{8 u}{(1-4 u)^2}\left(F^\circ_{-1,-1}+F^\circ_{11}\right) \\
& \qq - \frac{2}{(1+4u)(1-4 u)^2} \big( (F^\circ_{11} + 4 u\, F^\circ_{-1,-1})\,F^\circ_{11} + (F^\circ_{-1,-1}+4 u\,F^\circ_{11})\,F^\circ_{-1,-1} \big) \\
& = -\frac{u+1/4}{u-1/4} + \frac{1/4 (F^\circ_{11})^2}{(u-1/4)(u+1/4)} \,,
\\
s_{11}(u) & \mapsto 1 + \frac{1}{(u-1/4)(u+1/4)}  \left(1/4(F^\circ_{11})^2 + u F^\circ_{11}\right). 
}

The proof is now completed as follows. Let $\Phi : \mfso_2 \iso \mfso^{ \rho}_3$ be the isomorphism of Lie algebras given by $F^\circ_{11} \mapsto F_{11}^{  \rho} = 2F_{11}$ and let $\wh\Phi$ be the corresponding isomorphism of the enveloping algebras $\mfU\mfso_2 \iso \mfU\mfso^{ \rho}_3$. Since $\mcG={\rm diag}(1,-1,1)$ and $d=1/4$, it is straightforward to see that $\wh\Phi\circ {\rm ev}^\circ \circ \varphi$ yields \eqref{so3so2:evhom} as required. 
\end{proof}

Let $V(\mu)$ denote the irreducible $\mfso_3^\rho$-module with the highest weight $\mu\in \C$; this is the one-dimensional module in which $F_{11}$ acts as multiplication by $\mu$. Formula \eqref{so3so2:evhom} implies the following:
\begin{crl}\label{C:so3}
The evaluation module $V(\mu)$ with $\mu\in \C$ is isomorphic to the $X(\mfso_3,\mfso_2)^{tw}$-module $V(\mu(u))$ with 
\[
 \mu_0(u)=-1+2d\left(\frac{2\mu^2-(u+d)}{u^2-d^2} \right), \qu \mu_1(u)=1+2\mu\left(\frac{u+2d\mu}{u^2-d^2} \right) \qu\text{and}\qu d = \frac14.
\]
\end{crl}

In analogy to the two-parameter family $\{V(\mu_1,\mu_2)\}_{\mu_1,\mu_2\in \C}$ of the previous subsection, the one-parameter family $\{V(\mu)\}_{\mu\in \C}$, which includes $V(\mcG)=V(0)$, contains all of the one-dimensional representations of $X(\mfso_3,\mfso_2)^{tw}$ up to twisting by automorphisms of the form \eqref{nu_g}: see Remark \ref{R:1dim}.


\section{Necessary conditions in the general setting}\label{sec:Nec}

Assume that $N\geq 4$ if $\mfg_N=\mfsp_N$,  $N\geq 5$ if $\mfg_N=\mfso_N$ and that the pair $(\mfg_N,\mfg_{p}\oplus \mfg_q)$ is of type DI(a), BI or CII. In this section we use the machinery developed  in \cite{GRW2} to obtain a set of conditions on the highest weight $\mu(u)$ of the $X(\mfg_N,\mfg_{p}\oplus \mfg_q)^{tw}$-module $V(\mu(u))$ which are satisfied whenever this module is finite-dimensional: see Propositions \ref{P:necessary} and \ref{P:int}.

In this section it will often be convenient to write the symmetric pair $(\mfg_N,\mfg_{p}\oplus \mfg_q)$ as 
$(\mfg_N,\mfg_{N-2\ell}\oplus \mfg_{2\ell})$ (see \eqref{ley-key}). This will be of particular importance in Proposition \ref{P:ind}.

Let $m$ be an integer satisfying $1\leq m\leq \ley-\delta_{\ley,n}$. In what follows we will consider the quantum algebras $ X(\mfg_{N},\mfg_{N-2\ley}\oplus \mfg_{2\ley})^{tw}$ and $X(\mfg_{N-2m},\mfg_{N-2\ley}\oplus \mfg_{2(\ley-m)})^{tw}$ simultaneously, and for this reason we will
relabel the generating series $s_{ij}(u)$ of $ X(\mfg_{N-2m},\mfg_{N-2\ley}\oplus \mfg_{2(\ley-m)})^{tw}$ by $s^{m}_{ij}(u)$ for each $i,j\in \mcI_{N-2m}$. Similarly, the generating series $\wt s_{ij}(u)$ of $\wt X(\mfg_{N-2m},\mfg_{N-2\ley}\oplus \mfg_{2(\ley-m)})^{tw}$ shall be denoted $\wt s_{ij}^{m}(u)$. It will also be convenient to employ the notation
$\mcG_m$ and $p_m(u)$ for the matrix $\mcG$ and rational function $p(u)=p_{\mcG_m}(u)$ defined in \eqref{p(u)}, respectively, which correspond to the symmetric pair $(\mfg_{N-2m},\mfg_{N-2\ley}\oplus \mfg_{2(\ley-m)})$. In particular, $X(\mfg_{N-2m},\mcG_m)^{tw}$ is equal to $X(\mfg_{N-2m},\mfg_{N-2\ley}\oplus \mfg_{2(\ley-m)})^{tw}$ and we will 
sometimes use to former notation for brevity.

For each $i,j\in \mcI_{N-2m}$ define $s_{ij}^{\circ m}(u)\in g_{ij}+u^{-1}X(\mfg_N,\mfg_{N-2\ley}\oplus \mfg_{2\ley})^{tw}[[u^{-1}]]$ by
\begin{equation}
 s_{ij}^{\circ m}(u)=s_{ij}(u+\tfrac{m}{2})+\tfrac{\delta_{ij}}{2u}\sum_{a=n-m+1}^n s_{aa}(u+\tfrac{m}{2}). \label{scirc}
\end{equation}
Given an arbitrary representation $V$ of $X(\mfg_{N},\mfg_{N-2\ley}\oplus \mfg_{2\ley})^{tw}$, let $V_{(+,m)}\subset V$ be the subspace
\begin{equation}
 V_{(+,m)}=\{v\in V\,:\, s_{ij}(u)v=0\; \text{ for all }\; i<j \text{ with }n-m+1\leq j\leq n\}. \label{V(+,m)}
\end{equation}
Lemma 4.12 of \cite{GRW2} together with a simple induction on $m$ shows that each $s_{ij}^{\circ m}(u)$ may be regarded as an element of 
$\End \,V_{(+,m)}[[u^{-1}]]$ and that for any $h_m(u)\in (-1)^{\delta_{m,\ley}\delta_{p/2-m<\key}}+u^{-1}\C[[u^{-1}]]$ the assignment 
\begin{equation}
\wt s_{ij}^{m}(u)\mapsto h_m(u)s_{ij}^{\circ m}(u) \; \text{ for all }\; i,j\in \mcI_{N-2m} \label{wtX-act}
\end{equation}
gives rise to an $\wt X(\mfg_{N-2m},\mfg_{N-2\ley}\oplus \mfg_{2(\ley-m)})^{tw}$-module 
structure on $V_{(+,m)}$.  When $m=1$, Proposition 4.13 of \cite{GRW2} gave a sufficient condition on 
$h_m(u)$ which guarantees that the action of the reflection algebra $\wt X(\mfg_{N-2m},\mcG_m)^{tw}$ on $V_{(+,m)}$ factors through the extended twisted Yangian $X(\mfg_{N-2m},\mcG_m)^{tw}$. In fact, a careful reading of the proof shows that, provided $V_{(+,m)}\neq 0$, this condition is also necessary. 

The next proposition generalizes this result to the case where $1\leq m\leq \ley-\delta_{\ley,n}$. 

\begin{prop}\label{P:ind}
Suppose that $V_{(+,m)}$ is nonzero, and let $h_m(u)\in (-1)^{\delta_{m,\ley}\delta_{p/2-m<\key}}+u^{-1}\C[[u^{-1}]]$. Then the assignment $s_{ij}^m(u)\mapsto h_m(u)s_{ij}^{\circ m}(u)$ equips $V_{(+,m)}$ with a $X(\mfg_{N-2m},\mfg_{N-2\ley}\oplus \mfg_{2(\ley-m)})^{tw}$-module structure if and only if $h_m(u)$ satisfies the relation
 \begin{equation}
  h_m(u)h_m(\ka-m-u)^{-1}=p_m(u)p(u+\tfrac{m}{2})^{-1}.
 \label{h:cond}
 \end{equation}
If in addition $V$ is a highest weight module with the highest weight vector $\xi$, then $V_{(+,m)}\neq 0$ and the subrepresentation 
$X(\mfg_{N-2m},\mfg_{N-2\ley}\oplus \mfg_{2(\ley-m)})^{tw}\xi \subset V_{(+,m)}$ is a highest weight module with the highest weight $h_m(u)\mu^{\circ m}(u)=(h_m(u)\mu_i^{\circ m}(u))_{i\in \mcI_{N-2m}^+}$, where the components of  $\mu^{\circ m}(u)$ are uniquely determined by the relations 
\begin{equation}
 \wt {\mu_i^{\circ m}}(u)=\wt \mu_i(u+\tfrac{m}{2})\; \text{ for all }\; i\in \mcI_{N-2m}^+. \label{mu-circ}
\end{equation}
\end{prop}

\begin{proof}
That the reflection equation for $X(\mfg_{N-2m},\mcG_m)^{tw}$ is satisfied by $h_m(u)s_{ij}^{\circ m}(u)$ is a consequence of Lemma 4.12 of \cite{GRW2}. As for the symmetry relation for $X(\mfg_{N-2m},\mcG_m)^{tw}$, one can repeat the proof of Proposition 4.13 in \textit{loc.~cit}, replacing $\kappa'$ by $\kappa'-m$ and $n-1$ by $n-m$.
\end{proof}

Given a $X(\mfg_N,\mcG)^{tw}$-module $V$ and  $h_m(u)\in (-1)^{\delta_{m,\ley}\delta_{p/2-m<\key}}+u^{-1}\C[[u^{-1}]]$ satisfying \eqref{h:cond}, we denote by $V_{(h_m(u))}$ the $X(\mfg_{N-2m},\mfg_{N-2\ley}\oplus \mfg_{2(\ley-m)})^{tw}$-module which is equal to $V_{(+,m)}$ as a vector space with action determined by $s_{ij}^m(u)\mapsto h_m(u)s_{ij}^{\circ m}(u)$ for all $i,j\in \mcI_{N-2m}$. When $V$ is a highest weight module with $\xi\in V$ a fixed highest weight vector, we denote by $V_{h_m(u)}$ the submodule 
$X(\mfg_{N-2m},\mcG_m)^{tw}\xi$ of $V_{(h_m(u))}$. 

In the remarks following the proof of Proposition 4.13 of \cite{GRW2} it was explained that there always exists $h_m(u)$ satisfying \eqref{h:cond}. We now show that there is a particular series $h_m(u)$ which can be regarded as the most natural solution of \eqref{h:cond}.  Let $\mathscr{g}_m(u)$ be the rational function from \eqref{g(u)} associated to the pair $(\mfg_{N-2m},\mfg_{N-2\ley}\oplus \mfg_{2(\ley-m)})$.  
\begin{prop}\label{P:nat-h}
For any integer $1\leq m\leq \ley-\delta_{\ley,n}$, the series 
\begin{equation}
 h_m(u)=\frac{u}{u+\tfrac{m}{2}}\cdot \mathscr{g}_m(u) \mathscr{g}(u+\tfrac{m}{2})^{-1} \label{h_m:triv}
\end{equation}
belongs to $(-1)^{\delta_{m,\ley}\delta_{p/2-m<\key}}+u^{-1}\C[[u^{-1}]]$ and satisfies \eqref{h:cond}. Moreover, this is the unique choice of $h_m(u)$  with the property that $V(\mcG)_{h_m(u)}$ is isomorphic to the trivial representation $V(\mcG_m)$. 
\end{prop}
\begin{proof}
 Let $h_m(u)$ be given by \eqref{h_m:triv} and consider the trivial representation $V(\mcG)$ of $X(\mfg_N,\mcG)^{tw}$. Since $V(\mcG)_{(+,m)}=V(\mcG)$, 
 $V(\mcG)$ is also a $\wt X(\mfg_{N-2m},\mcG_m)^{tw}$-module with action given by \eqref{wtX-act}. Let's argue that this is just the representation 
 $\wt S^m(u)\mapsto \mcG_m(u)$, where $\mcG_m(u)$ is the matrix $\mcG(u)$ corresponding to the pair $(\mfg_{N-2m},\mfg_{N-2\ley}\oplus \mfg_{2(\ley-m)})$. The generating series
 $\wt s_{ii}^m(u)$ operates in $V(\mcG)$ as $h_m(u)\mu^\circ(u)$, where in this case $\mu^\circ(u)=(\mu^{\circ m}_i(u))_{i\in \mcI_{N-2m}^+}$ has components determined by 
 $\wt{\mu_i^{\circ m}}(u)=\wt g_{ii}(u+\tfrac{m}{2})$ and $g_{ii}(u)$ is the $(i,i)$-th entry of the diagonal matrix $\mcG(u)$. Let $g^m_{kk}(u)$ be the $(k,k)$-th entry 
 of the (diagonal) matrix $\mcG_m(u)$. Then, to show that $\wt s_{ii}^m(u)$ operates as $g^m_{ii}(u)$ for all $i\in \mcI_{N-2m}^+$, it is enough to show 
\begin{equation}
 h_m(u)\wt g_{ii}(u+\tfrac{m}{2})=\wt{g^m_{ii}}(u). \label{nat-h.1}
\end{equation}
For any $i\in \mcI_N^+$ satisfying $0\leq i\leq \key$, we have $h_m(u)=\wt{g_{ii}^m}(u)\cdot \wt g_{ii}(u+\tfrac{m}{2})^{-1}$ by definition of $\mathscr{g}(u)$ and $\mathscr{g}_m(u)$ - see \eqref{eq:wtgii2}. Hence 
 \eqref{nat-h.1} trivially holds for these values of $i$. If instead  $\key+1\leq i\leq n-m$, then the left-hand side of 
 \eqref{nat-h.1} is equal to 
 \begin{equation*}
  2u\cdot \left(\frac{p+q-2m-4u}{p-q[\pm]2m-4u}\right)\left(\frac{p-q-2m-4u}{p+q-2m-4u}\right)\frac{p-q[\pm](4u+2m)}{p-q-2m-4u}=2u\left(\frac{p-q[\pm]2m[\pm]4u}{p-q[\pm]2m-4u} \right),
 \end{equation*}
 which  is precisely $\wt g_{ii}^m(u)$ by \eqref{eq:wtgii1}. Therefore \eqref{nat-h.1} holds and $\wt s_{ii}^m(u)$ operates as $g_{ii}^m(u)$ for all $i\in \mcI_{N-2m}^+$. Since 
the action of the generating matrix $\wt S^m(u)$ on $V(\mcG)$ is diagonal and symmetric with respect to the transposition $t$, and $\mcG_m(u)$ has these same properties, 
we can conclude that $\wt S^m(u)$ operates as $\mcG_m(u)$ in $V(\mcG)=V(\mcG)_{(+,m)}$. 

We already know that this action factors through the extended twisted Yangian $X(\mfg_{N-2m},\mcG_m)^{tw}$ (yielding the trivial representation $V(\mcG_m)$), and therefore Proposition \ref{P:ind} implies that $h_m(u)$ satisfies the relation \eqref{h:cond}. 

To complete the proof of the proposition it remains to explain why \eqref{h_m:triv} is the unique choice of $h_m(u)$ such that $V(\mcG)_{h_m(u)}\cong V(\mcG_m)$. This follows from the fact that any such $h_m(u)$ must satisfy \eqref{nat-h.1}, and this property uniquely determines $h_m(u)$. 
\end{proof}
Henceforth we will assume $h_m(u)$ is given by \eqref{h_m:triv} and we will write $V_{(m)}$ for $V_{(h_m(u))}$ and $V_m$ for $V_{h_m(u)}$. 
\begin{crl}\label{C:nosub}
 Suppose that $\mu(u)$ satisfies the conditions of Proposition \ref{P:nontriv} and let $\xi \in V(\mu(u))$ be a highest weight vector. Then $V(\mu(u))_m$ is the only highest weight submodule of $V(\mu(u))_{(m)}$. In particular, if $V(\mu(u))$ is finite-dimensional then $V(\mu(u))_m$ is a finite-dimensional 
 irreducible module isomorphic to $V(h_m(u)\mu^{\circ  m}(u))$, where $h_m(u)$ is given by \eqref{h_m:triv} and $\mu^{\circ m}(u)$ is as in Proposition \ref{P:ind}. 
\end{crl}
\begin{proof}
 Suppose that $K\subset V(\mu(u))_{(m)}$ is any highest weight submodule of $V(\mu(u))_{(m)}$, and let $\eta\in K$ be a highest weight vector. 
 Since $V(\mu(u))_{(m)}$ is equal to $V(\mu(u))_{(+,m)}$ as a vector space, $s_{ij}(u)\eta=0=s_{-j,-i}(u)\eta$ for all $i<j$ with $n-m+1\leq j\leq n$. Since $\eta$ is a highest weight vector of 
 $K$, we must have $s_{kl}(u+\tfrac{m}{2})\eta=0=s_{-l,-k}(u+\tfrac{m}{2})\eta$ for all $-n+m\leq k<l\leq n-m$, and therefore $s_{kl}(u)\eta=0=s_{-l,-k}(u)\eta$ for these values of $k$ and $l$. Combining these two facts gives $s_{ij}(u)\eta=0$ for all $-n\leq i<j \leq n$, and thus $\eta \in V(\mu(u))^0$. By Corollary \ref{C:hw-1dim},  $V(\mu(u))^0=\C\xi$ and therefore $\eta$ is a nonzero scalar multiple of $\xi$. As $\xi$ generates $V(\mu(u))_m$, this implies $K=V(\mu(u))_m$.
 
 The second part of the corollary follows from the first part and the fact that, if $V(\mu(u))_m$ is finite-dimensional, every 
 proper nonzero submodule $K\subset V(\mu(u))_m$ must contain an irreducible submodule, and hence a  highest weight vector.
\end{proof}

The following proposition gives necessary conditions for the finite-dimensionality of $V(\mu(u))$. 
Recall that $\delta=0$ if $\mfg_N=\mfso_N$ and $\delta=1$ if $\mfg_N=\mfsp_{N}$.  
\begin{prop}\label{P:necessary}
Suppose  that the irreducible $X(\mfg_N,\mfg_{N-2\ley}\oplus \mfg_{2\ley})^{tw}$-module $V(\mu(u))$ is finite-dimensional.  Then there exists monic polynomials $P_1(u),\ldots,P_n(u)$ together with a scalar $\al\in \C\setminus Z(P_{\key+1}(u))$  such that 
\begin{equation}
 \frac{\wt \mu_{i-1}(u)}{\wt \mu_i(u)}=\frac{P_i(u+1)}{P_i(u)}\left(\frac{\al-u}{\al+u-\ley}\right)^{\del_{i,\key+1}} \; \text{ with }\; P_i(u)=P_i(-u+n-i+2)  \label{nec.0}
\end{equation}
for all $2\leq i\leq n$, while $P_1(u)$ satisfies  $P_1(u)=P_1(-u+\ka+2^\delta)$ and the relation
\begin{equation}
\begin{aligned}\label{nec.2}
 \frac{\wt \mu_0(u)}{\wt \mu_{1}(u)}&=\frac{P_1(u+\tfrac{1}{2})}{P_1(u)}\left(\frac{\alpha-u}{\alpha+u-n}\right)^{\del_{\key,0}} &&\quad \text{ if }\;N={2n+1},\\
 \frac{\wt \mu_1(\ka-u)}{\wt \mu_{2^{1-\delta}}(u)}&=\frac{\mathscr{g}(\ka-u)}{\mathscr{g}(u)}\cdot \frac{P_1(u+2^\delta)}{P_1(u)}\cdot \frac{\kappa-u}{u}\\
 &=\frac{(2u-\ka\pm 1)(2u+q-\ka\mp 1)}{(2u-\ka\mp 1)(2u-q-\ka\pm 1)}\cdot \frac{P_1(u+2^\delta)}{P_1(u)}\cdot \frac{\kappa-u}{u} &&\quad \text{ if }\; N=2n. 
\end{aligned}
\end{equation}
\end{prop}
\begin{proof}
The existence of $P_2(u),\ldots,P_n(u)$ together with the scalar $\alpha\in \C\setminus Z(P_{\key+1}(u))$ (provided $\key\neq 0$) satisfying \eqref{nec.0} was established in Proposition 4.18 of \cite{GRW2}. Therefore, it suffices to show that there exists $P_1(u)$ (together with $\alpha\in \C\setminus Z(P_1(u))$ if $\key=0$) satisfying $P_1(u)=P_1(-u+\ka+2^\delta)$ as well as the relation \eqref{nec.2}. 

Assume first that $\key\neq 0$ so that $(\mfg_N,\mfg_N^\rho)\neq(\mfso_{2n+1},\mfso_{2n})$. As a consequence of Propositions \ref{P:ind} and \ref{P:nat-h} we may consider the $X(\mfg_{N-2\ell},\mfg_{N-2\ell})^{tw}$-module $V(\mu(u))_{\ell}$. In fact, by Corollary \ref{C:nosub} this module is irreducible and isomorphic to 
$V(h_\ley(u)\mu^{\circ \ley}(u))$ with $h_\ley(u)$ given by \eqref{h_m:triv} with $m=\ley$. By Theorem \ref{T:BCD0-class} and the definition of $\mu^{\circ \ley}(u)$ (see 
\eqref{mu-circ}), there exists a monic polynomial $Q_1(u)$ such $Q_1(u)=Q_1(-u+\ka-\ley+2^\delta)$ and 
\begin{align}
\frac{\wt \mu_0(u+\frac{\ley}{2})}{\wt \mu_1(u+\frac{\ley}{2})}&=\frac{Q_1(u+\tfrac{1}{2})}{Q_1(u)} \; \text{ if }\; N=2n+1, \label{Nec:eq1}\\
 \frac{\wt \mu_1(\ka-\frac{\ley}{2}-u)}{\wt \mu_{2^{1-\delta}}(u+\frac{\ley}{2})}&=\frac{h_\ley(u)}{h_\ley(\ka-\ley-u)}\cdot \frac{Q_1(u+2^\delta)}{Q_1(u)}\cdot \frac{\ka-\ley-u}{u}\; \text{ if }\; N=2n. \label{Nec:eq2}
\end{align}
If $N=2n+1$, then setting $P_1(u)=Q_1(u-\frac{\ley}{2})$ we obtain the desired result from \eqref{Nec:eq1} above after shifting $u\mapsto u-\frac{\ley}{2}$. Suppose instead that $N=2n$. Since $m=\ley$, we have $\mathscr{g}_m(u)=1$ and, after substituting $u\mapsto u-\frac{\ley}{2}$ in \eqref{Nec:eq2}, we arrive at the relation
\begin{equation*}
 \frac{\wt \mu_1(\ka-u)}{\wt \mu_{2^{1-\delta}}(u)}=\frac{u-\frac{\ley}{2}}{u}\cdot \frac{\ka-u}{\ka-\frac{\ley}{2}-u}\cdot \frac{\mathscr{g}(\ka-u)}{\mathscr{g}(u)}\cdot \frac{P_1(u+2^\delta)}{P_1(u)}\cdot \frac{\ka-\frac{\ley}{2}-u}{u-\frac{\ley}{2}}=\frac{\mathscr{g}(\ka-u)}{\mathscr{g}(u)}\cdot \frac{P_1(u+2^\delta)}{P_1(u)}\cdot \frac{\ka-u}{u},
\end{equation*}
where $P_1(u)=Q_1(u-\frac{\ley}{2})$. Since
\begin{equation}
 \frac{\mathscr{g}(\ka-u)}{\mathscr{g}(u)}=\frac{(2u-\ka\pm 1)(2u+q-\ka\mp 1)}{(2u-\ka\mp 1)(2u-q-\ka\pm 1)}, \label{g(u):exp}
\end{equation}
we have established the existence  of $P_1(u)$ satisfying \eqref{nec.2} and $P_1(u)=P_1(-u+\ka+2^\delta)$.

 Assume now that $\key=0$, and consider the $X(\mfso_3,\mfso_2)^{tw}$-module $V(\mu(u))_{\ley-1}$ which is isomorphic to $V(h_{\ley-1}(u)\mu^{\circ (\ley-1)}(u))$. By Proposition \ref{P:so3class}, there exists a monic polynomial $Q_1(u)$ together with $\alpha^\circ \in \C\setminus Z(Q_1(u))$ such that $Q_1(u)=Q_1(-u+\frac{3}{2})$ and 
\begin{equation*}
 \frac{\wt \mu_0(u+\frac{\ley-1}{2})}{\wt \mu_1(u+\frac{\ley-1}{2})}=\frac{Q_1(u+\tfrac{1}{2})}{Q_1(u)}\cdot \frac{\alpha^\circ-u}{\alpha^\circ+u-1}.
\end{equation*}
After setting $P_1(u)=Q_1(u-\frac{\ley-1}{2})$, $\alpha=\alpha^\circ+\frac{\ley-1}{2}$ and substituting $u\mapsto u-\frac{\ley-1}{2}$, we obtain the first equality in \eqref{nec.2}. 
Since we also have $P_1(u)=P_1(-u+\ka+2^\delta)$ and $\alpha\in \C\setminus Z(P_1(u))$, this completes the proof of the existence of $(\alpha,P_1(u),\ldots,P_n(u))$. 
\end{proof}
The previous proposition indicates that finite-dimensional irreducible $X(\mfg_N,\mcG)^{tw}$-modules are intimately connected to tuples $(\alpha,P_1(u),\ldots,P_n(u))$ 
which satisfy certain relations. In light of this, it is desirable to have terminology which enables us to quickly translate between the language of highest weights and that of finite sequences $(\alpha,P_1(u),\ldots,P_n(u))$ of the form described in Proposition \ref{P:necessary}. 
 \begin{defn}\label{D:assoc}
  Suppose that $P_1(u),\ldots,P_n(u)$ are monic polynomials in $u$ such that 
  \begin{equation}
   P_1(u)=P_1(-u+\ka+2^\delta) \quad \text{ and }\quad P_i(u)=P_i(-u+n-i+2) \; \text{ for all }\; 2\leq i\leq n. \label{P-sym}
  \end{equation}
Then, given $\alpha \in \C\setminus Z(P_{\key+1}(u))$, we say that $\mu(u)$ is associated to the tuple $(\alpha,P_1(u),\ldots,P_n(u))$ 
if the relations \eqref{nec.0} and \eqref{nec.2} of Proposition \ref{P:necessary} are satisfied and additionally \eqref{nontriv.2} holds when $N=2n+1$.
 \end{defn}
Note that it also follows from the relations \eqref{nec.0}, \eqref{nec.2} and \eqref{P-sym} that \eqref{nontriv.1} holds, and hence that if $\mu(u)$ 
is associated to a tuple $(\alpha,P_1(u),\ldots,P_n(u))$ then the irreducible module $V(\mu(u))$ exists.

 In the case where $V(\mu(u))$ is also finite-dimensional,  we follow the convention in the literature and call $(\alpha,P_1(u),\ldots,P_n(u))$ the \textit{Drinfeld tuple} associated to $V(\mu(u))$ and the polynomials 
 $P_1(u),\ldots,P_n(u)$ the \textit{Drinfeld polynomials}. 
 \begin{lemma}\label{L:unique}
  Suppose that $\mu(u)$ is associated to $(\alpha,P_1(u),\ldots,P_n(u))$. Then this is the unique tuple associated to $\mu(u)$. Moreover, 
  if $\mu^\sharp(u)$ is also associated to $(\alpha,P_1(u),\ldots,P_n(u))$ then there exists $g(u)\in 1+u^{-2}\C[[u^{-2}]]$ such that $V(\mu^\sharp(u))\cong V(\mu(u))^{\nu_g}$.
 \end{lemma}
\begin{proof}
The uniqueness of $(\alpha,P_1(u),\ldots,P_n(u))$ is an immediate consequence of Lemma \ref{L:poly1}. Let's turn to the second statement of the lemma. Suppose that $\mu^\sharp(u)$ is also associated to $(\al,P_1(u),\ldots,P_n(u))$. We need to show there is $g(u)\in 1+u^{-2}\C[[u^{-2}]]$ such that $\mu_i^\sharp(u)=g(u-\ka/2)\mu_i(u)$ for all 
$i\in \mcI_{N}^+$. If this is true, then $V(\mu^\sharp(u))$ and $V(\mu(u))^{\nu_g}$ will have the same highest weight and hence be isomorphic. The proof that there exists $g(u)$ satisfying $\mu_i^\sharp(u)=g(u-\ka/2)\mu_i(u)$ for all $i\in \mcI_{N}^+$ follows from the same type of arguments as given in the $(\Longleftarrow)$ direction of the proof of Theorem 6.2 in \cite{GRW2}: see (6.9) therein.  For the sake of completeness we recall some of the details here. 

From \eqref{nec.0} and \eqref{nec.2} we obtain the equalities 
\begin{equation}
 \frac{\wt \mu_{i-1}(u)}{\wt \mu_i(u)}=\frac{\wt \mu_{i-1}^{\sharp}(u)}{\wt \mu_i^{\sharp}(u)} \; \text{ for all }\; 2\leq i\leq n \quad \text{ and }\quad \frac{\wt \mu_a(\ka-u)}{\wt \mu_b(u)}=\frac{\wt \mu_a^\sharp(\ka-u)}{\wt \mu_b^\sharp(u)}, \label{P-un}
\end{equation}
where $(a,b)=(1,1)$ if $\mfg_N = \mfsp_N$, $(a,b)=(1,2)$ if $\mfg_N=\mfso_{2n}$ and $(a,b)=(0,1)$ if $\mfg_N=\mfso_{2n+1}$. In the $\mfso_{2n+1}$ case this follows 
from \eqref{nec.2} together with the relation \eqref{nontriv.2}. Setting $g(u)=\mu_n^\sharp(u+\tfrac{\ka}{2})(\mu_n(u+\tfrac{\ka}{2}))^{-1}$, we obtain from the first 
relations in \eqref{P-un} that $\mu_i^\sharp(u)=g(u-\ka/2)\mu_i(u)$ for all $i\in \mcI_N^+$. The second relation in \eqref{P-un} then yields that $g(u)=g(-u)$, which completes the proof.
\end{proof}
The proof of Proposition \ref{P:necessary} provides a relation between the tuples associated to $\mu(u)$ and to the highest weight of $V(\mu(u))_{\ley}$ (if $\key\neq 0$) or $V(\mu(u))_{\ley-1}$ (if $\key=0$) under the assumption that $V(\mu(u))$ is finite-dimensional. The following corollary generalizes this result. 
\begin{crl}\label{C:poly-low}
Suppose that $\mu(u)$ is associated to $(\al,P_1(u),\ldots,P_n(u))$, and let $1\leq m \leq \ley-\delta_{\ley,n}$. Then the highest weight of the 
 $X(\mfg_{N-2m},\mfg_{N-2\ley}\oplus \mfg_{2(\ley-m)})^{tw}$-module $V(\mu(u))_m$ is associated to the tuple 
\[
(\al-\tfrac{m}{2},P_1(u+\tfrac{m}{2}),\ldots,P_{n-m}(u+\tfrac{m}{2})).
\]
\end{crl}

\begin{proof}
By Proposition \ref{P:ind}, $V(\mu(u))_{m}$ has highest weight $\mu^\sharp(u)=h_m(u)\mu^{\circ m}(u)$ with $h_m(u)$ given by \eqref{h_m:triv} and $\mu^\circ(u)=(\mu_i^{\circ m}(u))_{i\in \mcI_{N-2m}^+}$ determined by \eqref{mu-circ}. For each  $2\leq i\leq n-m$, \eqref{nec.0} implies the relation
\begin{equation*}
  \frac{\wt \mu_{i-1}^\sharp(u)}{\wt \mu_{i}^\sharp(u)}=\frac{\wt \mu_{i-1}(u+\tfrac{m}{2})}{\wt \mu_{i}(u+\tfrac{m}{2})}=\frac{P_i(u+\tfrac{m}{2}+1)}{P_i(u+\tfrac{m}{2})}\left(\frac{\alpha-\tfrac{m}{2}-u}{\alpha+\tfrac{m}{2}+u-\ley} \right)^{\del_{i,\key+1}} =\frac{Q_i(u+1)}{Q_i(u)}\left(\frac{\al^\sharp-u}{\al^\sharp+u-(\ley-m)} \right)^{\del_{i,\key+1}} \!,
 \end{equation*}
where $Q_i(u)=P_i(u+\tfrac{m}{2})$ and $\al^\sharp=\al-\tfrac{m}{2}$. It also follows from \eqref{nec.0} that $Q_i(u)=Q_i(-u+(n-m)-i+2)$ and $\al^\sharp\in \C\setminus Z(Q_{\key+1}(u))$. If $N=2n+1$, similar observations imply that the first relation in \eqref{nec.2} is satisfied with $(\wt \mu_0(u),\wt\mu_1(u))$ replaced by 
$(\wt \mu_0^\sharp(u),\wt \mu_1^\sharp(u))$, $P_1(u)$ replaced by $Q_1(u)=P_1(u+\tfrac{m}{2})$ and $\alpha$ by $\alpha^\sharp=\alpha-\tfrac{m}{2}$. 

It remains to show \eqref{nec.2} holds when $N=2n$ and  $(\wt \mu_1(u),\wt\mu_{2^{1-\delta}}(u),\mathscr{g}(u),P_1(u),\ka)$ is replaced by 
\begin{equation*}
(\wt \mu_1^\sharp(u),\wt\mu_{2^{1-\delta}}^\sharp(u),\mathscr{g}_m(u),P_1(u+\tfrac{m}{2}),\ka-m).
\end{equation*}
By definition of $\mu^\sharp(u)$ and of  $h_m(u)$, we have 
\begin{align*}
\frac{\wt \mu_1^\sharp(\ka-m-u)}{\wt \mu^\sharp_{2^{1-\delta}}(u)}&=\frac{h_m(\ka-m-u)}{h_m(u)}\cdot \frac{\mathscr{g}(\ka-\tfrac{m}{2}-u)}{\mathscr{g}(u+\tfrac{m}{2})}\cdot \frac{P_1(u+2^\delta+\tfrac{m}{2})}{P_1(u+\tfrac{m}{2})}\cdot \frac{\ka-\tfrac{m}{2}-u}{u+\tfrac{m}{2}}\\
                                                                  &=\frac{\mathscr{g}_m(\ka-m-u)}{\mathscr{g}_m(u)}\cdot \frac{P_1(u+2^\delta+\tfrac{m}{2})}{P_1(u+\tfrac{m}{2})}\cdot \frac{\ka-m-u}{u}. \qedhere
\end{align*}
\end{proof}

Our aim for the rest of this section is to show that if $(\alpha,P_1(u),\ldots,P_n(u))$ is a Drinfeld tuple then $\alpha-N/4$ must take half integer values unless $\mfg_N=\mfso_N$ and $q=2$: this will be made precise in Proposition \ref{P:int}. First, we proceed with two lemmas. 

\begin{lemma}\label{L:QxP=QP}
Suppose that  $L(\lambda(u))$ is finite-dimensional with Drinfeld polynomials $Q_1(u),\ldots,Q_n(u)$ and that $\mu(u)$ is associated to a tuple $(\alpha,P_1(u),\ldots,P_n(u))$. Assume also that $\alpha$ is not a root of $Q_{\key+1}(u-\ka/2)$ or $Q_{\key+1}(-u+\ley+2-\ka/2)$ and  let $\xi\in L(\lambda(u))$ and $\eta\in V(\mu(u))$ be highest weight vectors. Then the highest weight $\gamma(u)$ of the $X(\mfg_N,\mcG)^{tw}$-module $X(\mfg_N,\mcG)^{tw}(\xi\otimes \eta)\subset L(\lambda(u))\otimes V(\mu(u))$ is associated to the tuple $(\alpha,(Q_1\odot P_1)(u),\ldots,(Q_n\odot P_n)(u))$, where
\[
  (Q_i\odot P_i)(u)=\begin{cases}
                     (-1)^{\deg\, Q_i(u)}\,Q_i(u-\ka/2)\,Q_i(-u+n-i+2-\ka/2)\,P_i(u) & \text{if }\; i\ne1,\\
                     (-1)^{\deg\, Q_1(u)}\,Q_1(u-\ka/2)\,Q_1(-u+\ka/2+2^\del)\,P_i(u) & \text{if }\; i=1.
                    \end{cases}
\]
\end{lemma}

\begin{proof}
Apply the formulas \eqref{HWT:tensors} in conjunction with the relations of Proposition \ref{P:X-nontriv}: see the argument following (6.8) in the proof of Theorem 6.2 from \cite{GRW2}. 
\end{proof}

\begin{rmk}\label{R:QxP=QP}
 If $\alpha$ is a root of $Q_{\key+1}(u-\ka/2)$ or $Q_{\key+1}(-u+\ley+2-\ka/2)$ then $(\alpha,(Q_1\odot P_1)(u),\ldots,(Q_n\odot P_n)(u))$ will still satisfy the relations of Proposition \ref{P:necessary}, except that the condition 
 $\alpha\notin Z((Q_{\key+1}\odot P_{\key+1})(u))$ will fail to hold. In this case one must replace $(\alpha,(Q_{\key+1}\odot P_{\key+1})(u))$ by 
 $(\alpha-m\,\ell_\al^m ,(Q_1\odot P_1)_\al^m(u))$ where $m=\tfrac{1}{2}$ if $\key=0$ and $m=1$ otherwise: see Lemma \ref{L:poly2}. 
\end{rmk}

To state the next lemma we will need the following terminology: a $\mfg_N^\rho$-module $V$ is said to be a highest weight module with the highest weight $(\mu_i)_{i=1}^n$ if it is generated by a nonzero vector $\xi$ such that $F_{ij}^\rho\xi=0$  for all $i<j\in \mcI_N$ and $F_{ii}\xi=\mu_i\xi$ for all $1\leq i\leq n$ (see \eqref{grho->X}).
\begin{lemma}\label{L:gtw}
Suppose that $\mu(u)$ is associated to $(\alpha,P_1(u),\ldots,P_n(u))$, and let $\xi$ be a highest weight vector of $V(\mu(u))$. Then the $\mfg_N^\rho$-module $\mfU\mfg_N^\rho\xi$ is a highest weight module with highest weight $\mu=(\mu_i)_{i=1}^n$ given~by
\begin{equation} \label{gtw-action}
\mu_i= -\tfrac{1}{2}A(\mathbf{P},u)-\tfrac{1}{2}\sum_{a=2}^i\deg\, P_a(u) +\delta_{i>\key}(\al-\tfrac{N}{4}) \qu \text{for all}\qu 1\leq i\leq n, 
\end{equation}
where $A(\mathbf{P},u)$ is as in \eqref{Ag_N} with $\mathbf{P}=(P_1(u),\ldots,P_n(u))$ and $i\mapsto \delta_{i>\key}$ is the indicator function of the set $\{\key+1,\ldots,n\}$. 
\end{lemma}

\begin{proof}
Let $\mu_\al(u)=(\mu_{\al,i}(u))_{i\in \mathcal{I}_N^+}$ be the $X(\mfg_N,\mcG)^{tw}$-highest weight determined by
 \begin{equation}
  \wt \mu_{\al,i}(u)= 2u \cdot \mathscr{g}(u) \; \text{ for all }\; 0\leq i\leq \key , \quad \wt \mu_{\al,i}(u)=2u\cdot \mathscr{g}(u)\left(\frac{\ley-\al-u}{u-\al} \right) \; \text{ for }\; \key+1\leq i\leq n, \label{mu-al}
 \end{equation}
where the index $i=0$ is omitted when $N=2n$. Since $P_i(u)=P_i(-u+n-i+2)$ for all $i\geq 2$, there exists monic polynomials $Q_2(u),\ldots,Q_n(u)$ such that 
\begin{equation*}
 P_i(u)=(-1)^{\deg\, Q_i(u)}Q_i(u-\ka/2)Q_i(-u+n-i+2-\ka/2),
\end{equation*}
and similarly since $P_1(u)=P_1(-u+\ka+2^\delta)$, there is a monic polynomial $Q_1(u)$ such that $P_1(u)=(-1)^{\deg\, Q_1(u)}Q_1(u-\ka/2)Q_1(-u+\ka/2+2^\delta)$. Let 
$L(\lambda(u))$ be a finite-dimensional irreducible $X(\mfg_N)$-module with Drinfeld polynomials $Q_1(u),\ldots,Q_n(u)$, and suppose $\xi\in L(\lambda(u))$ and $\eta \in V(\mu_\al(u))$ are highest weight vectors. Then, as a consequence of Lemmas \ref{L:unique} and \ref{L:QxP=QP}, there is an even series $g(u)\in 1+u^{-2}\C[[u^{-2}]]$ such that $V(\mu(u))^{\nu_g}$ is isomorphic to the irreducible quotient of 
$X(\mfg_N,\mcG)^{tw}(\xi\otimes \eta)$. As $\mfg^\rho_N$ acts identically in  $V(\mu(u))^{\nu_g}$ and $V(\mu(u))$, we can assume without loss of generality that $g(u)=1$.

Since $\tfrac{1}{2u}\wt \mu_{\al,i}(u)$ has $u^{-1}$ coefficient $\tfrac{1}{2}(2\cdot \mu_{\al,i}^{(1)}-\ley g_{ii}+[\mp]\ley)$, $g(u)=[\pm]1+[\mp]\tfrac{q}{2}u^{-1}+O(u^{-2})$, and $\mfrac{\ell-\al-u}{u-\al}=-1+(\ley-2\al)u^{-1}+O(u^{-2})$, we obtain from \eqref{mu-al} the relations
\begin{equation}
 \mu_{\al,i}^{(1)}=[\mp](\tfrac{q}{2}-\ell) \; \text{ for }\; 0\leq i\leq \key, \quad \mu_{\al,i}^{(1)}=[\pm](\tfrac{q}{2}+\ley-2\al) \; \text{ for }\; i\geq \key+1. \label{mu-al:2}
\end{equation}
Since $F_{ii}^{  \rho}=2g_{ii}F_{ii}$ and the embedding $\mfU\mfg_N^\rho\into X(\mfg_N,\mcG)^{tw}$ sends $F_{ii}^{  \rho}$ to $s_{ii}^{(1)}-(g_{ii}-1)\tfrac{p-q}{4}$ (see \eqref{grho->X}), we obtain 
\begin{equation*}
 2g_{ii}\mu_{\al,i}+(g_{ii}-1)\tfrac{p-q}{4}=\mu_{\al,i}^{(1)} \; \text{ for all }\; i\in\mathcal{I}_N^+,
\end{equation*}
where $\mu_{\al,i}$ denotes the $F_{ii}$-weight of $\eta$. Combining this with \eqref{mu-al:2} yields 
\begin{equation*}
 \mu_{\al,i}=0 \; \text{ for }\; 0\leq i\leq \key,\quad \mu_{\al,i}=\al-\tfrac{N}{4}\;\text{ for }\; \key+1\leq i\leq n.
\end{equation*}
Since $F_{ii}(\xi\otimes \eta)=F_{ii}\xi\otimes \eta +\mu_{\al,i}(\xi\otimes \eta)$, \eqref{gtw-action} now follows immediately from the formulas
\eqref{g-action} of Corollary \ref{C:g-action} and the fact that $\deg\, P_i(u)=2\deg\, Q_i(u)$ for each $i$.
\end{proof}

\begin{rmk}
 Lemma \ref{L:gtw} (and its proof) also applies for the twisted Yangians of type BCD0. In this case, $\alpha$ should be removed from the tuple 
 $(\alpha,P_1(u),\ldots,P_n(u))$. Note that, since $\key=n$, the term $\delta_{i>\key}(\al-\tfrac{N}{4})$ does not actually make an appearance in \eqref{gtw-action}, and that the definition of $\mu_\al(u)$ given in \eqref{mu-al} reduces to $(g_{ii}(u))_{i\in \mcI_{N}^+}$, and hence $V(\mu_\al(u))\cong V(\mcG)$.
\end{rmk}

\begin{prop}\label{P:int}
Suppose that $N\geq 4$ if $\mfg_N=\mfsp_N$ and $N\geq 5$ with $q\neq 2$ if $\mfg_N=\mfso_N$. Let $\mu(u)$, $V(\mu(u))$ and $(\alpha,\mathbf{P})$ be as in Lemma \ref{L:gtw} and assume that $V(\mu(u))$ is finite-dimensional. Then $2^{1-\delta}(\alpha-\tfrac{N}{4})$ is an integer satisfying
\begin{equation*}
 2^{1-\delta}(\alpha-\tfrac{N}{4})\leq A(\mathbf{P},u)+\sum_{a=2}^{\key+1}\deg\, P_a(u)+(1-2^{\delta-1})\deg\, P_{\key+2}(u),
\end{equation*}
where we recall that $\delta=1$ if $\mfg_N=\mfsp_{2n}$ and $\delta=0$ if $\mfg_N=\mfso_N$. 
\end{prop}
\begin{proof}
Since $V(\mu(u))$ is finite-dimensional, so is the $\mfg_{2\ell}$ highest weight module $\mfU\mfg_{2\ell}\xi\subset \mfU\mfg_N^\rho \xi$. The highest weight 
of this module is $(\mu_{\key+1},\ldots,\mu_n)$ with each $\mu_i$ as in the relation \eqref{gtw-action} of Lemma \ref{L:gtw}. 

If $\mfg_N=\mfsp_{2n}$, then $\mfg_{2\ell}=\mfsp_{2\ell}$ and \eqref{g-class} implies that $-\mu_{\key+1}\in \Z_{\geq 0}$. Otherwise, $\mfg_{2\ell}=\mfso_{2\ell}$ and 
\eqref{g-class} yields $-\mu_{\key+1}-\mu_{\key+2}\in \Z_{\geq 0}$. Taking into account that $\mu_{\key+1}$ and $\mu_{\key+2}$ are given by \eqref{gtw-action} we obtain the statement of the proposition.
\end{proof}

\section{Classification of one-dimensional representations} \label{sec:1dim}

In this section we classify the one-dimensional representations of $X(\mfg_N,\mfg_p\oplus\mfg_q)^{tw}$ and $Y(\mfg_N,\mfg_p\oplus\mfg_q)^{tw}$ when $(\mfg_N,\mfg_p\oplus\mfg_q)$ is a symmetric pair of type BDI or CII: see Propositions \ref{P:no1dim} and \ref{P:1dimso2}. In fact, Proposition \ref{P:no1dim} also holds for the twisted Yangians of type BCD0. These results prove that $X(\mfg_N,\mfg_p\oplus\mfg_q)^{tw}$ (and thus  $Y(\mfg_N,\mfg_p\oplus\mfg_q)^{tw}$) admits non-trivial one-dimensional representations if and only if $\mfg_p\oplus\mfg_q$ does (which occurs precisely when $\mfg_p$ or $\mfg_q$ is isomorphic to the one-dimensional Lie algebra $\mfso_2$). 

The main difficulty in proving Proposition \ref{P:1dimso2} is the construction of a one-parameter family $\{V(a)\}_{a\in \C}$ of one-dimensional representations 
for $X(\mfso_N,\mfso_{N-2}\oplus \mfso_2)^{tw}$ when $N\geq 5$: this is proven in Lemma \ref{L:K-1dim}, and stated explicitly in Corollary \ref{C:V(a)}. This one-parameter family of representations will play a crucial role in the proof of the classification of finite-dimensional irreducible representations of $X(\mfso_N,\mfso_{N-2}\oplus \mfso_2)^{tw}$ given in Theorem \ref{T:DI(a)-Class} of the next section.

\subsection{Twisted Yangians for the symmetric pairs \texorpdfstring{$(\mfg_N,\mfg_p\op\mfg_q)$}{} when \texorpdfstring{$\mfg_p,\mfg_q\ncong \mfso_2$}{}}\label{subsec:1dim-gen}

We begin  by classifying the one-dimensional representations of $X(\mfg_N,\mfg_N^\rho)^{tw}$ in all cases where $\mfg_N^\rho=\mfg_p\oplus \mfg_q$ with $\mfg_p \ncong \mfso_2 \ncong \mfg_q$. 
We also relax the requirement that $q>0$ so as to include the twisted Yangians of type BCD0, which correspond to $q=0$. 
\begin{prop}\label{P:no1dim} Assume that the symmetric pair $(\mfg_N,\mfg_{p}\oplus \mfg_q)$ is of type BCD0, BI, CII, or DI(a), and in addition that $\mfg_p\ncong \mfso_2\ncong \mfg_q$. Then, a representation $V$ of $X(\mfg_N,\mfg_p\oplus \mfg_q)^{tw}$ is one-dimensional if and only if $V\cong V(\mcG)^{\nu_g}$ for some $g(u)\in 1+u^{-2}\C[[u^{-2}]]$. 
\end{prop}
\begin{rmk}
 The proof of the Proposition exploits the relationship between $X(\mfg_N,\mfg_p\oplus \mfg_q)^{tw}$ and the Molev-Ragoucy reflection algebra $\mcB(n,\ley)$ which was studied in Subsection 4.3 of \cite{GRW2}. For the definition and main properties of $\mcB(n,\ley)$, we refer the reader to \cite{MR} and Subsection 3.6 of \cite{GRW2}. 
\end{rmk}
\begin{proof}
Suppose that  $V$ is a one-dimensional representation of $X(\mfg_N,\mfg_p\oplus \mfg_q)^{tw}$.  By Proposition \ref{P:necessary} (if $q\neq 0$) and Theorem \ref{T:BCD0-class} (if $q=0$), $V$ can be associated to a tuple 
$(\alpha,P_1(u),\ldots,P_n(u))$, where the scalar $\alpha$ should be omitted if $q=0$. If $q=1$ or $q=0$, then $\mfg_p\oplus \mfg_q=\mfg_p$. Otherwise, both $\mfg_p$ and $\mfg_q$ are complex semisimple Lie algebras. In either case, $\mfg_p\oplus\mfg_q$ is semisimple and thus admits no nontrivial one-dimensional representations. Consequently, $V$ is isomorphic to the trivial representation of $\mfg_p\oplus \mfg_q$ when viewed as a module of this Lie algebra. Therefore, relation \eqref{gtw-action} of Lemma \ref{L:gtw} becomes equivalent to
\begin{equation*}
 A(\mathbf{P},u)+\sum_{a=2}^i\deg\, P_a(u)=\delta_{i>\key}(2\al-\tfrac{N}{2})  \quad \text{ for all }\quad 1\leq i\leq n,
\end{equation*}
from which it can be deduced that $P_a(u)=1$ for all $a\neq \key+1$, and $\deg\, P_{\key+1}(u)=2\al-\tfrac{N}{2}$. In the $q=0$ case, this completes the proof 
as $(P_1(u),\ldots,P_n(u))$ is equal to $(1,\ldots,1)$, the Drinfeld tuple corresponding to the trivial representation. 
To complete the proof in the $q\neq 0$ case, it suffices to show that $\deg\, P_{\key+1}(u)=0$, as this will imply $(\alpha,P_1(u),\ldots,P_n(u))=(\tfrac{N}{4},1,\ldots,1)$. Since $V(\mcG)$ is also associated to this tuple, the desired conclusion will follow from Lemma \ref{L:unique}. 

\noindent\textit{Case 1}: $\key>0$.

Suppose first that $\key>0$, and set $M=N-2\ley+2$. Consider the one-dimensional representation $V_{\ley-1}$ of $X(\mfg_{M},\mfg_{M-2}\oplus \mfg_2)^{tw}$ furnished by Proposition \ref{P:ind}. By Corollary \ref{C:poly-low}, this module is associated to $(\gamma,Q_1(u),\ldots,Q_{\key+1}(u))$, where $\gamma=\al-\frac{\ley-1}{2}$, 
$Q_a(u)=1$ for $1\leq a\leq \key$, and $Q_{\key+1}(u)=P_{\key+1}(u+\tfrac{\ley-1}{2})$. The central series $w(u)$ of $X(\mfg_{M},\mfg_{M-2}\oplus \mfg_2)^{tw}$ (see \eqref{Y=X/(w-1)}) operates as multiplication by a scalar series $\mathsf{w}(u)$ in $V_{\ley-1}$. Let $\mathsf{q}(u)\in 1+u^{-1}\C[[u^{-1}]]$ be the unique series such that 
$\mathsf{w}(u)=\mathsf{q}(u)\mathsf{q}(u+\ka)$. Since $\mathsf{w}(u)$ is even, the uniqueness of this expansion forces the relation $\mathsf{q}(u)=\mathsf{q}(\ka-u)$. In particular, the series $g(u)=\mathsf{q}(u+\ka/2)^{-1}$ is even and in the twisted module $W=(V_{\ley-1})^{\nu_g}$ the series $w(u)$ operates as $g(u-\ka/2)g(u+\ka/2)\mathsf{w}(u)=1$. Since $W$ is one-dimensional, Proposition 4.14  and Remark 4.15 of \cite{GRW2} imply that $W$ can be regarded as a representation of the Molev-Ragoucy reflection algebra $\mcB(\key+1,1)$, which is necessarily associated to the tuple $(\gamma,Q_2(u),\ldots,Q_{\key+1}(u))$ in the sense of Theorem 4.6 of \cite{MR} (see also (4.74) -- (4.76) of \cite{GRW2}). 

For $1\leq i,j\leq \key+1$, we let $b_{ij}(u)$ denote the standard generating series of $\mcB(\key+1,1)$, and for $1\leq i,j\leq 2$ we let $b_{ij}^\circ(u)$ denote the standard generating series of $\mcB(2,1)$ (see \cite[Definition 3.25]{GRW2}). Since $W$ is one-dimensional, it inherits the structure of a  $\mcB(2,1)$-module by allowing $b_{ij}^\circ(u)$ to operate as $b_{i+\key-1,j+\key-1}(u)$ for $i,j\in \{1,2\}$. This can be verified directly, but it also follows from a more general result observed in the first part of the proof of Theorem 4.6 of \cite{MR}. The resulting $\mcB(2,1)$-module has Drinfeld tuple $(\gamma, Q_{\key+1}(u))$. The reflection algebra $\mcB(2,1)$ is isomorphic to the twisted Yangian $Y^+(2)$ (see \cite[Proposition 4.3]{MR} as well as the remarks concluding Section 4.2 of  \textit{loc. cit.}), hence $W$ can also be viewed as a one-dimensional representation of $Y^+(2)$. The arguments used to prove \cite[Proposition 4.4]{MR} show that this irreducible $Y^+(2)$-module corresponds to the pair $(\gamma-\tfrac{1}{2},Q_{\key+1}(u+\tfrac{1}{2}))$ (see \eqref{Y+2:findim}). 
On the other hand, by Corollary 4.4.5 of \cite{Mobook}, as a $SY^+(2)$-module $W$ must be isomorphic to the module
%
%
$V(\gamma-1)$, which is the one-dimensional representation of $Y^+(2)$ with the highest weight 
\begin{equation*}
 \gamma(u)=\frac{1+(\gamma-1/2)u^{-1}}{1+1/2u^{-1}},
\end{equation*}
and may be viewed as a $SY^+(2)$ module by restriction (see equation (4.21) in \cite{Mobook}). Since this module corresponds to the pair $(\gamma-\tfrac{1}{2},1)$, we obtain 
$\deg P_{\key+1}(u)=\deg Q_{\key+1}(u)=0$.

\noindent \textit{Case 2}: $\key=0$. 

In this case, $V_{\ley-1}$ is a one-dimensional representation of $X(\mfso_3,\mfso_2)^{tw}$ which, by Corollary \ref{C:poly-low}, is associated to the pair $(\gamma,Q_1(u))=(\alpha-\frac{\ley-1}{2}, P_{1}(u+\tfrac{\ley-1}{2}))$. Moreover, it can be made into a $Y^+(2)$-module via the isomorphism \eqref{iso:so3so2}, and the proof of Proposition \ref{P:so3class} shows that, as a $Y^+(2)$-module, $V_{\ley-1}$ corresponds to the pair $(P^\circ(u),\gamma^\circ)=(2^{\deg Q_1(u)} Q(\tfrac{u+1}{2}),2\gamma-1)$. Repeating the last part of the argument of Case 1, we are able to conclude that $\deg P_1(u)=\deg P^\circ(u)=0$, which completes the proof of the Proposition. \qedhere
\end{proof}
The below corollary of Proposition \ref{P:no1dim} now follows immediately from the fact that $Y(\mfg_N,\mfg_p\oplus \mfg_q)^{tw}$ is the $\nu_g$-stable subalgebra of $X(\mfg_N,\mfg_p\oplus \mfg_q)^{tw}$.
\begin{crl}
 Let $(\mfg_N,\mfg_{p}\oplus \mfg_q)$ satisfy the conditions of Proposition \ref{P:no1dim}. Then, up to isomorphism, $V(\mcG)$ is the unique one-dimensional representation 
 of $Y(\mfg_N,\mfg_p\oplus \mfg_q)^{tw}$. 
\end{crl}
%

\subsection{Twisted Yangians for the symmetric pairs \texorpdfstring{$(\mfg_N,\mfg_p\op\mfg_q)$}{} when \texorpdfstring{$\mfg_q\cong \mfso_2$}{}} \label{subsec:1dim-so2}

We now turn to the twisted Yangians associated to pairs of the form $(\mfg_N,\mfg_N^\rho)=(\mfso_N,\mfso_{N-2}\oplus\mfso_2)$ with $N\geq 5$.  
 
The following technical lemma provides a one-parameter family of matrix solutions to \eqref{TX-RE} and \eqref{TX-symm} associated to the pair $(\mfso_N,\mfso_{N-2}\oplus \mfso_{2})$.

\begin{lemma} \label{L:K-1dim}
Let $a\in\C$ and let $\mcG$ be of type BI or DI with $q=2$ and $N\geq 5$. Then the matrix
\eq{
K(u;a) = k(u) \left( I - \frac{2u}{u-a} E_{-n,-n} - \frac{2u}{u+a-2d} E_{nn} \right), \qu\text{where}\qu k(u)=\frac{(u-a)(u+a-2d)}{(u-d)^2} \label{K-1dim}
}
and $d=N/4-1$, is a one-parameter solution of the reflection equation \eqref{TX-RE}. Moreover, it satisfies the symmetry relation 
\eq{
K^t(u;a) = K(\ka-u;a) + \frac{K(u;a)-K(\ka-u;a)}{2u-\ka} + \frac{\Tr(\mcG(u))\,K(\ka-u;a) - \Tr(K(u;a))\cdot I }{2u-2\ka}. \label{K-1dim:Sym}
}
where $\mcG(u)$ is the is the matrix defined in \eqref{G(u)}.
\end{lemma}

\begin{proof}
We begin by showing that $K(u;a)$ satisfies \eqref{TX-RE}, that is
\[
 R(u-v)\, K_1(u;a)\, R(u+v)\, K_2(v;a) = K_2(v;a)\,R(u+v) \,K_1(u;a)\,R(u-v). 
\]
Notice that we only need to show this for $k(u)^{-1} K(u;a)$. Denote 
\eq{
G(u) = - \frac{2u}{u-a} E_{-n,-n} - \frac{2u}{u+a-2d} E_{nn},  \label{L:K-1dim:a0}
}
so that $K(u;a)=k(u)(I + G(u))$. Since $I$ is a solution to \eqref{TX-RE}, our task is to show that 
\spl{
& R(u-v)\,G_1(u)\,R(u+v) + R(u-v)\, R(u+v)\,G_2(v) + R(u-v)\,G_1(u)\, R(u+v)\,G_2(v) \\
& \qq = G_2(v)\,R(u+v)\,R(u-v) + R(u+v)\,G_1(u)\,R(u-v)  + G_2(v)\,R(u+v)\,G_1(u)\,R(u-v). \label{L:K-1dim:a}
}
We first show that 
\spl{
& \left(1-\frac{P}{u-v}\right)\left( G_1(u)\left(1-\frac{P}{u+v}\right) + \left(1-\frac{P}{u+v}\right)G_2(v) + G_1(u)\left(1-\frac{P}{u+v}\right)G_2(v) \right) + H(u,v)
\\
& \qq = \left( G_2(v)\left(1-\frac{P}{u+v}\right)+ \left(1-\frac{P}{u+v}\right)G_1(u) + G_2(v)\left(1-\frac{P}{u+v}\right)G_1(u) \right) \left(1-\frac{P}{u-v}\right) , \label{L:K-1dim:b}
}
where 
\eq{
H(u,v) = \frac{8 u v (a-d)}{(u-a) (v-a) (u+a-2d) (v+a-2d)} \left( E_{-n,n}\ot E_{n,-n} - E_{n,-n}\ot E_{-n,n}\right ).  \label{L:K-1dim:b1}
}
The equality \eqref{L:K-1dim:b} reduces to
\[
\left[ P , 2v\,G_2(u) - 2u\,G_2(v) - (u-v)\,G_2(u)\,G_2(v) \right] = 0 .
\]
This follows from the following computations:
\eqn{
2v\,G(u) - 2u & \,G(v) - (u-v)\,G(u)\,G(v) \\
& \qq = -\left(\frac{4uv}{u-a} - \frac{4uv}{v-a}\right) E_{-n,-n} - \left( \frac{4uv}{u+a-2d} - \frac{4uv}{v+a-2d}\right) E_{nn} \\
& \qq\qq - \frac{4uv\,(u-v)}{(u-a)(v-a)}\, E_{-n,-n} - \frac{4uv\,(u-v)}{(u+a-2d)(v+a-2d)} E_{nn} = 0,
}
thus implying \eqref{L:K-1dim:b}. Next we use
\[
Q^2 = NQ, \qu (1-u^{-1}P)\,Q = (1-u^{-1})\,Q, \qu Q\,G(u)\,Q=g(u)\,Q \qu\text{where}\qu g(u)=  - \frac{2u}{u-a} - \frac{2u}{u+a-2d} \label{L:K-1dim:b2}
\]
and subtract \eqref{L:K-1dim:b} from \eqref{L:K-1dim:a}. Then \eqref{L:K-1dim:a} holds if and only if the following equality is verified:
%
%
\eqn{
& \frac{G_1(u)\,Q + Q\,G_2(v) + G_1(u)\,Q\,G_2(v)}{u+v-\ka} - \frac{G_2(v)\,Q + Q\,G_1(u) + G_2(v)\,Q\,G_1(u)}{u+v-\ka} \\
& -\frac{ G_2(u)\,Q + Q\,G_2(v) + G_2(u)\,Q\,G_2(v)}{(u-v)(u+v-\ka)} + \frac{G_2(v)\,Q + Q\,G_2(u) + G_2(v)\,Q\,G_2(u)}{(u-v)(u+v-\ka)} \\
& +\frac{Q\,( G_1(u)+ G_2(v)+G_1(u)\,G_2(v))}{u-v-\ka} - \frac{(G_2(v) + G_1(u) + G_2(v)\,G_1(u))\,Q}{u-v-\ka} \\
& - \frac{Q\,(G_2(u) + G_2(v) + G_2(u)\,G_2(v))}{(u+v)(u-v-\ka)} + \frac{(G_2(v)+ G_2(u) + G_2(v)\,G_2(u))\,Q}{(u+v)(u-v-\ka)}\\
& + \frac{ g(u)\,Q + N\,Q\,G_2(v) + g(u)\, Q\,G_2(v) }{(u-v-\ka)(u+v-\ka)} - \frac{N\,G_2(v)\,Q + g(u)\,Q + g(u)\,G_2(v)\,Q}{(u-v-\ka)(u+v-\ka)} = H(u,v) .  
}
Since $[G(u),G(v)]=0$ we can simplify the equation above to
\spl{
& \frac{[G_1(u)-G_2(v),Q] + G_1(u)\,Q\,G_2(v)-G_2(v)\,Q\,G_1(u)}{u+v-\ka} +\frac{[Q, G_1(u)+ G_2(v)+G_1(u)\,G_2(v)]}{u-v-\ka} \\
& + \frac{[Q,G_2(u)-G_2(v)] + G_2(v)\,Q\,G_2(u)-G_2(u)\,Q\,G_2(v)}{(u-v)(u+v-\ka)} - \frac{[Q,G_2(u) + G_2(v) + G_2(u)\,G_2(v)]}{(u+v)(u-v-\ka)} \\
& + \frac{ (N+ g(u))\, [Q,G_2(v)] }{(u-v-\ka)(u+v-\ka)}  = H(u,v).   \label{L:K-1dim:d}
}
Now observe that $G^t(u) = -u\,(2d-u)^{-1}G(2d-u)$ which implies that
\[
G_1(u)\,Q = -\frac{u}{2d-u}\,G_2(2d-u)\,Q, \qq Q\,G_1(u) = -\frac{u}{2d-u}\,Q\,G_2(2d-u) .
\]
Moreover,
\begin{align}
\begin{split}
G_2(u)\,Q\,G_2(v) &= \frac{4 u v}{(u-a) (v-a)} E_{nn}\ot E_{-n,-n} + \frac{4 u v}{(u-a) (v+a-2d)} E_{n,-n}\ot E_{-n,n} \\
&+ \frac{4 u v}{(u+a-2d) (v-a)} E_{-n,n}\ot E_{n,-n} + \frac{4 u v}{(u+a-2d) (v+a-2d)} E_{-n,-n}\ot E_{nn}. \label{L:K-1dim:GQG}
\end{split}
\end{align}
%
%
Using \eqref{L:K-1dim:b1}, \eqref{L:K-1dim:GQG} and $\ka=2d+1$ we compute the following identity 
\eqn{
& \frac{ G_1(u)\,Q\,G_2(v)-G_2(v)\,Q\,G_1(u)}{u+v-\ka} + \frac{G_2(v)\,Q\,G_2(u)-G_2(u)\,Q\,G_2(v)}{(u-v)(u+v-\ka)} \\
& \qu = \frac{4uv}{u+v-\ka} \left(\frac{1}{(u-a) (v-a)}-\frac{1}{(u+a-2d)(v+a-2d)}\right) \left( E_{-n,n}\ot E_{n,-n} - E_{n,-n}\ot E_{-n,n} \right) \\
& \qu + \frac{4 u v}{(u-v)(u+v-\ka)} \left(\frac{1}{(u-a) (v+a-2d)}-\frac{1}{(v-a) (u+a-2d)}\right) \left( E_{-n,n}\ot E_{n,-n} - E_{n,-n}\ot E_{-n,n} \right) \\
& \qu = \frac{8 u v (a-d)}{(u-a)(v-a) (u+a-2d) (v+a-2d)} \left( E_{-n,n}\ot E_{n,-n} - E_{n,-n}\ot E_{-n,n} \right) = H(u,v) . 
}
By combining the identities above and denoting $\wt u = u\,(2d-u)^{-1}$ we rewrite \eqref{L:K-1dim:d} as
\eqn{
& \bigg[ Q, \frac{\wt u\,G_2(2d-u)+G_2(v)}{u+v-\ka} + \frac{G_2(u)-G_2(v)}{(u-v)(u+v-\ka)} + \frac{ (N+ g(u))\, G_2(v) }{(u-v-\ka)(u+v-\ka)}  \\
& \qq -\frac{\wt u\,G_2(2d-u) - G_2(v)+\wt u\, G_2(2d-u)\,G_2(v)}{u-v-\ka} - \frac{G_2(u) + G_2(v) + G_2(u)\,G_2(v)}{(u+v)(u-v-\ka)} \bigg] = 0.   
}
Denoting the commutator above by $[Q,1\ot F(u,v)]$ we only need to verify that $F(u,v)=0$, which follows by a direct computation using \eqref{L:K-1dim:a0}, the explicit form of $g(u)$ and $\ka=2d+1=N/2-1$, as we now illustrate. After reorganizing the various terms and multiplying by $(u-v-\ka) (u+v-\ka)$, we obtain, with $F'(u,v) = (u-v-\ka) (u+v-\ka) F(u,v)$:
\eqn{
F'(u,v) &= -\frac{ 2 \ka v\, G(u)}{u^2-v^2} -\frac{2 u v\, G(2d-u)}{2d-u} + 2 u \bigg(1-\frac{1}{u-a}-\frac{1}{u+a-2d}+\frac{\ka}{u^2-v^2}\bigg)G(v) \\ & \hspace{7.3cm} -(u+v-\ka) \bigg(\frac{u\,G(2d-u)\,G(v)}{2d-u}+\frac{G(u)\,G(v)}{u+v}\bigg) \\
& = 4uv\bigg( -\bigg(\frac{1}{u+a-2d} - \frac{\ka}{(u-a)(u^2-v^2)} \bigg) E_{-n,-n} - \bigg( \frac{1}{u-a} - \frac{\ka}{(u+a-2d)(u^2-v^2)} \bigg) E_{nn} \\
& \hspace{1.3cm} - \bigg(1-\frac{1}{u-a}-\frac{1}{u+a-2d}+\frac{\ka}{u^2-v^2}\bigg)\bigg( \frac{1}{v-a} E_{-n,-n} + \frac{1}{v+a-2d} E_{nn} \bigg) \\
& \hspace{1.3cm} -\frac{u+v-\ka}{(u-a)(u+a-2d)(u+v) }\bigg( \frac{(u+a-2d)-(u-a)(u+v)}{v-a} E_{-n,-n} \\ & \hspace{8.5cm}+  \frac{(u-a)-(u+a-2d)(u+v)}{v+a-2d} E_{nn}\bigg)\bigg) 
\\
& = \frac{2d+1-\ka}{(v-a)(u+a-2d)} E_{-n,-n} +\frac{2d+1-\ka}{(u-a)(v+a-2d)} E_{nn} = 0 . 
}
This completes the proof that $K(u;a)$ is a solution to~\eqref{TX-RE}.

We now turn to proving \eqref{K-1dim:Sym}. Our work thus far shows that the assignment $\wt S(u)\mapsto K(u;a)$ extends to a homomorphism of algebras $\phi_a:\wt X(\mfg_N,\mcG)^{tw}\to \C$ where $\wt X(\mfg_N,\mcG)^{tw}$ is the extended reflection algebra (see Subsection \ref{subsec:twYa}). By Theorem 5.2 of \cite{GR}, $K(u;a)$ satisfies \eqref{K-1dim:Sym} if and only if $\phi_a(c(u))=1$: see  \eqref{c(u)}. Since $N\geq 5$, we have $p=N-2>2=q$ and we may therefore apply Corollary \ref{C:c(u)} which, by \eqref{g(u):exp}, implies that 
\begin{equation}
 \phi_a(c(u))=\frac{\mathscr{g}(\ka-u)}{\mathscr{g}(u)}\cdot \frac{\Tr(K(u;a))}{\Tr(K(\ka-u;a))}=\bigg(\frac{u+1-N/4}{u-N/4}\bigg)^2\frac{\Tr(K(u;a))}{\Tr(K(\ka-u;a))}. \label{K-1dim:phi}
\end{equation}
By definition of $K(u;a)$, 
\[
 \Tr(K(u;a))=k(u)\left(N-\frac{2u}{u-a}-\frac{2u}{u+a-2d}\right)=\frac{N(u-a)(u+a-2d)-2u(2u-2d)}{(u-d)^2}.
\]
Let $P(u)$ be the numerator of the right-hand side. Using that $N=2\ka+2$ and $2d=\ka-1$ we find that 
\[
 P(u)=N(u-a)(u+a-\ka+1)-2u(2u-\ka+1)=N(u-a)(u+a-\ka)-2u(2u-2\ka)-Na, 
\]
and hence $P(u)$ is invariant under the substitution $u\mapsto \ka-u$. This implies that 
\[
 \frac{\Tr(K(u;a))}{\Tr(K(\ka-u;a))}=\bigg(\frac{\ka-u-d}{u-d} \bigg)^2=\bigg(\frac{u-N/4}{u+1-N/4}\bigg)^2,
\]
which by \eqref{K-1dim:phi} proves that $\phi_a(c(u))=1$ for any $a\in \C$, as required. 
\end{proof}

As a consequence of this lemma we obtain the existence of a family of one-dimensional representations $\{V(a)\}_{a\in \C}$:
\begin{crl}\label{C:V(a)}
Let $a\in \C$. Then the assignment $S(u)\mapsto K(u;a)$ yields a one-dimensional representation $V(a)$ of $X(\mfso_N,\mfso_{N-2}\oplus \mfso_{2})^{tw}$ with the highest weight 
$\gamma^a(u)$ given by 
\begin{equation}
\gamma_i^a(u)=\frac{(u-a)((-1)^{\delta_{in}}u+a-2d)}{(d-u)^2} \; \text{ for all }\; i\in \mcI_N^+, \label{K-1dim:hw}
\end{equation}
where $d=N/4-1$. 
\end{crl}

This corollary is an immediate consequence of Lemma \ref{L:K-1dim}, the formula for each $\gamma_i^a(u)$ following from \eqref{K-1dim}. 
Note that \eqref{K-1dim:hw} implies that the auxiliary tuple $\wt \gamma^a(u)$ is determined by 
\begin{equation*}
 \wt \gamma_i^a(u)=2u\cdot\frac{(u-a)((-1)^{\delta_{in}}u+a-2d-1+\delta_{in})}{(d-u)^2} \; \text{ for all }\; i\in \mcI_N^+,
\end{equation*}
and thus that $V(a)$ has Drinfeld tuple $(\alpha,P_1(u),\ldots,P_n(u))$ equal to 
\begin{equation}
(\ka-a,1,\ldots,1) \label{V(a):tuple} 
\end{equation}
because $1+2d=\kappa$. The last two results of this subsection provide classifications of the one-dimensional representations of $X(\mfso_N,\mfso_{N-2}\oplus \mfso_2)^{tw}$ and of $Y(\mfso_N,\mfso_{N-2}\oplus \mfso_2)^{tw}$, respectively, when $N\geq 5$. 
\begin{prop}\label{P:1dimso2}
 Let $N\geq 5$. Then a representation $V$ of $X(\mfso_N,\mfso_{N-2}\oplus \mfso_2)^{tw}$ is one-dimensional if and only if $V\cong V(a)^{\nu_g}$ for some $g(u)\in 1+u^{-2}\C[[u^{-2}]]$. 
\end{prop}
\begin{proof}
 By Corollary \ref{C:V(a)}, for any $a\in \C$, $V(a)$ provides a one-dimensional representation of $X(\mfso_N,\mfso_{N-2}\oplus \mfso_2)^{tw}$. Hence, to prove the proposition we need to establish that, if $V$ is a one-dimensional representation of $X(\mfso_N,\mfso_{N-2}\oplus \mfso_2)^{tw}$, then it can be associated to a tuple of the form 
 \begin{equation*}
  (\alpha,P_1(u),\ldots,P_n(u))=(\al,1,\ldots,1). 
 \end{equation*}
By Lemma \ref{L:unique}, this will imply that $V\cong V(a)^{\nu_g}$ for some $g(u)\in 1+u^{-2}\C[[u^{-2}]]$, where $a$ is determined by $\al=1+2d-a$.  

Let $V$ be a one-dimensional representation. Since $N\geq 5$, $\mfso_{N-2}$ is semisimple and thus $V$ is equal to the trivial representation when viewed as a $\mfso_{N-2}$-module. 
Since $\mfso_2$ is one-dimensional, $F_{11}$ operates in $V$ as multiplication by a scalar $\gamma$. In particular, the highest weight of $V$ as a $(\mfso_{N-2}\oplus \mfso_{2})$-module is 
$(\mu_i)_{i=1}^n=(0,\ldots,0,\gamma)$. The relation \eqref{gtw-action} of Lemma \ref{L:gtw} therefore implies that $P_i(u)=0$ for all $1\leq i\leq n-1$ and 
\begin{equation*}
 \deg P_n(u)=2\alpha-\tfrac{N}{2}-2\gamma.
\end{equation*}
To complete the proof, we need to see that $\deg P_n(u)=0$. This can shown using the same argument as given in the $\key>0$ case of the proof of Proposition \ref{P:no1dim}. 
\end{proof}
\begin{crl}\label{C:1dimso2}
 A representation $V$ of $Y(\mfso_N,\mfso_{N-2}\oplus \mfso_2)^{tw}$ (with $N\geq 5$) is one-dimensional if and only if there is $a\in \C$ such that $V\cong V(a)$. 
\end{crl}

\begin{rmk}\label{R:1dim}
$ $ 
\begin{enumerate}[leftmargin=1.8em,topsep=-3pt]
 \item If $V$ is a one-dimensional representation of $X(\mfso_4,\mfso_2\oplus \mfso_2)^{tw}$ (resp. of $X(\mfso_3,\mfso_2)^{tw}$) then there exists 
 $(\mu_1,\mu_2)\in \C^2$ (resp. $\mu \in \C)$ and $g(u)\in 1+u^{-2}\C[[u^{-2}]]$ such that $V\cong V(\mu_1,\mu_2)^{\nu_g}$ (resp. $V\cong V(\mu)^{\nu_g}$): see 
 Corollaries \ref{C:so4} and \ref{C:so3} for the definitions of $V(\mu_1,\mu_2)$ and $V(\mu)$, respectively. This follows from the isomorphisms \eqref{iso:so4} and \eqref{iso:so3so2} together with the fact that the family $\{V(\gamma)\}_{\gamma\in \C}$, which appeared in the proof of Proposition \ref{P:no1dim}, is a complete list of the isomorphisms classes of one-dimensional representations of $Y^+(2)$ up to twisting by automorphisms of the form $S^\circ(u)\mapsto h(u)S^\circ(u)$ with $h(u)\in 1+u^{-2}\C[[u^{-2}]]$. This fact follows from Corollary 4.4.5 of \cite{Mobook}. 
 
 \item By Lemma 6.1 of \cite{GRW2}, the twisted Yangians associated to the symmetric pairs $(\mfg_{2n},\mfgl_n)$ of type CI and DIII (see \cite{GR,GRW2}) admit a family
 $\{V(a)\}_{a\in \C}$ of one-dimensional representations. The arguments used in this section can also be applied for twisted Yangians of these types, yielding analogues of Proposition \ref{P:1dimso2} and Corollary \ref{C:1dimso2} for $X(\mfg_{2n},\mfgl_n)^{tw}$ and $Y(\mfg_{2n},\mfgl_n)^{tw}$, respectively. 
\end{enumerate}
\end{rmk}


\section{Classification of finite-dimensional irreducible representations} \label{sec:main}

In this section, we prove the main results of this paper: the classification of finite-dimensional irreducible modules of the twisted Yangians associated to the pairs $(\mfso_N,\mfso_{N-2}\oplus \mfso_2)$ with $N\geq 5$ and $(\mfso_{2n+1},\mfso_{2n})$  with $n\geq 2$. 
 
The one-dimensional representations $\{V(a)\}_{a\in\C}$ constructed in Lemma \ref{L:K-1dim} and Corollary \ref{C:V(a)} provide the last ingredient 
necessary to classify the finite-dimensional irreducible representations of $X(\mfso_N,\mfso_{N-2}\oplus \mfso_2)^{tw}$ ($N\geq 5$) and its unextended counterpart, as is proven in Theorem \ref{T:DI(a)-Class} and Corollary \ref{C:DI(a)class}. In particular, Theorem \ref{T:DI(a)-Class} proves that for these twisted Yangians the necessary conditions of Proposition \ref{P:necessary} are in fact sufficient.
 
In the general setting the necessary conditions of Section \ref{sec:Nec} are not sufficient for determining exactly when the irreducible module $V(\mu(u))$ is finite-dimensional. We will illustrate this in Subsection \ref{subsec:q=1} by classifying the finite-dimensional irreducible representations of $X(\mfso_{2n+1},\mfso_{2n})^{tw}$ and of $Y(\mfso_{2n+1},\mfso_{2n})^{tw}$ in Theorem \ref{T:BI(b)-Class} and Corollary \ref{C:BI(b)class}, respectively. In the process we will obtain a stronger version of Proposition \ref{P:int} for these twisted Yangians which is proven without directly using the representation theory of $\mfso_{2n}$: see Proposition \ref{P:q=1-nec}.


\subsection{Twisted Yangians for the symmetric pairs \texorpdfstring{$(\mfso_N,\mfso_{N-2}\oplus \mfso_{2})$}{}}\label{subsec:q=2}

The following theorem provides a classification of the finite-dimensional irreducible representations of $X(\mfso_N,\mfso_{N-2}\oplus \mfso_2)^{tw}$ with $N\geq 5$. 
\begin{thrm}\label{T:DI(a)-Class}
Let $N\geq 5$ and suppose that $\mu(u)$ satisfies the conditions of Proposition \ref{P:nontriv} so that the Verma module $M(\mu(u))$ is non-trivial. Then
the irreducible $X(\mfso_{N},\mfso_{N-2}\oplus \mfso_2)^{tw}$-module $V(\mu(u))$ is finite-dimensional if and only if there are  monic polynomials $P_1(u),\ldots,P_n(u)$ 
together with a scalar $\alpha\in \C\setminus Z(P_n(u))$ such that
\begin{equation}
\frac{\wt \mu_{i-1}(u)}{\wt \mu_i(u)}=\frac{P_i(u+1)}{P_i(u)}\left(\frac{\al-u}{\al+u-1}\right)^{\del_{i,n}} \;\text{ with }\;  P_i(u)=P_i(-u+n-i+2) \label{T:DI(a)-Class.1}
\end{equation}
for all $2\leq i\leq n$, while $P_1(u)$ must satisfy $P_1(u)=P_1(-u+\tfrac{N}{2})$ and the relation 
\begin{equation}\label{T:DI(a)-Class.2}
 \frac{\wt \mu_{\mathscr{a}}(\ka-u)}{\wt \mu_{\mathscr{a}+1}(u)}=\left(\frac{u+1-\tfrac{N}{4}}{u-\tfrac{N}{4}}\right)^2\frac{P_1(u+2^{\mathscr{a}-1})}{P_1(u)}\cdot \frac{\ka-u}{u} \quad \text{ where }\quad \mathscr{a}=2n+1-N.
\end{equation}
Additionally, when $V(\mu(u))$ is finite-dimensional the corresponding tuple $(\al,P_1(u),\ldots, P_n(u))$ is unique. 
\end{thrm}

\begin{proof}
$(\Longrightarrow)$ If $V(\mu(u))$ is finite-dimensional, then the existence of $\alpha$ and $P_1(u),\ldots,P_n(u)$ satisfying  the conditions of the theorem 
 is guaranteed by Proposition \ref{P:necessary}. Here we note that 
as a consequence of \eqref{nontriv.2}, the $\mathscr{a}=0$ version of \eqref{T:DI(a)-Class.2} is equivalent to the $N=2n+1$ version of \eqref{nec.2} with $\key=n-1$. 
Moreover, the uniqueness assertion has been proven in Lemma \ref{L:unique}. 
 
 $(\Longleftarrow)$ Conversely, assume that $\mu(u)$ satisfies the conditions of Proposition \ref{P:nontriv} and there exists a scalar $\al$ and monic polynomials $P_1(u),\ldots,P_n(u)$ satisfying the relations of the theorem. Since $P_1(u)=P_1(-u+\tfrac{N}{2})$ and $P_i(u)=P_i(-u+n-i+2)$ for all $2\leq i\leq n$, there exists monic polynomials $Q_1(u),\ldots,Q_n(u)$ such that 
  \begin{gather*}
   P_1(u)=(-1)^{\deg Q_1(u)}Q_1(u-\ka/2)Q_1(-u+\tfrac{N}{2}-\ka/2),\\ P_i(u)=(-1)^{\deg Q_i(u)}Q_i(u-\ka/2)Q_i(-u+n-i+2-\ka/2) \; \text{ for each }\; 2\leq i\leq n. 
  \end{gather*}
Let $L(\lambda(u))$ be a finite-dimensional irreducible $X(\mfso_N)$-module with Drinfeld polynomials $Q_1(u),\ldots,Q_n(u)$. It follows from Theorem \ref{T:X-class} that such a module exists and is uniquely determined up to twisting by automorphisms of the form $\mu_f:T(u)\mapsto f(u)T(u)$ (see also Corollary 5.19 of \cite{AMR}). Let $\xi\in L(\lambda(u))$ be a highest weight vector. Consider the one-dimensional module $V(a)$ from Corollary \ref{C:V(a)} with $a=\ka-\al$, and let $\eta\in V(a)$ be any nonzero vector. Since, by \eqref{V(a):tuple}, $V(\ka-\al)$ is associated to the tuple $(\alpha,1,\ldots,1)$, Lemma~\ref{L:QxP=QP} implies that the finite-dimensional $X(\mfso_N,\mfso_{N-2}\oplus \mfso_2)^{tw}$ highest weight module 
\begin{equation*}
   X(\mfso_N,\mfso_{N-2}\oplus \mfso_2)^{tw}(\xi\otimes \eta)\subset L(\lambda(u))\otimes V(a)
\end{equation*}
has highest weight $\mu^\sharp(u)$ which is also associated to $(\alpha,P_1(u),\ldots,P_n(u))$. By Lemma \ref{L:unique}, there exists $g(u)\in 1+u^{-2}\C[[u^{-2}]]$ such that $V(\mu^\sharp(u))\cong V(\mu(u))^{\nu_g}$. Since $V(\mu^\sharp(u))$ is finite-dimensional, we can conclude that the same is true for $V(\mu(u))$.
\end{proof}
 
\begin{rmk}
Before reinterpreting the above theorem as a classification result for finite-dimensional irreducible representations of $Y(\mfso_{N},\mfso_{N-2}\oplus \mfso_{2})^{tw}$, we give a few remarks on the relation \eqref{T:DI(a)-Class.2}: 
\begin{enumerate}[leftmargin=1.8em,topsep=-2pt]
  \item As has been pointed out in the proof of Theorem \ref{T:DI(a)-Class},  when $N$ is odd the relation \eqref{nontriv.2} implies that \eqref{T:DI(a)-Class.2} is equivalent to the relation 
  \begin{equation*}
   \frac{\wt \mu_0(u)}{\wt \mu_1(u)}=\frac{P_1(u+\tfrac{1}{2})}{P_1(u)}. 
  \end{equation*}
 \item If $N=2n$ then the existence of $P_1(u)$ satisfying $P_1(u)=P_1(-u+n)$ and condition \eqref{T:DI(a)-Class.2} can be replaced by the equivalent requirement that there exists
 a monic polynomial $Q_1(u)$ such that $Q_1(u)=Q_1(-u+n)$, $n/2\in Z(Q_1(u))$ and 
 \begin{equation*}
  \frac{\wt \mu_1(\ka-u)}{\wt \mu_2(u)}=\frac{Q_1(u+1)}{Q_1(u)}\cdot \frac{\ka-u}{u}. 
 \end{equation*}
 \end{enumerate}
\end{rmk}
By Theorem 4.5 of \cite{GRW2}, the isomorphism class of any finite-dimensional irreducible module $V$ has a unique representative of the form $V(\mu(u))$, and by Theorem \ref{T:DI(a)-Class} we 
can assign to this representative a unique tuple $(\al,P_1(u),\ldots,P_n(u))$. This assignment is not injective: the second statement of Lemma \ref{L:unique} shows that $(\al,P_1(u),\ldots,P_n(u))$ determines $V(\mu(u))$ only up to twisting by automorphisms of the form $\nu_g$.  Note, however, that there is a unique series $g(u)\in 1+u^{-2}\C[[u^{-2}]]$ such that the central series $w(u)$ operates as the identity in $V(\mu(u))^{\nu_g}$. Indeed, by  Proposition 4.6 of \cite{GRW2} $w(u)$ acts in $V(\mu(u))^{\nu_g}$ as multiplication by the series
$h(u)h(u+\ka)\nu(u)$, where $h(u)=g(u-\ka/2)$ and $\nu(u)=\mu_n(u)\mu_n(-u)$. Therefore $w(u)$ will act as the identity in $V(\mu(u))^{\nu_g}$ precisely when $h(u)h(u+\ka)=\nu(u)^{-1}$. A simple argument shows that there exists a unique series $a(u)\in 1+u^{-1}\C[[u^{-1}]]$ such that $a(u)a(u+\ka)=\nu(u)^{-1}$, and this series satisfies $a(u)=a(\ka-u)$ since otherwise $b(u)=a(\ka-u)$ would be a distinct solution of $b(u)b(u+\ka)=\nu(u)^{-1}$ (as $\nu(u)$ is even). Hence, $w(u)$ will operate as multiplication by $1$ in $V(\mu(u))^{\nu_g}$ exactly when $g(u)=a(u+\ka/2)$. 

This discussion shows that we have an injective correspondence between isomorphism classes of finite-dimensional irreducible modules
of $X(\mfso_N,\mfso_{N-2}\oplus \mfso_2)^{tw}$ and finite sequences $(g(u); \al,P_1(u),\ldots,P_n(u))$, where $P_1(u),\ldots,P_n(u)$ are monic polynomials satisfying \eqref{P-sym}, $\alpha$ is a complex number which is not contained in $Z(P_{n}(u))$, and $g(u)\in 1+u^{-2}\C[[u^{-2}]]$. A straightforward modification of the $(\Longleftarrow)$ direction of the proof of Theorem \ref{T:DI(a)-Class} shows that this is a bijective correspondence. 

The next corollary provides the analogous classification result for finite-dimensional irreducible modules of the twisted Yangians of type BDI(a) with $q=2$ and $N\geq 5$.
\begin{crl}\label{C:DI(a)class}
The isomorphism classes of finite-dimensional irreducible $Y(\mfso_{N},\mfso_{N-2}\oplus \mfso_2)^{tw}$-modules are parameterized by
tuples of the form $ (\al,P_1(u),\ldots,P_n(u))$, 
where $P_1(u),\ldots,P_n(u)$ are monic polynomials in $u$ satisfying 
  \begin{equation*}
   P_1(u)=P_1(-u+\tfrac{N}{2}) \quad \text{ and }\quad P_i(u)=P_i(-u+n-i+2) \; \text{ for all }\; 2\leq i\leq n,
  \end{equation*}
and $\alpha$ is a complex scalar which is not contained in $Z(P_n(u))$. 
\end{crl}
\begin{proof}
The corollary follows from the discussion in the paragraphs preceding its statement combined with Proposition \ref{P:Y^tw-fd}. It can also be proven using the same arguments as employed to prove Corollary~6.3 of \cite{GRW2}. 
\end{proof}

Recall from Subsection \ref{subsec:Ya}  the definition of the fundamental representation $L(i:\alpha)$ of the Yangian $Y(\mfg_N)$. 
For each $a\in \C$, the one-dimensional irreducible representation $V(a)$ of $X(\mfso_N,\mfso_{N-2}\oplus \mfso_2)^{tw}$ from Corollary \ref{C:V(a)} can be regarded as an irreducible representation of $Y(\mfso_N,\mfso_{N-2}\oplus \mfso_2)^{tw}$ via restriction. The following result is the analogue of Corollary 6.4 of \cite{GRW2} for 
$Y(\mfso_N,\mfso_{N-2}\oplus \mfso_2)^{tw}$ and can be proven in the exact same way. A similar argument is given in the proof of Corollary \ref{C:BI(b)fun} below. 
\begin{crl}\label{C:DI(a)fun}
 Let $V$ be a finite-dimensional irreducible representation of $Y(\mfso_N,\mfso_{N-2}\oplus \mfso_2)^{tw}$. Then there is $m\geq 0$, $1\leq i_1,\ldots,i_m\leq n$, and 
 $a,\alpha_{i_1},\ldots,\alpha_{i_m}\in \C$ such that $V$ is isomorphic to the unique irreducible quotient of the $Y(\mfso_{N},\mfso_{N-2}\oplus \mfso_2)^{tw}$-module 
 \begin{equation*}
  Y(\mfso_{N},\mfso_{N-2}\oplus \mfso_2)^{tw}(\xi_1\otimes \cdots \otimes \xi_m \otimes \eta)\subset L(i_1:\al_{i_1})\otimes \cdots \otimes L(i_m: \al_{i_m})\otimes V(a),
 \end{equation*} 
 where $\eta\in V(a)$ is any nonzero vector and for each $1\leq k\leq m$, $\xi_k\in L(i_k:\al_{i_k})$ is a highest weight vector.  
\end{crl}
%

\subsection{Twisted Yangians for the symmetric pairs \texorpdfstring{$(\mfso_{2n+1},\mfso_{2n})$}{}}\label{subsec:q=1}

In this subsection we study the finite-dimensional irreducible representations of $X(\mfso_{2n+1},\mfso_{2n})^{tw}$ and $Y(\mfso_{2n+1},\mfso_{2n})^{tw}$ with $n\geq 2$ (which are of type BI(b) with $q=1$), the culmination of our effort being a classification of all such representations in Theorem \ref{T:BI(b)-Class} and Corollary \ref{C:BI(b)class}, respectively. Our first step towards proving this theorem is to study how a certain automorphism $\psi_\sigma^n$ interacts with highest weight modules of $X(\mfso_{2n+1},\mfso_{2n})^{tw}$: this will play a critical role in the rest of this section, one that is similar to the role played in the analogous classification for the twisted Yangian of the symmetric pair $(\mfgl_{2n},\mfso_{2n})$ by the automorphism (4.69) of \cite{Mobook}. 


\subsubsection{The automorphism $\psi_\sigma^n$}
For any $n\geq 1$, let $\mfS_N$ denote the symmetric group on the $N=2n+1$ symbols $\{-n,\ldots,-1,0,1,\ldots,n\}$ and let $\sigma$ be the transposition $(1,-1)$. Set $A_\sigma=\sum_{i=-n}^nE_{i,\sigma(i)}$, and note that $A_{\sigma}^t=A_\sigma=A_{\sigma}^{-1}$. Since we also have $A_\sigma\mcG A_\sigma^t=\mcG$, \eqref{al_A} implies that the assignment
$\psi_\sigma^n(S(u))=A_\sigma S(u)A_\sigma^t$ extends to an automorphism of $X(\mfso_{2n+1},\mfso_{2n})^{tw}$.
Our present goal is to determine the highest weight of the twisted module 
$V(\mu(u))^{\psi_{\sigma}^n}$. To this end, we return to the low rank setting. 
\begin{lemma}\label{BI:L-LR}
Suppose that the $X(\mfso_3,\mfso_2)^{tw}$-module $V(\mu(u))$ is finite-dimensional with associated Drinfeld tuple $(\al,P(u))$ as in Proposition \ref{P:so3class}. Then $V(\mu(u))^{\psi_{\sigma}^1}$ is isomorphic to $V(\mu^\sharp(u))$, where the components of $\mu^\sharp(u)$ are determined by the relations 
\begin{equation}\label{BI:sh}
 \wt\mu_0^\sharp(u)=\wt \mu_0(u)\cdot \frac{3-2u-2\al}{2\al-2u}\cdot \frac{2u-2\al+2}{2u+2\al-1},\qquad \wt \mu_1^\sharp(u)=\wt \mu_1(u)\cdot \frac{2u-2\al+1}{2u+2\al-2}\cdot \frac{2u-2\al+2}{2u+2\al-1}.
\end{equation}
In particular, $V(\mu(u))^{\psi_{\sigma}^1}$ has the Drinfeld tuple $(\tfrac{3}{2}-\al,P(u))$. 
\end{lemma}
\begin{proof}
 We appeal to the isomorphism $\varphi:X(\mfso_3,\mfso_2)^{tw}\to Y^+(2)$ given in \eqref{iso:so3so2}, which we recall is given by 
 \begin{equation}
  S(u)\mapsto \tfrac{1}{2}R_{12}^\circ(-1)S_1^\circ(2u-1)R_{12}^\circ(-4u+1)^{t_+}S_2^\circ(2u)K_1K_2, \label{vphi}
 \end{equation}
where $K=E_{11}-E_{-1,-1}$. Let $A=E_{1,-1}-E_{-1,1}$, and let $\beta_A$ be the automorphism of $Y^+(2)$ given by $S^\circ(u)\mapsto -AS^\circ(u)A^t$. More explicitly, 
$\beta_A$ is defined on generators by the assignment $s_{ij}^\circ(u)\mapsto (-1)^{\delta_{i,-1}+\delta_{j,-1}}s_{-i,-j}^\circ(u)$ for $i,j\in \{-1,1\}$. 
We claim that $\varphi^{-1}\circ \beta_A \circ \varphi=\psi_{\sigma}^1$. 
Applying $\beta_A$ to the right hand side of \eqref{vphi} and performing a few straightforward manipulations, we obtain 
\begin{equation*}
 \tfrac{1}{2}R_{12}^\circ(-1)(A_1A_2)S_1^\circ(2u-1)R_{12}^\circ(-4u+1)^{t_+}S_2^\circ(2u)K_1K_2(A_1A_2).
\end{equation*}
Hence, $(\varphi^{-1}\circ \beta_A \circ \varphi)(S(u))=\tilde A S(u) \tilde A$ where $\tilde A=\tfrac{1}{4}R_{12}^\circ(-1)A_1A_2 R_{12}^\circ(-1)$ (which 
is an element of $\End(V)\cong \End(\C^3)$). We have $\tilde A v_{i}=-v_{-i}$ for all $-1\leq i\leq 1$, so $\tilde A=-\sum_{i=-1}^1 E_{i,\sigma(i)}=-A_\sigma$. Since $\tilde AS(u)\tilde A=A_{\sigma}S(u)A_{\sigma}$, $\varphi^{-1}\circ \beta_A \circ \varphi=\psi_{\sigma}^1$. 

Now let $V(\mu(u))$ be as in the statement of the lemma. We have already seen in the proof of Proposition \ref{P:so3class} that there exists $\mu^\circ(u)\in 1+u^{-1}\C[[u^{-1}]]$ such that 
$\wt \mu_0(u)=-2u\left(\frac{4u-3}{4u-1}\right)\mu^\circ(2u)\mu^\circ(1-2u)$, $\wt \mu_1(u)=2u\mu^\circ(2u)\mu^\circ(2u-1)$, and that viewed as a $Y^+(2)$-module 
$V(\mu(u))$ is isomorphic to $V(\mu^\circ(u))$. Moreover, $V(\mu^\circ(u))$ has the Drinfeld tuple $(2\al-1,Q(u))$, where $Q(u)=2^{\deg P(u)}P(\tfrac{u+1}{2})$.
By Lemma 4.4.13 of \cite{Mobook}, the twisted module $V(\mu^\circ(u))^{\beta_A}$ is isomorphic to $V(\mu^\bullet(u))$, where $\mu^\bullet(u)$ is given by 
\begin{equation*}
 \mu^\bullet(u)=\mu^\circ(u)\cdot \frac{u-2\al+2}{u+2\al-1}.
\end{equation*}
As a $X(\mfso_3,\mfso_2)^{tw}$-module $V(\mu^\circ(u))^{\beta_A}$ is isomorphic to $V(\mu_0^\sharp(u),\mu_1^\sharp(u))$ where  
\begin{gather*}
 \wt \mu_0^\sharp(u)=-2u\left(\frac{4u-3}{4u-1}\right)\mu^\bullet(2u)\,\mu^\bullet(1-2u)=\wt \mu_0(u)\cdot \frac{3-2u-2\al}{2\al-2u}\cdot \frac{2u-2\al+2}{2u+2\al-1},\\
 \wt \mu_1^\sharp(u)=2u\cdot\mu^\bullet(2u)\,\mu^\bullet(2u-1)=\wt \mu_1(u)\cdot \frac{2u-2\al+1}{2u+2\al-2}\cdot \frac{2u-2\al+2}{2u+2\al-1}.
\end{gather*}
As $\varphi^{-1}\circ \beta_A \circ \varphi=\psi_{\sigma}^1$, we can conclude that $V(\mu(u))^{\psi_{\sigma}^1}$ is isomorphic to $V(\mu^\sharp(u))$ with $\mu^\sharp(u)$ as in 
\eqref{BI:sh}. Consequently, we have 
\begin{equation*}
\frac{\wt \mu_0^\sharp(u)}{\wt \mu_1^\sharp(u)}=\frac{\wt \mu_0(u)}{\wt \mu_1(u)}\cdot \frac{(\tfrac{3}{2}-\al)-u}{\al-u}\cdot \frac{\al+u-1}{(\tfrac{3}{2}-\al)+u-1}=\frac{P_1(u+\tfrac{1}{2})}{P_1(u)}\cdot \frac{(\tfrac{3}{2}-\al)-u}{(\tfrac{3}{2}-\al)+u-1}. \qedhere 
\end{equation*}
\end{proof}

We now consider the case where $n>1$: 

\begin{prop}\label{BI:P-psi}
 Suppose that $n>1$ and that the $X(\mfso_{2n+1},\mfso_{2n})^{tw}$-module $V(\mu(u))$ is finite-dimensional with Drinfeld tuple $(\al,P_1(u),\ldots,P_n(u))$. Then $V(\mu(u))^{\psi_\sigma^n}$ is isomorphic to 
 $V(\mu^\sharp(u))$, where $\mu_i^\sharp(u)=\mu_i(u)$ for all $2\leq i\leq n$, while $\mu_0(u)$ and $\mu_1(u)$ are determined by the relations
\begin{gather}
 \wt\mu_0^\sharp(u)=\wt \mu_0(u)\cdot \frac{N-2u-2\al}{2\al-2u}\cdot \frac{2u-2\al+2}{2u+2\al-N+2}, \label{mu^sh:0}\\ 
 \wt \mu_1^\sharp(u)=\wt \mu_1(u)\cdot  \frac{2u-2\al+1}{2u+2\al-N+1}\cdot \frac{2u-2\al+2}{2u+2\al-N+2}. \label{mu^sh:1}
\end{gather}
\end{prop}
\begin{proof}
Since $V(\mu(u))^{\psi_\sigma^n}$ is finite-dimensional and irreducible, it is isomorphic to $V(\mu^\sharp(u))$ for some $\mu^\sharp(u)$. Throughout the proof
we fix highest weight vectors $\xi\in V(\mu(u))$ and $\xi_\sigma\in V(\mu(u))^{\psi_\sigma^n}$. 

 Since $\psi_\sigma^n(s_{ij}(u))=s_{\sigma(i)\sigma(j)}(u)$ for all $i,j\in \mcI_N$, $V(\mu(u))_{(+,n-1)}$ and $(V(\mu(u))^{\psi_\sigma^n})_{(+,n-1)}$ are equal as subspaces 
 of $V(\mu(u))$ (see \eqref{V(+,m)}). This implies that the identity map provides a linear isomorphism between the $X(\mfso_3,\mfso_2)^{tw}$-modules
 $(V(\mu(u))_{(n-1)})^{\psi_\sigma^1}$ and $(V(\mu(u))^{\psi_\sigma^n})_{(n-1)}$. The first step of our proof is to show that this is also a module homomorphism. 
 
 \noindent \textit{Step 1: } The identity map $\mathrm{id}:(V(\mu(u))_{(n-1)})^{\psi_\sigma^1}\to (V(\mu(u))^{\psi_\sigma^n})_{(n-1)}$ is an isomorphism of $X(\mfso_3,\mfso_2)^{tw}$-modules. 
 
 To prove that this is the case it suffices to show that, for each $i,j\in \mcI_{3}$, the generating series $s_{ij}^{n-1}(u)$ of $X(\mfso_3,\mfso_2)^{tw}$ operates the same in both of these modules. Since $s_{ij}^{n-1}(u)$ acts in $V(\mu(u))_{(n-1)}$ as the operator $h_{n-1}(u)(s_{ij}^{\circ (n-1)}(u))$  (see \eqref{scirc} and \eqref{h_m:triv}), it operates in 
 $(V(\mu(u))_{(n-1)})^{\psi_\sigma^1}$ as the operator 
 \begin{equation*}
  h_{n-1}(u)\left(s_{\sigma(i)\sigma(j)}(u+\tfrac{n-1}{2})+\frac{\delta_{ij}}{2u}\sum_{a=2}^n s_{aa}(u+\tfrac{n-1}{2})\right).
 \end{equation*}
As $\sigma(a)=a$ for all $a\geq 2$, this is also equal to $h_{n-1}\psi_\sigma^n(s_{ij}(u))^{\circ (n-1)}$ which is precisely the operator by which $s_{ij}^{n-1}(u)$ acts in 
$(V(\mu(u))^{\psi_\sigma^n})_{(n-1)}$. 
 
 \noindent \textit{Step 2:} $\xi_\sigma$ is contained in $V(\mu(u))_{n-1}$. Moreover  $(V(\mu(u))_{n-1})^{\psi_\sigma^1}\cong (V(\mu(u))^{\psi_\sigma^n})_{n-1}$. 
 
 Let $\xi_{\sigma}^1$ be a highest weight vector of the irreducible $X(\mfso_3,\mfso_2)^{tw}$-module $(V(\mu(u))_{n-1})^{\psi_{\sigma}^1}$. Since this
 is a submodule of $(V(\mu(u))_{(n-1)})^{\psi_\sigma^1}$, Step 1 shows that $\xi_{\sigma}^1$ is also contained in $(V(\mu(u))^{\psi_\sigma^n})_{(n-1)}$ and generates a highest weight submodule. By Corollary \ref{C:nosub}, this submodule must be equal to $(V(\mu(u))^{\psi_\sigma^n})_{n-1}$ and thus $\xi_{\sigma}^1$ must be proportional to $\xi_\sigma$. 
 This implies that $\xi_\sigma$, being a scalar multiple of $\xi_\sigma^1$, is contained in $V(\mu(u))_{n-1}$. 
 
 Next, Let $W$ be the image of the module $(V(\mu(u))_{n-1})^{\psi_{\sigma}^1}$ under the isomorphism $\mathrm{id}: (V(\mu(u))_{(n-1)})^{\psi_\sigma^1}\to (V(\mu(u))^{\psi_\sigma^n})_{(n-1)}$. As $W$ is the irreducible submodule of $(V(\mu(u))^{\psi_\sigma^n})_{(n-1)}$ generated by $\xi_\sigma$, it is equal to 
 $(V(\mu(u))^{\psi_\sigma^n})_{n-1}$.
 
 \noindent \textit{Step 3: } $\mu^\sharp_i(u)=\mu_i(u)$ for all $2\leq i\leq n$. 
 
 Since $V(\mu(u))_{n-1}$ is a $X(\mfso_3,\mfso_2)^{tw}$ highest weight module, it is generated by monomials of the form 
 $(s_{i_1j_1}^{n-1 (r_1)}\cdots s_{i_c j_c}^{n-1 (r_c)}) \cdot \xi$  where $-1\leq j_a<i_a \leq 1$ and $r_k\geq 1$ for all $1\leq a\leq c$, $c$ being a non-negative integer. By definition of the action of $X(\mfso_3,\mfso_2)^{tw}$, this implies  $V(\mu(u))_{n-1}$ is also spanned by monomials of the form 
 \begin{equation}
  s_{i_1j_1}^{(r_1)}\cdots s_{i_c j_c}^{(r_c)}\xi \label{V_n-1:mon}
 \end{equation}
with the same restrictions on the indices. In particular, by Step 2 the highest weight vector $\xi_\sigma$ must be a linear combination of such monomials.  Let $J_{n-1}$ be the left ideal of $X(\mfso_{2n+1},\mfso_{2n})^{tw}$ generated by all elements $s_{ij}^{(r)}$ with 
$r\geq 1$ and $i<j$ for $2\leq j\leq n$. In particular, $V(\mu(u))_{(n-1)}$ is, as a vector space, equal to the subspace of $V(\mu(u))$ annihilated by $J_{n-1}$. For every
$2\leq k\leq n$ and pair $(i,j)$ with $-1\leq j<i\leq 1$ the defining reflection equation \eqref{TX-RE} (see also \cite[(3.11)]{GRW2}) implies that 
\begin{equation*}
 [s_{kk}(u),s_{ij}(v)]=0 \mod J_{n-1},
\end{equation*}
and hence the action of $s_{kk}(u)$ on a monomial of the form \eqref{V_n-1:mon} is given by
\begin{equation*}
 s_{kk}(u)(s_{i_1j_1}^{(r_1)}\cdots s_{i_c j_c}^{(r_c)}\xi)=s_{i_1j_1}^{(r_1)}\cdots s_{i_c j_c}^{(r_c)}(s_{kk}(u)\xi)=\mu_k(u)(s_{i_1j_1}^{(r_1)}\cdots s_{i_c j_c}^{(r_c)}\xi).
\end{equation*}
Since $s_{\sigma(k)\sigma(k)}(u)=s_{kk}(u)$ for all $2\leq k\leq n$, the above observation yields that $\psi_\sigma^n(s_{kk}(u))\xi_\sigma=\mu_k(u)\xi_\sigma$ for all these values of $k$. Hence, $\mu_k(u)=\mu_k^\sharp(u)$ for all  $2\leq k\leq n$.

\noindent \textit{Step 4: } The formulas \eqref{mu^sh:0} and \eqref{mu^sh:1} hold.

 To compute $(\mu_0^\sharp(u),\mu_1^\sharp(u))$, we use Step 3 in conjunction with Lemma \ref{BI:L-LR}. Since $V(\mu(u))_{n-1}$ has Drinfeld tuple $(\al-\tfrac{n-1}{2},P_1(u+\tfrac{n-1}{2}))$ (by Corollary \ref{C:poly-low}) and $(V(\mu(u))_{n-1})^{\psi_\sigma^1}\cong (V(\mu(u))^{\psi_\sigma^n})_{n-1}\cong V(\mu^\sharp(u))_{n-1}$, Lemma \ref{BI:L-LR} together with 
 \eqref{mu-circ} gives 
\begin{gather*}
 \wt\mu_0^\sharp(u+\tfrac{n-1}{2})=\wt \mu_0(u+\tfrac{n-1}{2})\cdot \frac{3-2u-2\al+n-1}{2\al-n+1-2u}\cdot \frac{2u-2\al+n+1}{2u+2\al-n},\\ \wt \mu_1^\sharp(u+\tfrac{n-1}{2})=\wt \mu_1(u+\tfrac{n-1}{2})\cdot \frac{2u-2\al+n}{2u+2\al-n-1}\cdot \frac{2u-2\al+n+1}{2u+2\al-n}.
\end{gather*}
Substituting $u\mapsto u-\tfrac{n-1}{2}$ we obtain the formulas \eqref{mu^sh:0} and \eqref{mu^sh:1}. \qedhere
\end{proof}
\begin{prop}\label{P:q=1-nec}
 Suppose that $V(\mu(u))$ is finite-dimensional with Drinfeld tuple $(\al,P_1(u),\ldots,P_n(u))$. It follows that $\al\in \frac{1}{2}\Z+\tfrac{N}{4}$ and $S(\al,\tfrac{N}{2}-\al)\cup S(\al+\tfrac{1}{2},\tfrac{N}{2}-\al+\tfrac{1}{2})\subset Z(P_2(u))$.
\end{prop}

\begin{rmk}
 By definition, the strings $S(\al,\tfrac{N}{2}-\al)$ and $S(\al+\tfrac{1}{2},\tfrac{N}{2}-\al+\tfrac{1}{2})$ are empty unless $\al>\tfrac{N}{4}$. Therefore, 
 the condition $S(\al,\tfrac{N}{2}-\al)\cup S(\al+\tfrac{1}{2},\tfrac{N}{2}-\al+\tfrac{1}{2})\subset Z(P_2(u))$ is vacuous whenever $\al\leq \tfrac{N}{4}$. 
\end{rmk}

\begin{proof}[Proof of Proposition \ref{P:q=1-nec}]
 It was shown in Proposition \ref{P:int} that $\al\in \frac{1}{2}\Z+\tfrac{N}{4}$. However, we will not assume this in our proof. As a consequence of Proposition \ref{BI:P-psi}
 we have $V(\mu(u))^{\psi_\sigma^n}\cong V(\mu^\sharp(u))$ where $\mu_{k}^\sharp(u)=\mu_k(u)$ for all $2\leq k\leq n$ and the pair $(\mu_0^\sharp(u),\mu_1^\sharp(u))$ is determined from the relations \eqref{mu^sh:0} and \eqref{mu^sh:1}. Since $\mu(u)$ is associated to $(\al,P_1(u),\ldots,P_n(u))$, the components of $\wt \mu^\sharp(u)$ satisfy the relations 
 \begin{gather}\label{sigmaPoly1}
 \frac{\wt \mu_0^\sharp(u)}{\wt \mu_1^\sharp(u)}=\frac{P_1(u+\tfrac{1}{2})}{P_1(u)}\cdot \frac{(\tfrac{N}{2}-\al)-u}{u+(\tfrac{N}{2}-\al)-n}, \\
  \frac{\wt \mu_1^\sharp(u)}{\wt \mu_2^\sharp(u)}=\frac{P_2(u+1)}{P_2(u)}\cdot  \frac{2u-2\al+1}{2u+2\al-N+1}\cdot \frac{2u-2\al+2}{2u+2\al-N+2}, \label{sigmaPoly1a} \\
 \frac{\wt \mu_{i-1}^\sharp(u)}{\wt \mu_i^\sharp(u)}=\frac{P_i(u+1)}{P_i(u)}\quad  \text{ for all }\quad  3\leq i\leq n. \label{sigmaPoly2}
 \end{gather}
 On the other hand, since $V(\mu^\sharp(u))$ is finite-dimensional Proposition \ref{P:necessary} implies that $\mu^\sharp(u)$ can be associated to a Drinfeld tuple 
 $(\al^\sharp,Q_1(u),\ldots,Q_n(u))$. Consequently, from relation \eqref{sigmaPoly1a} we obtain the equality 
 \begin{equation*}
 \frac{Q_2(u+1)}{Q_2(u)}\cdot \frac{(n-\al)-u}{u+(n-\al)-n+1}=\frac{P_2(u+1)}{P_2(u)}\cdot \frac{(\al-\tfrac{1}{2})-u}{u+(\al-\tfrac{1}{2})-n+1}.
 \end{equation*}
 Applying Lemma \ref{L:poly2} to both sides (with $m=1$ and $l=n$) we find that there exists monic polynomials $Q_2^\bullet(u)$ and $P_2^\bullet(u)$, together with non-negative integers 
 $\ell_P$ and $\ell_Q$ such that $P_2^\bullet(u)=P_2^\bullet(-u+n)$, $Q_2^\bullet(u)=Q_2^\bullet(-u+n)$, and 
  \begin{equation*}
 \frac{Q_2^\bullet(u+1)}{Q_2^\bullet(u)}\cdot \frac{(n-\al-\ell_Q)-u}{u+(n-\al-\ell_Q)-n+1}=\frac{P_2^\bullet(u+1)}{P_2^\bullet(u)}\cdot \frac{(\al-\tfrac{1}{2}-\ell_P)-u}{u+(\al-\tfrac{1}{2}-\ell_P)-n+1}, 
 \end{equation*}
 with $Q_2^\bullet(n-\al-\ell_Q)\neq 0$ and $P_2^\bullet(\al-\tfrac{1}{2}-\ell_P)\neq 0$. By Lemma \ref{L:poly1}, we must have $Q_2^\bullet(u)=P_2^\bullet(u)$ and 
 $n-\al-\ell_Q=\al-\tfrac{1}{2}-\ell_P$. The latter relation implies that $2\al-\tfrac{N}{2}=\ell_P-\ell_Q\in \Z$, and thus that $\al\in \frac{1}{2}\Z+\tfrac{N}{4}$.
 
 If in addition $\al>\tfrac{N}{4}$, then $\ell_P\geq \ell_P-\ell_Q=2\al-\tfrac{N}{2}>0$. Since $(\ell_P,P_2^\bullet(u))$ is the pair $(\ell_{\al-1/2}^1,P_{\al-1/2}^1(u))$ from Lemma \ref{L:poly2} (where $P(u)=P_2(u)$), $P_2^\bullet(u)$ is equal to $P_2(u)$ divided by the polynomial $Q(u)$ from \eqref{poly1'} with $m=1$ and $\al$ replaced by $\al-1/2$. Therefore
 $P_2(u)$ is divisible by the polynomial 
 \begin{equation}\label{P-gamma}
  P_\al(u)=\prod_{k=0}^{2\al-\tfrac{N}{2}-1}(u-\al+1/2+k)(u-\tfrac{N}{2}+\al-k)=\prod_{k=0}^{2\al-\tfrac{N}{2}-1}(u-\al+1/2+k)(u-\al+1+k).
 \end{equation}
 The proof of the proposition is completed by observing that the roots of  $P_\al(u)$ are precisely the elements of $S(\al,\tfrac{N}{2}-\al)\cup S(\al+\tfrac{1}{2},\tfrac{N}{2}-\al+\tfrac{1}{2})$.
\end{proof}
 \begin{rmk}
The statement of Proposition \ref{P:q=1-nec} is much stronger than that of Proposition \ref{P:int} (in the case $(\mfg_N,\mfg_N^\rho)=(\mfso_{2n+1},\mfso_{2n})$). The latter 
tells us that $\al\in \frac{1}{2}\Z+\tfrac{N}{4}$ and that 
$
\al-\tfrac{N}{4}\leq \tfrac{1}{4}(\deg P_1(u)+\deg P_2(u))
$
but says nothing about the roots of $P_2(u)$. In fact, since the strings $S(\al,\tfrac{N}{2}-\al)$ and  $S(\al+\tfrac{1}{2},\tfrac{N}{2}-\al+\tfrac{1}{2})$
are disjoint and both have length $2\al-\tfrac{N}{2}$, Proposition \ref{P:q=1-nec} implies that $\al-\tfrac{N}{4}\leq \tfrac{1}{4}\deg P_2(u)$. 
\end{rmk}

Provided $\al>\tfrac{N}{4}$, the polynomial $P_\al(u)$ from \eqref{P-gamma} satisfies the relation 
\begin{equation}
 \frac{P_\al(u)}{P_\al(u+1)}= \frac{2u-2\al+1}{2u+2\al-N+1}\cdot \frac{2u-2\al+2}{2u+2\al-N+2}. \label{P-gamma.2}
\end{equation}
If instead $\al\leq \tfrac{N}{4}$, let $P_\al^-(u)$ be the polynomial 
\begin{equation*}
 P_\al^-(u)=\prod_{k=0}^{\frac{N}{2}-2\al-1}(u-n+\al+k)(u-\al-k)=\prod_{k=0}^{\frac{N}{2}-2\al-1}(u-\al-\tfrac{1}{2}-k)(u-\al-k),
\end{equation*}
where the equality $P_\al^-(u)=1$ is understood to hold if $\al=\tfrac{N}{4}$. Then $P_\al^-(u)$ satisfies the relation 
\begin{equation*}
 \frac{P_\al^-(u+1)}{P_\al^-(u)}= \frac{2u-2\al+1}{2u+2\al-N+1}\cdot \frac{2u-2\al+2}{2u+2\al-N+2}.
\end{equation*}
These observations together with the relations \eqref{sigmaPoly1} and \eqref{sigmaPoly2} imply the following corollary. 
\begin{crl}\label{C:Vtw-poly}
 Suppose that $V(\mu(u))$ is finite-dimensional with Drinfeld tuple $(\al,P_1(u),\ldots,P_n(u))$. Then the Drinfeld tuple of the finite-dimensional irreducible 
 module $V(\mu(u))^{\psi_\sigma^n}$ is $(\tfrac{N}{2}-\al,P_1^\sharp(u),\ldots,P_n^\sharp(u))$, where $P_i^\sharp(u)=P_i(u)$ for all $i\neq 2$ and 
\begin{equation}\label{P2-sharp}
P_2^\sharp(u)=
 \begin{cases}
  P_2(u)P_\al^-(u) \; & \text{ if }\; \al\leq \tfrac{N}{4},\\
  P_2(u)/P_\al(u) \; & \text{ if }\; \al> \tfrac{N}{4}.
 \end{cases}
\end{equation}
\end{crl}

Observe that these formulas together with those of Proposition \ref{BI:P-psi} imply that, under the assumption that $V(\mu(u))$ finite-dimensional, 
the isomorphism $V(\mu(u))\cong V(\mu(u))^{\psi_{\sigma}^n}$ will hold if and only if the scalar $\al$ corresponding to $V(\mu(u))$ is equal to $\tfrac{N}{4}$. 

%


\subsubsection{Classification}

With Proposition \ref{P:q=1-nec} at our disposal we can now classify the finite-dimensional irreducible representations of the extended twisted Yangian 
$X(\mfso_{2n+1},\mfso_{2n})^{tw}$. 
\begin{thrm}\label{T:BI(b)-Class}
Suppose that $\mu(u)$ satisfies the conditions of Proposition \ref{P:nontriv} so that the Verma module $M(\mu(u))$ is non-trivial. Then
the irreducible $X(\mfso_{2n+1},\mfso_{2n})^{tw}$-module $V(\mu(u))$ is finite-dimensional if and only if there are  monic polynomials $P_1(u),\ldots,P_n(u)$ in $u$, with
\begin{equation}
 P_i(u)=P_i(-u+n-i+2) \; \text{ for all }\; i\geq 2 \; \text{ and }\;  P_1(u)=P_1(-u+\tfrac{N}{2}), \label{T:BI(b).0}
\end{equation}
together with a scalar $\al\in \tfrac{1}{2}\Z+\tfrac{N}{4}$ such that the following relations are satisfied:
\begin{gather}
\al \notin Z(P_1(u)),\quad S(\al,\tfrac{N}{2}-\al)\cup S(\al+\tfrac{1}{2},\tfrac{N}{2}-\al+\tfrac{1}{2})\subset Z(P_2(u)),  \label{T:BI(b).1}\\[.5em]
\frac{\wt \mu_{i-1}(u)}{\wt \mu_i(u)}=\frac{P_i(u+1-\tfrac{\delta_{i1}}{2})}{P_i(u)}\left(\frac{\al-u}{\al+u-n}\right)^{\del_{i,1}}\text{ for all }\;1\leq i\leq n. \label{T:BI(b).2}
\end{gather}
Additionally, when $V(\mu(u))$ is finite-dimensional the corresponding tuple $(\al,P_1(u),\ldots, P_n(u))$ is unique. 
\end{thrm}
\begin{proof}
 $(\Longrightarrow)$ If $V(\mu(u))$ is finite-dimensional, then it follows from Propositions \ref{P:necessary} and \ref{P:q=1-nec} that $\mu(u)$ can be associated to 
 a (Drinfeld) tuple $(\al,P_1(u),\ldots,P_n(u))$ satisfying the relations of the theorem. The uniqueness statement has also been proven in Lemma \ref{L:unique}. 
 
 $(\Longleftarrow)$ Conversely, suppose that $\mu(u)$ can be associated to a tuple $(\al,P_1(u),\ldots,P_n(u))$ as in Definition \ref{D:assoc} which also satisfies $\al\in \tfrac{1}{2}\Z+\tfrac{N}{4}$ and the relation \eqref{T:BI(b).1}. We will show that $V(\mu(u))$ is finite-dimensional, splitting our proof into two cases. 
 
 \noindent \textit{Case 1:} $\al\leq \tfrac{N}{4}$. 
 
 Let $P_i^\circ(u)=P_i(u)$ for all $i\neq 1$ and set $P_1^\circ(u)$ to be the polynomial obtained by multiplying $P_1(u)$ with $\prod_{k=0}^{N/2-2\al-1}(u-\tfrac{N}{4}+\frac{k}{2})(u-\tfrac{N}{4}-\frac{k}{2})$.
 Then $P_1^\circ(u)$ satisfies $P_1^\circ(u)=P_1^\circ(-u+\tfrac{N}{2})$ and the relation 
 \begin{equation}
  \frac{P_1^\circ(u+\tfrac{1}{2})}{P_1^\circ(u)}\cdot\frac{\tfrac{N}{4}-u}{\tfrac{N}{4}+u-n}=\frac{P_1(u+\tfrac{1}{2})}{P_1(u)}\cdot \frac{\al-u}{\al+u-n}. \label{<N/4.1}
 \end{equation}
Since $P_1^\circ(u)=P_1^\circ(-u+\tfrac{N}{2})$, there exists a monic polynomial $Q_1(u)$ such that $P_1(u)=(-1)^{\deg Q_1(u)}Q_1(u-\ka/2)Q_1(-u+\tfrac{N}{2}-\ka/2)$, and similarly, 
since $P_i^\circ(u)=P_i^\circ(-u+n-i+2)$ for all $i\geq 2$, there exists monic polynomials 
$Q_2(u),\ldots,Q_n(u)$ satisfying 
\begin{equation*}
 P_i(u)=(-1)^{\deg Q_i(u)}Q_i(u-\ka/2)Q_i(-u+n-i+2-\ka/2) \; \text{ for all }\; 2\leq i\leq n. 
\end{equation*}
Let $L(\lambda(u))$ be a finite-dimensional irreducible $X(\mfso_N)$-module with Drinfeld polynomials $Q_1(u),\ldots,Q_n(u)$, and let $\xi\in L(\lambda(u))$ be a highest weight vector. 
Then Lemma \ref{L:QxP=QP} applied with $V(\mu(u))=V(\mcG)$, together with Remark \ref{R:QxP=QP} and the relation \eqref{<N/4.1} imply that the finite-dimensional highest weight module 
$X(\mfso_{2n+1},\mfso_{2n})^{tw}\xi \subset L(\lambda(u))$ has highest weight $\nu(u)$ which is also associated to $(\al,P_1(u),\ldots,P_n(u))$. By Lemma \ref{L:unique}, there exists $g(u)\in 1+u^{-2}\C[[u^{-2}]]$ such that $V(\nu(u))\cong V(\mu(u))^{\nu_g}$. Since $V(\nu(u))$ is finite-dimensional, this proves that $V(\mu(u))$ also is. 

\noindent \textit{Case 2:} $\al>\tfrac{N}{4}$.

Let $\mu^\sharp(u)=(\mu^\sharp_i(u))_{i\in \mcI_N^+}$ be the tuple determined by $\mu_i^\sharp(u)=\mu_i(u)$ for all $2\leq i\leq n$ and with 
$\wt \mu_0^\sharp(u)$, $\wt \mu_1^\sharp(u)$ given by the formulas \eqref{mu^sh:0}, \eqref{mu^sh:1}, respectively.  Since $S(\al,\tfrac{N}{2}-\al)\cup S(\al+\tfrac{1}{2},\tfrac{N}{2}-\al+\tfrac{1}{2})\subset Z(P_2(u))$, the polynomial $P_\al(u)$ from \eqref{P-gamma} divides $P_2(u)$. By \eqref{sigmaPoly1}, \eqref{sigmaPoly2} and \eqref{P-gamma.2}, $\mu^\sharp(u)$ is associated to $(\tfrac{N}{2}-\al,P_1^\sharp(u),\ldots,P_n^\sharp(u))$, where $P_i^\sharp(u)=P_i(u)$ for all $i\neq 2$ and 
$P_2^\sharp(u)=P_2(u)/P_\al(u)$. 

Since $\tfrac{N}{2}-\al<\tfrac{N}{4}$, the argument of Case 1 implies that $V(\mu^\sharp(u))$ is finite-dimensional. By Proposition 4.7, 
$V(\mu^\sharp(u))^{\psi_{\sigma}^n}$ is isomorphic to $V(\mu(u))$, and thus $V(\mu(u))$ is also finite-dimensional. 
\end{proof}

A similar argument to that given after Theorem \ref{T:DI(a)-Class} shows that, as a consequence of Theorem \ref{T:BI(b)-Class}, the isomorphism classes of 
finite-dimensional irreducible $X(\mfso_{2n+1},\mfso_{2n})^{tw}$-modules are in bijective correspondence with tuples 
$(g(u); \al,P_1(u),\ldots,P_n(u))$, where 
\begin{enumerate}
 \item $P_1(u),\ldots,P_n(u)$ are monic polynomials satisfying \eqref{T:BI(b).0}, \label{Xclassi}
 \item $\al\in \tfrac{1}{2}\Z+\tfrac{N}{4}$ is not a root of $P_1(u)$, \label{Xclassii}
 \item $S(\al,\tfrac{N}{2}-\al)\cup S(\al+\tfrac{1}{2},\tfrac{N}{2}-\al+\tfrac{1}{2})\subset Z(P_2(u))$, \label{Xclassiii}
 \item $g(u)$ is an element of $1+u^{-2}\C[[u^{-2}]]$. \label{Xclassiv}
\end{enumerate}
More explicitly, the correspondence assigns to $V(\mu(u))$ the finite sequence $(g(u); \al,P_1(u),\ldots,P_n(u))$ where $(\al,P_1(u),\ldots,P_n(u))$ is the Drinfeld tuple 
associated to $\mu(u)$ and $g(u)$ is the unique series in $1+u^{-2}\C[[u^{-2}]]$ such that the central series $w(u)$ acts as the identity operator in $V(\mu(u))^{\nu_g}$. 

The next corollary provides a classification of finite-dimensional irreducible  $Y(\mfso_{2n+1},\mfso_{2n})^{tw}$-modules, and it follows from the above classification of finite-dimensional irreducible $X(\mfso_{2n+1},\mfso_{2n})^{tw}$-modules together with Proposition \ref{P:Y^tw-fd}.
\begin{crl}\label{C:BI(b)class} The isomorphism classes of finite-dimensional irreducible $Y(\mfso_{2n+1},\mfso_{2n})^{tw}$-modules are parameterized by
tuples $(\al,P_1(u),\ldots,P_n(u))$ satisfying conditions \eqref{Xclassi}-\eqref{Xclassiii}.
\end{crl}

We now turn towards obtaining a result analogous to Corollary \ref{C:DI(a)fun} for  $Y(\mfso_{2n+1},\mfso_{2n})^{tw}$. 

Since the automorphism $\psi_\sigma^n$ of $X(\mfso_{2n+1},\mfso_{2n})^{tw}$ fixes $w(u)$, it induces an automorphism of the Yangian $Y(\mfso_{2n+1},\mfso_{2n})^{tw}$ which is given by the assignment $\Sigma(u)\mapsto A_\sigma \Sigma(u) A_\sigma^t$ - see \eqref{al_A}. We also denote this automorphism by $\psi_\sigma^n$. It is not difficult to see that the restriction of the irreducible $X(\mfso_{2n+1},\mfso_{2n})^{tw}$-module $V(\mu(u))^{\psi_\sigma^n}$ to the subalgebra $Y(\mfso_{2n+1},\mfso_{2n})^{tw}$ is isomorphic to $V^{\psi_\sigma^n}$, where 
$V$ is $V(\mu(u))$ viewed as a $Y(\mfso_{2n+1},\mfso_{2n})^{tw}$-module. 

Given $m\geq 0$, $1\leq i_1,\ldots,i_m\leq n$, and $\al_{i_1},\ldots,\al_{i_m}\in \C$, we can consider the tensor product of $Y(\mfso_{2n+1})$ fundamental representations 
\begin{equation}
 L(i_1:\al_{i_1})\otimes \cdots \otimes L(i_m:\al_{i_m}). \label{Y-fun}
\end{equation}
If $m=0$ then we we will identify this tensor product with the trivial representation of $Y(\mfso_{2n+1})$. For each $1\leq k\leq m$, let 
$\xi_i$ be a highest weight vector of $L(i_k:\al_{i_k})$ and set 
\begin{equation*}
 \boldsymbol{\xi}=\xi_1\otimes \cdots \otimes \xi_{m}. 
\end{equation*}
If $m=0$ the vector $\boldsymbol{\xi}$ is understood to be equal to $1\in \C$, where $\C$ is viewed as the space of the trivial representation of $Y(\mfso_{2n+1})$. We can then consider the 
$Y(\mfso_{2n+1},\mfso_{2n})^{tw}$ highest weight module 
\begin{equation}
 Y(\mfso_{2n+1},\mfso_{2n})^{tw}\boldsymbol{\xi}\subset L(i_1:\al_{i_1})\otimes \cdots \otimes L(i_m:\al_{i_m}). \label{BI(b)fun.1}
\end{equation}

\begin{crl}\label{C:BI(b)fun}	
 Let $V$ be a finite-dimensional irreducible representation of $Y(\mfso_{2n+1},\mfso_{2n})^{tw}$ with Drinfeld tuple $(\al,P_1(u),\ldots,P_n(u))$. Then  
 \begin{enumerate}[label=(\alph*)]
  \item  $V$ is isomorphic to the unique irreducible quotient of a module of the form \eqref{BI(b)fun.1} if and only if $\al\leq \tfrac{N}{4}$, \label{(a)}
  \item $V^{\psi_\sigma^n}$ is isomorphic to the unique irreducible quotient of a module of the form \eqref{BI(b)fun.1} if and only if $\al \geq \tfrac{N}{4}$. \label{(b)}
  \end{enumerate}
\end{crl}
\begin{proof}
 If $V$ is isomorphic to the irreducible quotient of the module \eqref{BI(b)fun.1}, Lemma \ref{L:QxP=QP} and Remark \ref{R:QxP=QP} with $V(\mu(u))=V(\mcG)$ imply that  $\al=\tfrac{N}{4}-\tfrac{\ell_\al}{2}$ where $\ell_\al$ is non-negative integer. This proves the 
$(\Longrightarrow)$ direction of \ref{(a)}. 

Suppose now that $V^{\psi_\sigma^n}$ is isomorphic to the irreducible quotient of the module \eqref{BI(b)fun.1}. By Corollary \ref{C:Vtw-poly}, $V^{\psi_\sigma^n}$ is associated to the Drinfeld tuple $(\tfrac{N}{2}-\al,P_1^\sharp(u),\ldots,P_n^\sharp(u))$ where $P_i^\sharp(u)=P_i(u)$ for all $i\neq 2$ and $P_2^\sharp(u)$ is given by \eqref{P2-sharp}. Hence, the same argument as given in the previous paragraph shows that $\tfrac{N}{2}-\al\leq \tfrac{N}{4}$, thus proving the $(\Longrightarrow)$ direction of \ref{(b)}. 

Now let us turn to proving the $(\Longleftarrow)$ direction of \ref{(a)} and \ref{(b)}, beginning with the former. 
Assume that $\al\leq \tfrac{N}{4}$. Then the argument used to treat Case 1 of the $(\Longleftarrow)$ direction of Theorem \ref{T:BI(b)-Class}'s proof shows that there is a tuple of monic polynomials $\mathbf{Q}=(Q_1(u),\ldots,Q_n(u))$ such that $V$ is isomorphic to the irreducible quotient of $Y(\mfso_{2n+1},\mfso_{2n})^{tw}\xi_\mathbf{Q}\subset L(\mathbf{Q})$, where $L(\mathbf{Q})$ is the unique up to isomorphism $Y(\mfso_{2n+1})$-module with the Drinfeld tuple $\mathbf{Q}$ and $\xi_\mathbf{Q}\in L(\mathbf{Q})$ is a highest weight vector: see Theorem \ref{T:X-class} and the two paragraphs immediately following it. The conclusion that $V$ is isomorphic to the unique irreducible quotient of a module of the form \eqref{BI(b)fun.1} now follows from the well-known result that $L(\mathbf{Q})$ must be isomorphic to a subquotient of a $Y(\mfso_{2n+1})$-module of the form \eqref{Y-fun}: see \cite[Lemma 5.17]{AMR} and \cite[Corollary 12.1.13]{CP}.

If instead $\al \geq \tfrac{N}{4}$, then the argument of the previous paragraph can be used after replacing $V$ by $V^{\psi_\sigma^n}$ and using instead Case 2 of the  $(\Longleftarrow)$ direction of the proof of Theorem \ref{T:BI(b)-Class}. This yields that $V^{\psi_\sigma^n}$ is isomorphic to the irreducible quotient of a module of the form \eqref{BI(b)fun.1}. \end{proof}


\end{document}